\documentclass[11pt]{amsart}
\usepackage{latexsym}
\usepackage{hyperref}
\usepackage{amsmath,amsfonts,amssymb}
\usepackage{amsxtra}
\numberwithin{equation}{section}
\newtheorem{theo}{Theorem}[section]
\newtheorem{lemma}[theo]{Lemma}
\newtheorem{corol}[theo]{Corollary}
\newtheorem{prop}[theo]{Proposition}
\newtheorem{remark}[theo]{Remark}
\newtheorem{remarks}[theo]{Remarks}
\newtheorem{example}[theo]{Example}

\newcommand{\Real}{\mathbb R}
\newcommand{\Nat}{\mathbb N}
\newcommand{\norm}{\|\cdot\|}
\newcommand{\asnorm}{\|\cdot|_{p,q}}
\newcommand{\inter}{\operatorname{int}}
\newcommand{\clos}{\operatorname{cl}} %
\newcommand{\extreme}{\operatorname{ext}}
\newcommand{\conv}{\operatorname{co}}
\newcommand{\diam}{\operatorname{diam}}
\newcommand{\clco}{\overline{\operatorname{co}}}
\newcommand{\id}{\operatorname{id}}
\newcommand{\hull}{\operatorname{sp}} %
\newcommand{\graph}{\operatorname{Graph}} %
\newcommand{\slip}{\operatorname{SLip}} %

\newcommand{\bequ}{\begin{equation}} %
\newcommand{\eequ}{\end{equation}} %
\newcommand{\bequs}{\begin{equation*}} %
\newcommand{\eequs}{\end{equation*}} %
\newcommand{\vphi}{\varphi}
\newcommand{\epsi}{\epsilon} %
\newcommand{\rondu}{\mathcal U}%
\newcommand{\rondv}{\mathcal V}%
\newcommand{\rondw}{\mathcal W}%
\begin{document} %

\title[Asymmetric normed spaces]{ Functional analysis in asymmetric normed spaces}

\author{  S. Cobza\c s}
\thanks{\bf This research was supported by Grant CNCSIS 2261, ID 543.}

\address{\it Babe\c s-Bolyai University, Faculty of Mathematics
and Computer Science, 400 084 Cluj-Napoca, Romania}\;\;
\email{scobzas@math.ubbcluj.ro} %

\begin{abstract}

 The aim of this paper is to present a survey of some recent results obtained in the study of spaces with asymmetric norm. The presentation follows the ideas from the theory of normed  spaces (topology, continuous linear operators, continuous linear functionals, duality, geometry of asymmetric normed spaces, compact operators) emphasizing similarities as well as differences with respect to the classical theory. The main difference comes form the fact that the dual of an asymmetric normed space $X$ is not a linear space, but merely a convex cone in the space of all linear functionals on $X.$   Due to this fact, a careful treatment of the duality problems (e.g. reflexivity) and of other results as, for instance, the extension of fundamental principles of functional analysis -the open mapping theorem and the closed graph theorem - to this setting,  is needed.

 AMS 2000 MSC: Primary: 46-02;\;
 Secondary: 41A65, 46B20, 46B42, 46B50, 47A05, 46B07, 54E05, 54E15, 54E25

\end{abstract}
\maketitle
\begin{center}
{\bf Contents}
\end{center}

 {\bf 1. Introduction}: {\small 1.1 Quasi-metric spaces and asymmetric normed spaces; 1.2
 The topology of a quasi-semimetric space; 1.3 Quasi-uniform spaces.}

  {\bf 2. Completeness and compactness in quasi-metric and in quasi-uniform spaces}: {\small 2.1 Various notions of completeness for quasi-metric spaces; 2.2 Compactness, total bondedness and precompactness; 2.3  Completeness in quasi-uniform spaces; 2.4 Baire category.}

  {\bf 3. Continuous linear operators between asymmetric normed spaces}: {\small 3.1 The asymmetric norm of a continuous linear operator; 3.2 The normed cone  of continuous linear  operators - completeness; 3.3 The bicompletion of an asymmetric normed space; 3.4 Open mapping and closed graph theorems for asymmetric normed spaces; 3.5 Normed cones; 3.6 The $w^\flat $ topology of the dual space $X^\flat_p$; 3.7 Compact subsets of an asymmetric normed space; 3.8  The conjugate  operator, precompact operators between asymmetric normed spaces and a Schauder type theorem;  3.9 Asymmetric moduli of smoothness and rotundity; 3.10 Asymmetric topologies on normed lattices.}

   {\bf 4. Linear functionals on an asymmetric normed space}: {\small 4.1 Some properties of continuous linear functionals; 4.2 Hahn-Banach type theorems; 4.3 The bidual space, reflexivity and Goldstine theorem.}

   {\bf 5. The Minkowski functional and the separation of convex sets}: {\small 5.1 The Minkowski gauge functional - definition and properties; 5.2 The separation of convex sets; 5.3 Extremal points and Krein-Milman theorem.}

{\bf 6. Applications to best approximation}: {\small 6.1 Characterizations of nearest points in convex sets and duality; 6.2 The distance to a hyperplane; 6.3 Best approximation by elements of sets with convex complement; 6.4 Optimal points; 6.5 Sign-sensitive approximation in spaces of continuous or integrable functions. }

{\bf 7. Spaces of semi-Lipschitz functions}: {\small 7.1 Semi-Lipschitz functions - definition,  the extension property, applications to best approximation in quasi-metric spaces; 7.2 Properties of the cone of semi-Lipschitz   functions - linearity, completeness.}

{\sc References}

\section{Introduction}\label{Introd}

The aim of this  paper is to present the basic results on the so called asymmetric normed spaces, by analogy with the classical theory of normed linear spaces. An asymmetric norm is a positive sublinear functional $p$ defined on a real linear  space $X$ such that $p(x)=p(-x)=0$ implies $x=0.$ This functional can be obtained as the Minkowski gauge functional of an  absorbing convex subset, and the possibility $p(x)\neq p(-x)$ is not excluded, explaining the adjective "asymmetric". Asymmetric metric spaces are called quasi-metric spaces.
It is difficult to localize the first moment when asymmetric norms  were used, but it goes back as early as 1968 in a paper by Duffin and Karlovitz \cite{duff-karlovitz68}, who proposed the term asymmetric norm.   Krein and Nudelman \cite{Krein-Nudelman} used also asymmetric norms in their study of some extremal problems related to the Markov moment problem.  Remark that the
relevance of sublinear functionals for some problems of convex analysis and of
mathematical analysis was emphasized also  by H. K\" onig, see \cite{konig72a,konig72b,konig01}
and the survey paper \cite{konig82}.  But a systematic study of the properties of asymmetric normed spaces started with the papers of S. Romaguera, from the Polytechnic University of Valencia, and his collaborators  from the same university and from other universities in Spain: Alegre, Ferrer, Garc\'{\i}a-Raffi, S\' anchez P\' erez,                                                  S\' anchez \' Alvarez,  Sanchis, Valero (see the bibliography at the end of this paper).  Beside its intrinsic interest, their study was motivated also by the applications  in Computer Science, namely to the complexity analysis of programs,  results obtained in cooperation with Professor Schellekens from National University of Ireland.

A general idea about the topics included in this survey can be obtained from the above contents. Due to the fact that completeness and compactness play a central role in functional analysis, we emphasize in the second section some of the difficulties arising in studying the relations between completeness, compactness, total boundedness and precompactness within the framework of quasi-metric and quasi-uniform spaces.

A word must be said about notation. We denote by $\Nat=\{1,2,\dots\}$ the set of natural numbers (positive integers). Intervals are denoted by $[a;b],\,(a;b),\,(a;b],\, [a;b),$ while the notation $(a,b)$ is used to designate an ordered pair.  A closed ball in a quasi-metric space $(X,\rho)$ is denoted by $B_\rho[x,r]=\{y\in X : \rho(x,y)\leq r\},$ and an open ball by $B_\rho(x,r)=\{y\in X : \rho(x,y)< r\}.$  The closed unit ball of an asymmetric normed space $(X,p)$  is denoted by $B_p,$  the open unit ball by $B'_p$ and the unit sphere by $S_p.$ The rest of the notation is standard or explained in the text.

\subsection{ Quasi-metric spaces and asymmetric normed spaces}

A {\it quasi-semimetric} on an arbitrary set $X$ is a mapping $\rho: X\times X\to [0;\infty)$ satisfying
the following conditions:
\begin{align*}%
\mbox{(QM1)}&\qquad \rho(x,y)\geq 0, \quad and  \quad \rho(x,x)=0;\\
\mbox{(QM2)}&\qquad \rho(x,z)\leq\rho(x,y)+\rho(y,z),  %
\end{align*}%
for all $x,y,z\in X.$ If, further,
$$%
\mbox{(QM3)}\qquad \rho(x,y)=\rho(y,x)=0\Rightarrow x=y,
$$%
for all $x,y\in X,$ then $\rho$ is called a {\it quasi-metric}. The pair $(X,\rho)$ is
called a {\it quasi-semimetric space}, respectively  a {\it quasi-metric space}. The conjugate of the
quasi-semimetric $\rho$ is the quasi-semimetric $\bar \rho(x,y)=\rho(y,x),\, x,y\in X.$ The mapping
$ \rho^s(x,y)=\max\{\rho(x,y),\bar \rho(x,y)\},\, x,y\in X,$ is a semimetric on $X$ which is
a metric if and only if $\rho$ is a quasi-metric. Sometimes one works with {\it extended } quasi-semimetrics, meaning that the
quasi-semimetric  $\rho$ can take the value $+\infty$ for some $x,y\in X.$    The following inequalities hold for these
quasi-semimetrics for all $x,y\in X$: %
\bequ\label{ineq-rho} %
\rho(x,y)\leq \rho^s(x,y)\quad\mbox{and}\quad \bar\rho(x,y)\leq \rho^s(x,y). %
\eequ %

An {\it asymmetric norm} on a real vector space $X$ is a functional
$p:X\to [0,\infty)$ satisfying the conditions %
$$ %
\mbox{(AN1)}\; p(x)=p(-x)=0\Rightarrow x=0;\quad\mbox{(AN2)}\;
p(\alpha x)=\alpha p(x);\quad \mbox{(AN3)}\; p(x+y)\leq p(x)+p(y), %
$$%
for all $x,y\in X$ and $\alpha \geq 0.$     If $p$ satisfies only the conditions (AN2) and (AN3),
then it is called an {\it asymmetric seminorm}. The pair $(X,p)$ is called an {\it asymmetric normed} (respectively
{\it seminormed}) {\it space}. Again, in some instances, the value $+\infty$ will be allowed for $p$
in which case we shall call it an {\it extended asymmetric norm } (or seminorm). An asymmetric seminorm $p$ defines a
quasi-semimetric $\rho_p$ on $X$ through the formula %
\bequ\label{def.qm} %
\rho_p(x,y)=p(y-x),\; x,y\in X.%
\eequ

In this case, the inequalities \eqref{ineq-rho} become  %
\bequ\label{ineq-p} %
p(x)\leq p^s(x)\quad\mbox{and}\quad \bar p(x)\leq p^s(x), %
\eequ %
for all $x\in X.$

The conjugates of $\rho$ and $p$ are denoted also by $\rho^{-1}$ and $p^{-1},$  a notation that we shall use occasionally.

If $(X,\rho)$ is a quasi-semimmetric space, then for $x\in X$ and $r>0$ we define the balls in $X$ by the formulae %
\begin{align*}%
B_\rho(x,r)=&\{y\in X : \rho(x,y)<r\} \; \mbox{-\; the open ball, and }\\ %
B_\rho[x,r]=&\{y\in X : \rho(x,y)\leq r\} \; \mbox{-\; the closed ball. } %
\end{align*} %

In the case of an asymmetric seminormed space $(X,p)$ the balls are given by %
\bequs  %
B_p(x,r)=\{y\in X : p(y-x)<r\}, \; \mbox{respectively}\;
B_p[x,r]=\{y\in X : p(y-x)\leq r\}.
\eequs %

The closed unit ball of $X$ is $B_p=B_p[0,1]$ and the open unit ball is $B'_p=B_p(0,1).$
In this case the following formulae hold true %
\bequ\label{translations} %
B_p[x,r]=x+rB_p \quad\mbox{and}\quad B_p(x,r)=x+rB_p'\,, %
\eequ%
that is, any of the unit balls of $X$ completely determines its metric structure.  If necessary, these
balls will be denoted by $B_{p,X}$ and $B'_{p,X},$ respectively.

The conjugate $\bar p$ of $p$ is defined by $p(x)=p(-x),\, x\in X,$ and the associate
seminorm is $p^s(x)=\max\{p(x),\bar p(x)\},\, x\in X.$
The seminorm  $p$ is an asymmetric norm if and only if  $p^s$  is a norm on $X.$  Sometimes an asymmetric norm will
be denoted by the symbol $\|\cdot|,$ a notation proposed by Krein and Nudelman, \cite[Ch. IX, \S 5]{Krein-Nudelman}, in their  book on the theory of moments.

\begin{remark}\label{re.q-Ban} %
{\rm Since the terms "quasi-norm", "quasi-normed space" and "quasi-Banach space" are already "registered trademarks" (see, for instance, the survey by Kalton \cite{kalton03}), we can not use these terms  to designate an asymmetric norm, an asymmetric normed space or an asymmetric biBanach space. A} quasi-normed  space {\rm is a vector space $X$  equipped with a functional $\norm:X\to[0;\infty),$ satisfying all  the axioms of a norm, excepting the triangle inequality which is replaced by:}
$$%
\|x+y\|\leq C(\|x\|+\|y\|),\;x,y\in X,%
$$%
{\rm for some constant $C\geq 1. $ Obviously that for $C=1$ the functional $ \norm$  is a norm. }%
\end{remark} %

\subsection{  The topology of a quasi-semimetric space}

 The topology $\tau(\rho)$ of a quasi-semimetric  $(X,\rho)$ can be defined starting from the family
$\rondv_\rho(x)$ of neighborhoods  of an arbitrary  point $x\in X$:%
\bequs %
\begin{aligned}
V\in \rondv_\rho(x)\;&\iff \; \exists r>0\;\mbox{such that}\; B_\rho(x,r)\subset V\\
                             &\iff \; \exists r'>0\;\mbox{such that}\; B_\rho[x,r']\subset V. %
\end{aligned} %
\eequs %

Obviously that, to see the equivalence  in the above definition, we can take, for instance, $r'=r/2$.

A set $G\subset X$ is $\tau(\rho)$-open if and only if for every $x\in G$ there exists $r=r_x>0$ such that
$B_\rho(x,r)\subset G.$ Sometimes we shall say that $V$ is a $\rho$-neighborhood of $x$ or that the set $G$ is $\rho$-open.

The convergence of a sequence $(x_n)$ to $x$ with respect to $\tau(\rho),$ called $\rho$-convergence and denoted by
$x_n\xrightarrow{\rho}x,$ can be characterized in the following way %
\bequ\label{char-rho-conv1} %
         x_n\xrightarrow{\rho}x\;\iff\; \rho(x,x_n)\to 0. %
\eequ %

Also
\bequ\label{char-rho-conv2} %
         x_n\xrightarrow{\bar\rho}x\;\iff\;\bar\rho(x,x_n)\to 0\; \iff\; \rho(x_n,x)\to 0. %
\eequ %

Using the conjugate quasi-semimetric  $\bar\rho$ one obtains another topology $\tau(\bar\rho).$  A third one is the
topology $\tau(\rho^s)$ generate by the semimetric $\rho^s.$  Sometimes, (see, for instance,
Menucci \cite{menuci04} and Collins and Zimmer \cite{zimmer07}) the balls with respect to $\rho$
are called {\it forward balls} and  the topology $\tau(\rho)$ is called the
{\it forward topology}, while the balls with respect to $\bar\rho$
are called {\it backward balls} and  the topology $\tau(\bar\rho)$  the {\it backward topology}. We shall use
sometimes the alternative notation $\tau_\rho,\,\tau_{\bar \rho},\,\tau_{\rho^s}$ to designate
these topologies.

As a space with two topologies, $\tau_{\rho}$ and $\tau_{\bar \rho},$  a quasi-semimetric space can be viewed as a bitopological space   in the sense of Kelly \cite{kelly63} (see also the book \cite{Bitop05}) and so, all the results valid for bitopological spaces apply to a quasi-semimetric space.  A {\it bitopological space} is simply a set $T$ endowed with two topologies $\tau$ and $\sigma.$ A bitopological space is denoted by $(T,\tau,\sigma).$

The following example is very important in what follows. %
\begin{example}\label{ex.upper-topR} %
{\rm On the field $\Real $ of  real numbers consider the asymmetric norm $u(\alpha)=\alpha^+:=\max\{\alpha,0\}.$
Then, for  $\alpha\in \Real,\, \bar u(\alpha)=\alpha^-:=\max\{-\alpha,0\}$ and $u^s(\alpha)=|\alpha|.$
The topology $\tau(u)$ generated by $u$ is called the {\it upper topology } of $\Real,$ while the topology $\tau(u)$ generated by $\bar u$ is called the {\it lower topology } of $\Real.$ A basis of $\tau(u)$-neighborhoods of a point $\alpha\in \Real$ is formed of the intervals $(-\infty;\alpha+\epsi),\, \epsi > 0.$ A basis of $\tau(\bar u)$-neighborhoods  is formed of the intervals}
$(\alpha-\epsi;\infty),\, \epsi > 0.$

{\rm In this space the addition is continuous from $(\Real\times\Real,\tau_u\times\tau_u)$ to
$(\Real,\tau_u)$, but the multiplication
is not continuous   at any point } $(\alpha,\beta)\in \Real\times\Real.$ %
\end{example}  %

The continuity property can be directly verified. To see the last assertion, let
$V=(-\infty;\alpha\beta+\epsilon),$  be a $\tau_u$-neighborhood of $\alpha\beta,$ for some $\epsi>0.$
Since the  $\tau_u$-neighborhoods of $\alpha$ and $\beta$ contain $-n,$ for $n\in \Nat$
sufficiently large, it follows that   $n^2=(-n)(-n)$ does not belong to $V,$
for  $n$ large enough.

Other important topological example is  the so called Sorgenfrey topology on $\Real$. %
\begin{example}[The Sorgenfrey line]\label{ex.Sorgenfrey} %
 {\rm For $x,y\in \Real $ define a quasi-metric $\rho$ by $\rho(x,y)=y-x,\, $ if $x\leq y$ and $\rho(x,y)=1$ if $x>y.$   A basis of $\tau_\rho$-neighborhoods of a point $x\in \Real$ is formed by the family $[x;x+\epsi),\, 0<\epsi <1.$ The family of intervals $(x-\epsi;x],\, 0<\epsi<1,\,$ forms a basis of $\tau_{\bar \rho}$-neighborhoods of $x.$  Obviously that the topologies $\tau_\rho$ and $\tau_{\bar \rho}$ are Hausdorf and $\rho^s(x,y)=1$ for $x\neq y,$ so that \ $\tau(\rho^s)$ is the discrete topology of $\Real.$ }
   \end{example}

   We shall present, for the  convenience of the reader, the separation axioms. A topological space $(T,\tau)$
   is called  %
   \vspace*{2mm}                                                                                                      \par
   $\bullet$ \;$T_0$  if for any pair $s,t$ of distinct points in $T$, at least one of them has a neighborhood
   not containing the other;
    \par                                                                                                                  $\bullet$ \;$T_1$    if for any pair $s,t$ of distinct points in $T$, each of them has a neighborhood not
   containing the other (this is equivalent to the fact that the set $\{t\}$ is closed for every $t\in T$);
    \par                                                                                                                  $\bullet$ \;{\it  Hausdorff} or $T_2$    if for any pair $s,t$ of distinct points in $T$, there exist the
   neighborhoods $U$ of
   $s$ and $V$ of $t$ such that $U\cap V=\emptyset$;                                                                  \par
   $\bullet$ \;{\it regular}  if  for each $t\in T$ and each closed subset $S$ of $T,$
   not containing $t,$ there are disjoint open subsets $U,V$ of $T$ such
   that $t\in U$ and $S\subset V.$ In other words a point and a closed
   set not containing it can be separated by open sets.  This is equivalent to the fact that every point in $T$
   has a neighborhood base formed of closed sets. If $T$ is regular and $T_1$, then it is called a $T_3$ space.
   \par
   $\bullet$ \;{\it completely regular}, or {\it Tychonoff}, or $T_{3\frac{1}{2}},$                               if for every $ t\in T$ and every closed subset $S$ of  $T$ not containing $t$ there is a continuous function
   $f:T\to [0;1]$ such that $f(t)=1$ and $f(s) =0$ for each $s\in S$.
    \par
   $\bullet$ \;{\it normal} if any pair $S_1,S_2$ of disjoint closed sets can be separated by open sets,
   that is there exist two disjoint  open sets $G_1\supset S_1$ and $G_2\supset S_2.$ A normal $T_1$ space is
   called a $T_4$ space.

   \vspace*{2mm}
   We introduce, following Kelly \cite{kelly63}, some separation properties specific to a bitopological space
   $(T,\tau,\sigma).$
   The bitopological space  $(T,\tau,\sigma)$  is called {\it pairwise Hausdorff} if for
   each pair  of distinct points $s,t\in T$ there exists a $\tau$-neighborhood $U$ of $s$ and                         a $\sigma$-neighborhood $V$ of $t$ such that $U\cap V=\emptyset.$ It is obvious that if $T$ is pairwise
   Hausdorff, then each of the topologies $\tau$ and $\sigma$ are $T_1.$
   The topology $\tau$ is called {\it regular with respect to } $\sigma$ if every $t\in T$ has a
   $\tau$-neighborhood base formed of $\sigma$-closed sets or, equivalently, if for every $t\in T$ and every
   $\tau$-closed subset $S$ of $T$ not containing $t,$ there exist a $\tau$-open set $U$ and a $\sigma$-open set
   $V$ such that $\, t\in U,\, S\subset V$ and $U\cap V=\emptyset.$    The bitopological space $(T,\tau,\sigma)$      is called {\it pairwise regular}  if $\tau$ is regular with respect to $\sigma$ and $\sigma$ is regular with       respect to $\tau.$    The                                                                                          bitopological space $(T,\tau,\sigma)$ is called {\it pairwise normal} if given a $\tau$-closed subset $A$ of       $T$ and a $\sigma$-closed subset $B$ of $T$ with $A\cap B=\emptyset,$ there exist a $\sigma$-open subset $U$       of $T$ and a $\tau$-open subset $V$ of $T$ such that $A\subset U,\, B\subset V,$ and $U\cap V=\emptyset.$
   Using these notions, Kelly proved in \cite{kelly63} some extension and existence results for semicontinuous        functions similar to the classical theorems of Tietze and Uryson.  Note also the following result from Kelly       \cite{kelly63}.                                                                                                   \begin{theo}\label{t.normal-metriz}                                                                               If $(T,\tau,\sigma)$ is a pairwise normal bitopological space  such that both $\tau$ and $\sigma$ satisfy the      second
   axiom of countability, then it is quasi-semimetrizable. If further, $T$ is  pairwise Hausdorff, then it is
   quasi-metrizable.                                                                                                 \end{theo} %
   A bitopological space $(T,\tau,\sigma)$ is called {\it quasi-semimetrizable} if there exists a
   quasi-semimetric   $\rho$ on $T$ such that $\tau=\tau_\rho$ and $\sigma=\tau_{\bar\rho}.$ If $\rho$ is
   a semimetric, then  $\tau=\sigma.$
   The following   topological properties are true for  quasi-semimetric spaces.
    \begin{prop}\label{p.top-qsm1}
   If $(X,\rho)$ is a quasi-semimetric space, then
   \begin{itemize}
   \item[{\rm 1.}] \;  Any ball $B_\rho(x,r)$ is $\tau(\rho)$-open and  a ball $B_\rho[x,r]$ is
       $\tau(\bar\rho)$-closed. The ball    $B_\rho[x,r]$ need not be $\tau(\rho)$-closed.                            \par
       Also, the following inclusions hold
   $$B_{\rho^s}(x,r)\subset B_\rho(x,r)\; \;\mbox{and}   \;\; B_{\rho^s}(x,r)\subset B_{\bar\rho}(x,r), $$
with similar inclusions for the closed balls.
   \item[{\rm 2.}] \; The topology $\tau(\rho^s)$ is finer than the topologies   $\tau(\rho)$ and
    $\tau(\bar \rho).$ This means that:
    \par
   $\bullet$\; any $\tau(\rho)$-open (closed) set is $\tau(\rho^s)$-open (closed); similar results hold for
   the topology $\tau(\bar \rho)$;                                                                                    \par
   $\bullet$\; the identity mappings  from $(X,\tau(\rho^s))$ to $(X,\tau(\rho))$ and to
   $(X,\tau(\bar\rho))$ are continuous;                                                                               \par
   $\bullet$\;  a sequence $(x_n)$ in $X$ is $\tau(\rho^s)$-convergent if and only if it is
   $\tau(\rho)$-convergent  and  $\tau(\bar\rho)$-convergent.
   \item[{\rm 3.}] \;  The topology of a quasi-metric space is pairwise regular and pairwise normal.\\
   If $\rho $ is a quasi-metric, then the topology $\tau(\rho)$ is $T_0,$ but not necessarily $T_1$ (and so           nor $T_2$ as in the case of metric spaces).  \\ The topology $\tau(\rho)$ is $T_1$ if and only if
   $\rho(x,y)>0$ whenever $x\neq y.$  In this case, as a bitopological space, $T$ is  pairwise                        Hausdorff.
   \item[{\rm 4.}] \; For every fixed $x\in X,$ the mapping $\rho(x,\cdot):X\to (\Real,|\cdot|)$ is                   $\tau_\rho$-usc and $\tau_{\bar \rho}$-lsc. \\
   For every fixed $y\in X,$ the mapping $\rho(\cdot,y):X\to (\Real,|\cdot|)$ is $\tau_\rho$-lsc and                  $\tau_{\bar \rho}$-usc.
   \item[{\rm 5.}] \; If  the mapping $\rho(x,\cdot):X\to (\Real,|\cdot|)$ is $\tau_\rho$-continuous for              every $x\in X,$ then the topology $\tau_\rho$ is regular. \\ %
       If $\rho(x,\cdot):X\to (\Real,|\cdot|)$ is $\tau_{\bar\rho}$-continuous for every $x\in X,$ then the           topology $\tau_{\bar\rho}$ is semi-metrizable.
  \end{itemize}%
   \par
   Obviously that, similar results hold for an asymmetric seminorm $p,$ its conjugate $\bar p$ and the
   associated  seminorm $p^s.$                                                                                                    \end{prop} %
   \begin{proof}    1. For $y\in B_\rho(x,r)$ we have $  B_\rho(x,r')\subset B_\rho(x,r),$
   where $r':=r-\rho(x,y)>0.$ Also, if  $y\in B_\rho[x,r]$ and $r':=\rho(x,y)-r>0,$ then
   $B_{\bar\rho}(x,r')\cap B_\rho[x,r]=\emptyset,$ or,
     equivalently, $B_{\bar\rho}(x,r')\subset X\setminus B_\rho[x,r].$ Indeed,
     if $z\in B_{\bar\rho}(x,r')\cap B_\rho[x,r],$ then                                                                                              $$                                                                                                             \rho(x,y)\leq \rho(x,z)+\rho(z,y)=\rho(x,z)+\bar\rho(y,z)<r+r'=\rho(x,y),
     $$
   a contradiction.

   Working in the asymmetric normed space $(\Real,u)$ from Example \ref{ex.upper-topR}, it follows that
   $\Real\setminus B_u[0,1]=(1;\infty)=B_{\bar u}(2,1)$ is $\tau(\bar u)$-open, but not  $\tau(u)$-open.

   The inclusions from 1 follows from the inequalities \eqref{ineq-rho} and, in their turn, they imply
   the assertions from the second point of the proposition.

   3. Since
   $\{B_\rho[x,r] : r>0\}$ is a neighborhood base  of the point $x$ formed of $\tau(\bar \rho)$-closed
   sets and $\{B_{\bar \rho}[x,r] : r>0\}$ is a neighborhood base  of the point $x$ formed of $\tau(\rho)$-closed     sets, it follows that the bitopological space $(X,\tau(\rho),\tau(\bar\rho))$ is pairwise regular.
                                                                                                                      To prove the normality of $X$, let $A,B\subset X,\, A\; \tau(\rho)$-closed,
   $B \; \tau(\bar \rho)$-closed and $A\cap B=\emptyset.$ For $x\in X$ put
   $$%
   \rho(x,A)=\inf\{\rho(x,a) : a\in A\}\quad\mbox{and}\quad \bar\rho(x,B)=\inf\{\bar\rho(x,b) : b\in B\}.$$%
   Obviously that $ \rho(x,A)=0$ for $x\in A$ and $\bar\rho(x,B)=0$ for $x\in B.$ If $x\notin A,$ then, since
   $A$ is $\tau(\rho)$-closed, there exists $r>0$ such that $B_\rho(x,r)\cap A=\emptyset,$ implying
   $\rho(x,A)\geq r>0.$ Similarly, $\bar\rho(x,B)>0$ for $x\notin B,$ showing that %
   $$%
   A=\{x\in X : \rho(x,A)=0\}  \quad\mbox{and}\quad B=\{x\in X : \bar\rho(x,B)=0\}.$$%
                                                                                                                      Let
   $$%
   U=\{x\in X : \rho(x,A)<\bar\rho(x,B)\}  \quad\mbox{and}\quad V=\{x\in X : \bar\rho(x,B)<\rho(x,A)\}.$$%
   Obviously that $U\cap V=\emptyset. $ If $x\in A,$ then $x\notin B,$ so that
   $\rho(x,A)=0<\bar\rho(x,B),$ showing that $A\subset U.$  Similarly, $B\subset V.$
                                                                                                                      Let us show now that $V$ is $\tau(\rho)$-open. For  $x\in V$   put                                                 $r:=\rho(x,A)-\bar\rho(x,B)>0$ and let $y\in X\,$ with $ \rho(x,y)<r/4.$ Then, for every $b\in B,$ %
   $$%
   \bar\rho(y,B)\leq \rho(b,y)\leq\rho(b,x)+\rho(x,y), $$%
   implying %
   \bequ\label{eq1.top-qsm1}%
   \bar\rho(y,B)\leq \bar\rho(x,B) +\rho(x,y).\eequ%

   Similarly, the inequality%
   $$%
   \rho(x,A)\leq \rho(x,a)\leq\rho(x,y)+\rho(y,a),$$
   valid for every $a\in  A,$ implies %
   \bequ\label{eq2.top-qsm1}%
   \rho(x,A) \leq\rho(x,y)+\rho(y,A).\eequ

   Now from \eqref{eq1.top-qsm1} and \eqref{eq2.top-qsm1}, one obtains %
   $$%
   \rho(y,A)-\bar\rho(y,B)\geq\rho(x,A)-\bar\rho(x,B)-2\rho(x,y)>r/2>0.$$%

   Consequently, $B_\rho(x,r/4)\subset V,$ proving that $V$ is $\tau(\rho)$-open. The $\tau(\bar\rho)$-openess of
   $U$ follows by the symmetry between $\rho$ and $\bar \rho$ ($\bar{\bar \rho}=\rho$).

   If $x,y$ are distinct points in the quasi-metric space $(X,\rho),$ then $\max\{\rho(x,y),\rho(y,x)\}> 0. $
   If $\rho(x,y)>0,$ then $y\notin B_\rho(x,r),\, $ where $r=\rho(x,y).$  Similarly, if $\rho(y,x)>0,$ then
   $x\notin B_\rho(y,r'),$ where $r'=\rho(y,x).$

   If  $2r:=\rho(x,y)>0,$    then
     $B_\rho(x,r)\cap B_{\bar\rho}(y,r)=\emptyset.  $ Indeed, if   $z\in B_\rho(x,r)\cap B_{\bar\rho}(y,r),$
     then %
     $$%
     \rho(x,y)\leq\rho(x,z)+\rho(z,y)<r+r=\rho(x,y), %
     $$%
     a contradiction.

        It is easy to check that if $\tau_\rho$ is $T_1,$ then $\rho(x,y)>0$ for every pair of distinct elements
     $x,y\in X.$

     The quasi-metric  space $(\Real,u)$ from Example \ref{ex.upper-topR} is not $T_1$ because, for instance,
     any neighborhood of 1 contains 0.

   4. To prove that $\rho(x,\cdot)$ is $\tau_\rho$-usc and $\tau_{\bar\rho}$-lsc, we have to show that the set
    $\{y\in X : \rho(x,y)<\alpha\}$ is $\tau_\rho$-open and $\{y\in X : \rho(x,y)>\alpha\}$ is
    $\tau_{\bar\rho}$-open, for every $\alpha \in \Real, $   properties that are easy to check.

   Indeed, for $y\in X$ such that   $\rho(x,y)<\alpha,$ let $r:=\alpha-\rho(x,y)>0.$ If $z\in X\,$ is such that
   $\rho(y,z)<r,$ then %
   $$%
   \rho(x,z)\leq\rho(x,y)+\rho(y,z)<\rho(x,y)+r=\alpha,$$%
   showing that $B_\rho(y,r)\subset  \{y\in X : \rho(x,y)<\alpha\}$.

   Similarly, for $y\in X$ with $\rho(x,y)>\alpha$ take $r:=rho(x,y)-\alpha>0.$ If $z\in X$ satisfies
    $\rho(z,y)=\bar\rho(y,z)<r,$ then %
    $$%
    \rho(x,y)\leq\rho(x,z)+\rho(z,y)<\rho(x,z)+r,$$%
    so that $\rho(x,z)>\rho(x,y)-r=\alpha.$ Consequently,
     $B_{\bar\rho}(y,r)\subset  \{y\in X : \rho(x,y)>\alpha\}$.

   For a proof of 5 see Kelly  \cite{kelly63}. %
   \end{proof} %

\begin{remark}\label{rem.top-qsm1} %
{\rm The lower and upper continuity properties  from the assertion 4 of Proposition \ref{p.top-qsm1} are equivalent to the fact that the mapping $\rho(x,\cdot):X\to \Real$ is $(\tau_\rho,\tau_u)$-continuous, respectively   $(\tau_{\bar \rho},\tau_{\bar u})$-continuous. Similar equivalences hold for the semicontinuity properties of the mapping} $\rho(\cdot,y).$ %
\end{remark}

As we have seen in the preceding proposition, the topology  generated by an asymmetric norm is not
always Hausdorff.  A condition that this topology  be Hausdorff  was found by Garc\'{\i}a-Raffi, Romaguera and S\'{a}nchez-P\'{e}rez \cite{rafi-rom-per03b},
in terms of a functional $p^\diamond:X\to [0;\infty)$ associated to an asymmetric seminorm $p$
defined on a real vector space $X.$ The result was extended to asymmetric locally convex spaces
in \cite{cobz05a}. The functional $p^\diamond$ is defined by the formula %
\begin{equation}\label{def.diamond-p} %
p^\diamond(x) = \inf\{p(x')+p(x'-x) : x'\in X\},\; x\in X. %
\end{equation} %

In the following proposition we present the properties of $p^\diamond.$
\begin{prop}\label{p.diamond-p} %
 The functional $p^\diamond$ is a (symmetric) seminorm on $X,\;
p^\diamond\leq p,\,$ and $p^\diamond$ is the greatest of the seminorms
 on  $X$ majorized by $p$.   %
 \end{prop} %
 \begin{proof} %
 First observe that, replacing $x'$ by $x'-x$ in \eqref{def.diamond-p}, we get%
$$%
\begin{aligned}%
p^\diamond(-x) =&\inf\{p(x') + p(x'+x) :x'\in X\} = \\
=&\inf\{p(x'-x) +
p((x'-x)+x) :x'\in X\} = p^\diamond(x), \end{aligned}$$%
so that $p^\diamond$ is symmetric. The positive homogeneity of
$p^\diamond,\; p^\diamond(\alpha x) =\alpha p^\diamond(x),\;x\in X,\,
\alpha \geq 0,\,$ is obvious. For $x,y\in X$ and arbitrary
$x',y'\in X$ we have %
$$%
p^\diamond(x+y) \leq p(x'+y') + p(x'+y'-x-y)\leq p(x') + p(x'-x)
+p(y') + p(y'-y), $$
so that, passing to infimum with respect to $\,x',y'\in X,\, $ we
obtain
the subadditivity of $p^\diamond,$ %
$$%
p^\diamond(x+y) \leq p^\diamond(x) + p^\diamond(y). $$%

Suppose now that there exists a seminorm  $q$ on $X$ such that
$q\leq p,\,$ i.e, $\,\forall z\in X,\; q(z)\leq p(z),\,$ and
$p^\diamond(x) < q(x) \leq p(x)$, for some $x\in X$. Then, by the
definition of $p^\diamond$, there exists $x'\in X$ such
that %
$p^\diamond (x) < p(x') +p(x'-x) < q(x),$ %
leading to the contradiction %
$$%
q(x)  \leq q(x') + q(x-x') = q(x') + q(x'-x) \leq p(x') +
p(x'-x)< q(x). $$%
\end{proof} %

 In the following proposition we collect the separation properties of an asymmetric normed space.%
  \begin{prop}[\cite{rafi-rom-per03b}]\label{p.Hausd}  %
Let $(X,p)$ be an asymmetric normed  space.
\begin{itemize}%
\item[{\rm 1.}]   The topology $\tau_p$ is always $T_0.$ The topology $\tau_p$ is $T_1$ if and only if
$p(x)>0$ for every $x\in X,\, x\neq 0.$
\item[{\rm 2.}] The topology $\tau _p$  is Hausdorff if and
only if $p^\diamond(x) > 0$  for every $x\neq 0.$ %
\end{itemize} %
\end{prop} %
\begin{proof} %
The assertions from 1 follows from Proposition \ref{p.top-qsm1}.3.

2. If  $\,p^\diamond(x)>0$ whenever $x\neq 0,$  then $p^\diamond $ is a norm on $X,$ so that
the topology $\tau_{p^\diamond}$ generated by $p^\diamond $ is Hausdorff. The inequality
$p^\diamond\leq p$ implies that the topology $\tau_p$ is finer than $\tau_{p^\diamond},$
so it is Hausdorff, too.

  Suppose $p^\diamond(x)=0$ for some $x\neq 0.$
 By the definition \eqref{def.diamond-p}
of $p^\diamond,$  there exists a sequence $(x_n)$ in $X$ such that $\lim_n[p(x_n)+p(x_n-x)]=p^\diamond(x)=0.$
This implies $\lim_np(x_n)=0$ and $\lim_np(x_n-x)=0,$  showing that the sequence $(x_n)$ has two limits with respect to $\tau_\rho$. Consequently, the topology $\tau_p$ is not Hausdorff.   %
 \end{proof}

 \begin{remark}\label{rem.sep-asym-ns} {\rm A deep results in functional analysis asserts that a $T_0$
 topological vector space  (TVS) is Hausdorff and completely regular (see \cite[Theorem 2.2.14]{Meg98}), a
 result that is no longer true in asymmetric normed spaces.}\end{remark}

   The following proposition shows that an asymmetric normed space is not necessarily a topological vector space.

 \begin{prop}\label{p.top-asym-n1} %
If  $(X,p)$ is   an asymmetric normed space, then
the topology $\tau_p$ is translation invariant, so that the addition
$+:X\times X\to X$ is continuous. Also any additive mapping between two  asymmetric normed spaces $(X,p),(Y,q)$ is continuous if and only if it is continuous at $0\in X$   (or at an arbitrary point $x_0\in X$).

 The multiplication is not always continuous, so that  an asymmetric normed space need not be a topological vector space. %
\end{prop} %
\begin{proof} %
 The fact that the topology $\tau_p$ is translation invariant follows from the formulae
\eqref{translations}.  Example \ref{ex.upper-topR} shows  that the multiplication by scalars need not be continuous.%
\end{proof}

Another example was given by Borodin \cite{borod01}.  %
\begin{example}\label{ex.borodin} %
 {\rm In the space $X=C_0[0;1],$ where %
$$ C_0[0;1] = \{ f\in C[0;1] : \int_0^1 f(t) dt = 0\},$$%
with the asymmetric seminorm  $\,p(f) = \max\, f([0;1]),$ the multiplication by scalars is
not continuous at the point}  $ (-1,0)\in \Real\times X.$  %
\end{example}

 To prove this, we show that  $(-1)B_p[0,r]$ is not contained in $B_p[0,1] $ for any $r>0.$ Indeed,
 let $t_n=1/n$ and
 $$%
  f_n(t)=
 \begin{cases} %
  r(n-1)(nt-1), \;\;\; 0\leq t\leq t_n,\\
    r\frac{n-1}{n}(t-\frac{1}{n}),\;\;\;\qquad t_n<t\leq 1,%
    \end{cases} %
    $$%
for $n\in \Nat.\,$
  Then $\,f_n\in C_0[0;1], \, p(f_n) =r, \, -f_n\in
   C_0[0;1]\,$  and  $\,p(-f_n)=(n-1)r>r\,$ for $\,n>2.$

The following proposition shows that $\rho^s$-separability  of a quasi-metric space is  stronger than the
$\rho$-separability. %
\begin{prop}\label{p.separable-qm} %
Let $(X,\rho)$ be a quasi-semimetric space. If the semimetric space  $(X,\rho^s)$ is separable, then the space
$(X,\rho)$ is separable (with respect to the topology $\tau_\rho$). %
\end{prop} %
\begin{proof} %
Let $Y$ be a countable $\tau_{\rho^s}$-dense subset of $X.$ Then, for every $\epsi >0$ and any $x\in X$ there exists $y\in Y$ such that $x\in B_{\rho^s}(y,\epsi),$ that is $\rho^s(y,x)<\epsi. $ But then $\rho(y,x)\leq \rho^s(y,x)<\epsi, $ i.e., $x\in B_\rho(y,x),\,$ showing that   $Y$ is  $\rho$-dense in $X.$ %
\end{proof}

The following example shows that the converse is not true, in general. %
\begin{example}[Borodin \cite{borod01}]\label{ex.separable} %
{\rm There exists an asymmetric normed space $(X,p)$ which is $\tau_p$-separable but not
$\tau_{p^s}$-separable.}%
\end{example} %

Take %
\bequ\label{def.separable} %
X=\{x\in \ell^\infty : x=(x_k),\; \sum_{k=1}^\infty\frac{x_k}{2^k}=0\}, %
\eequ %
with the asymmetric norm $p(x)=\sup_kx_k.$ It is clear that $p$ is an asymmetric norm on $X$ satisfying
$p(x)>0$ whenever $x\neq 0$ and $p^s(x)=\|x\|_\infty=\sup_k|x_k|$ is the usual sup-norm on $\ell^\infty.$  Because $\vphi(x)=\sum_{k=1}^\infty 2^{-k}x_k, \, x=(x_k)\in \ell^\infty, $
is a continuous linear functional on $\ell^\infty,$ it follows that $X=\ker \vphi$ is a codimension one closed subspace  of $\ell^\infty.$ Since $\ell^\infty$ is nonseparable with respect to $p^s,\, X$ is also nonseparable with respect to $p^s.$

Let us show that $X$ is $p$-separable. Consider the set $Y$ formed of all $y=(y_k)$
such that $y_k\in \mathbb Q$ for all $k$ and there exists $n=n(y)$ such that $y_k=y_{n+1}$ for all $k>n.$
It is clear that $Y$ is contained in   $X$ and that $Y$ is countable. To show that $Y$ is $p$-dense in $X,$
let $x\in X$ and $\epsi > 0.$

Choose $n\in \Nat$ such that %
\bequ\label{eq1.separable} %
\sum_{i=1}^n2^{-i}\epsi-\sum_{i=1}^n2^{-i}x_i>\|x\|_\infty\sum_{j=n+1}^\infty 2^{-j}. %
\eequ

This is possible because the left-hand side of \eqref{eq1.separable} tends to $\epsi$ for $n\to \infty,$
while the right-hand side tends to 0. Choose $y_k\in \mathbb Q\cap(x_k-2\epsi;x_k-\epsi)$ for $k=1,\dots,n,$
and let %
 $$%
y_k=\alpha:=-\big(\sum_{i=1}^n2^{-i}y_i\big):\big(\sum_{j=n+1}^\infty 2^{-j}\big), %
$$%
for $k>n.$ Then  $y=(y_k)\in Y,\, \epsi < y_k-x_k<2\epsi,\,$ for $k=1,\dots,n\,$ and,
 by \eqref{eq1.separable}, %
$$%
y_k=\alpha > \big(\sum_{i=1}^n2^{-i}(\epsi-x_i)\big):\big(\sum_{j=n+1}^\infty 2^{-j}\big)>\|x\|_\infty,%
$$%
for $k>n.$ It follows $x_k-y_k\leq \|x\|_\infty-y_k<0\,$ for $k>n,$ so that $p(x-y)=\max\{x_k-y_k : 1\leq k\leq n\}< 2\epsi. $

As it is well-known, by the classical Banach-Mazur theorem, any separable real Banach space can be linearly and isometrically  embedded in the Banach  space $C[0;1]$ of all continuous real-valued functions on $[0;1]$ with the sup-norm. In other words, $C[0;1]$ is a universal space in the category of separable real Banach spaces. The validity of this result in the case of asymmetric normed spaces was discussed by Alimov \cite{alimov03} and Borodin \cite{borod01}. The above example shows that  some attention must be paid to the notion of separability we are using.

Denote by $(C[0;1],1,0)$ the space of all  real-valued continuous on $[0;1]$ functions
with the asymmetric norm %
\bequ\label{def.norm-C} %
\|f|=\max\{f(t) : t\in [0;1]\}, %
\eequ %
for $f\in C[0;1].$  %
\begin{theo}[\cite{borod01}]\label{t.Maz1} %
  Any $T_1$ asymmetric normed space $(X,p)$ such that the associated normed space $(X,p^s)$ is separable is isometrically isomorphic to a subspace of the asymmetric normed space $(C[0;1],1,0).$
\end{theo}

Since any linear isometry from $(X,p)$  to  $(C[0;1],1,0)$ induces a linear isometry from $(X,p^s)$ to the usual Banach space $C[0;1],$ the separability condition with respect to the symmetric norm $p^s$ is necessary for the validity of the Mazur theorem.

Some complements to this result were done by Alimov \cite{alimov03}. %
\begin{theo}[\cite{alimov03}]\label{t.Maz2}
A $T_1$ asymmetric normed space $(X,p)$ is isometrically isomorphic to a subspace of the usual Banach space $C[0;1]$ if and only if the topology $\tau_p$ is metrizable and separable.
\end{theo}

Note that the topology $\tau_p$ of a $T_1$ asymmetric normed space  $(X,p)$ is metrizable if and only if $\tau_p$ is generated by a norm if and only if $\tau_p=\tau_{\bar p}$ (see \cite{alimov03}).

The Banach-Mazur theorem asserts that any separable metric space can be isometrically embedded in $C[0;1].$
Kleiber and Pervin \cite{kleiber-pervin69} extended this result to metric spaces of arbitrary density charcter $\alpha,$ where $\alpha$ is an uncountable cardinal number, by proving that such a space can be isometrically embedded in the space $C([0;1]^\alpha).$ The density character of a topological space $T$ is the smallest cardinal number $\alpha$ such that $T$ contains a dense subset of cardinality $\alpha.$ As remarked  Alimov \cite{alimov03}, these results hold also for a quasi-metric space $(X,\rho),$ the density character being that of the associated metric space $(X,\rho^s).$

\subsection {Quasi-uniform spaces}

Quasi-semimetric spaces are particular cases of  quasi-uniform spaces.
A quasi-uniformity on a set $X$ is a filter $\mathcal U$ on $X\times X$ such that %
\bequs %
\begin{aligned} %
\mbox{(QU1)}&\qquad \Delta(X)\subset U,\; \forall U\in \rondu;\\
\mbox{(QU1)}&\qquad \forall U\in \rondu,\; \exists V\in \rondu,\;
\mbox{such that } \; V\circ
V\subset U, %
\end{aligned}
\eequs %
where $\Delta(X)=\{(x,x) : x\in X\}$ denotes the diagonal of $X$
and, for $M,N\subset X\times X,$  %
\bequs %
M\circ N =\{(x,z)\in X\times X : \exists y\in X,\; (x,y)\in
M\;\mbox{and} \; (y,z)\in N\}. %
\eequs %

If the filter $\rondu$ satisfies also the condition %
\bequs %
\mbox{(U3)}\qquad \forall U,\; U\in \rondu\;\Rightarrow\;
U^{-1}\in\rondu,
\eequs %
where %
$$%
U^{-1}=\{(y,x)\in X\times X : (x,y)\in U\}, %
 $$ %
then $\rondu$ is called a {\it uniformity} on $X.$  The sets in
$\rondu$ are called {\it entourages}.

For $U\in \rondu, \, x\in X$ and $Z\subset X$ put %
$$%
U(x) =\{y\in X: (x,y)\in U\}\quad\mbox{and}\quad U[Z]=\bigcup\{U(z):
z\in Z\}. %
$$ %
 A quasi-uniformity $\rondu$ generates a topology $\tau(\rondu)$
 on $X$ for which the family of sets%
 \bequs %
\{ U(x): U\in \rondu\} %
\eequs %
is a base of neighborhoods of the point $x\in X.$  A mapping $f$
between two quasi-uniform spaces $(X,\rondu),\,(Y,\mathcal W)$ is
called {\it quasi-uniformly continuous} if for every $W\in
\mathcal W$ there exists $U\in \rondu$ such that $(f(x),f(y))\in
W$ for all $(x,y)\in U.$ By the definition of the topology
generated by a quasi-uniformity, it is clear that a
quasi-uniformly continuous mapping is continuous with respect to
the topologies $\tau(\rondu),\, \tau(\mathcal W).$

If $(X,\rho)$ is a quasi-semimetric space, then %
\bequs %
 B'_\epsilon =\{(x,y)\in X\times X : \rho(x,y)< \epsilon\},\;
 \epsilon > 0, %
 \eequs %
 is a basis for a quasi-uniformity $\rondu_\rho$ on $X.$  The
 family %
 \bequs %
 B_\epsilon =\{(x,y)\in X\times X : \rho(x,y)\leq \epsilon\},\; \epsilon > 0, %
 \eequs %
 generates the same quasi-uniformity. Since $U'_\epsi(x)=B_\rho(x,\epsi)$ and
 $U_\epsi(x)=B_\rho[x,\epsi]$, it follows that the topologies generated by the
 quasi-semimetric $\rho$ and by the quasi-uniformity $\rondu_\rho$
 agree, i.e., $\, \tau_\rho=\tau(\rondu_\rho).$     As it was shown by Pervin \cite{pervin62}, every topological space is quasi-uniformizable, but the quasi-uniformity generating the topology is not unique.

 An account of the theory of quasi-uniform and quasi-metric spaces up to 1982
is given in the book by Fletcher and Lindgren \cite{FL}. The
survey papers by K\" unzi \cite{kunz93a,kunz95,kunz01,kunz02b,kunz07}
are good guides for subsequent developments. Another book on
quasi-uniform spaces is \cite{MN}.  The properties of bitopologies and quasi-uniformities on function spaces were  studied in \cite{romag-gomez95}.

\section{Completeness and compactness in quasi-metric and in
quasi-uniform spaces}\label{Complete-Comp}

\subsection{Various notions of completeness for quasi-metric spaces}

The lack of symmetry in the definition of quasi-metric and
quasi-uniform spaces causes a lot of troubles, mainly concerning
completeness, compactness and total boundedness in such spaces.
There are a lot of completeness notions in quasi-metric and in
quasi-uniform spaces, all agreeing with the usual notion of
completeness in the case of metric or uniform spaces,  each of
them having its advantages and weaknesses.

We shall describe briefly some of these notions along with some of
their properties.

In the case of a quasi-metric space $(X,\rho)$ there are several
 completeness notions, which  we present  following
\cite{reily-subram82}, starting with the definitions of Cauchy
sequences.

A sequence $(x_n)$ in $(X,\rho)$ is called

(a) \;{\it left (right) $\rho$-Cauchy} if for every $\epsilon > 0$
there exist $x\in X$ and $n_0\in \Nat$ such that

\quad\quad $\forall n,\, n_0\geq n,\;\Rightarrow\; \rho(x,x_n)<\epsilon$
(respectively $\rho(x_n,x)<\epsilon$)\,;

(b) \; $\rho^s$-{\it Cauchy} if for every $\epsilon >0$ there exists
$n_0\in \Nat$ such that

\quad\quad $\forall n,k\geq n_0,\; \rho(x_n,x_k)< \epsilon\,$;

(c)\;  {\it left (right)$K$-Cauchy} if  for every $\epsilon >0$
there exists $n_0\in \Nat$ such that

\quad \quad$\forall n,k, \; n_0\leq k\leq n \; \Rightarrow \;
\rho(x_k,x_n)< \epsilon$ (respectively
$\rho(x_n,x_k)<\epsilon$)\,;

(d)\; {\it weakly left (right) $K$-Cauchy} if for every $\epsilon >0$
there exists $n_0\in \Nat$ such that

\quad \quad$\forall n,\, n_0\leq n,\;\Rightarrow \rho(x_{n_0},x_n)< \epsilon$
(respectively $\rho(x_n,x_{n_0})< \epsilon$).

Some remarks are in order.%
\begin{remarks}\label{rem.Cauchy-seq} %
{\rm 1.}\; {\rm These notions are related in the following way:} \vspace*{2mm}

{\rm left(right) $K$-Cauchy $\;\Rightarrow\; $ weakly left(right)
$K$-Cauchy $\;\Rightarrow\;$ left(right) $\rho$-Cauchy, }
\vspace*{2mm}\newline%
{\rm and no one of the above implications is reversible
(see \cite{reily-subram82}).}

{\rm 2.\; Obviously that a sequence is left Cauchy (in some sense) with respect to $\rho$
if and only if it is right Cauchy (in the same sense) with respect to }$\bar\rho.$

{\rm 3\; A $\rho$-convergent sequence is left $\rho$-Cauchy and a
$\bar \rho$-convergent sequence is right $\rho$-Cauchy}

{\rm 4.    A sequence is $\rho^s$-Cauchy if and only if it is both left and right
$K$-Cauchy.}

{\rm 5. If each convergent sequence in a regular quasi-metric space $(X,\rho)$
admits a left $K$-Cauchy subsequence, then $X$ is metrizable
(\cite{kunz-reily93})}. %
 \end{remarks}

By the remark 3 from above,  each $\rho$-convergent sequence is left $\rho$-Cauchy, but
for each of the other notions there are examples of
$\rho$-convergent sequences that are not Cauchy, which is a major
inconvenience. Another one is that closed subspaces of complete
(in some sense) quasi-metric spaces need not be complete.  The above remark 5  shows that putting too many
conditions on a quasi-metric, or on a quasi-uniform space, in
order to obtain results similar to those in the symmetric case,
there is the danger to force the quasi-metric to be a metric and
the quasi-uniformity a uniformity. In fact, this is a general
problem when dealing with generalizations.

  The quasi-semimetric space is called {\it bicomplete} if the associated semimetric space $(X,\rho^s)$ is complete.
  It is called $\rho$-{\it sequentially  complete} if any $\rho^s$-Cauchy sequence is $\tau_\rho$-convergent.
  A bicomplete asymmetric normed space $(X,p)$ is called a {\it biBanach space.}

To  each of the other notions of Cauchy sequence corresponds a notion
of sequential  completeness,  by asking that each corresponding Cauchy
sequence be convergent in $(X,\tau_\rho).$

It follows that the implications between these
completeness notions are obtained by reversing the implications between the corresponding notions of Cauchy sequence from Remark \ref{rem.Cauchy-seq}.1.  %

\begin{remark}\label{rem.completeness} %
 {\rm These notions of completeness are related in the following way:} \vspace*{2mm}

{\rm left (right) $\rho$-sequentially complete $\;\Rightarrow\;$ weakly
left (right) $K$-sequentially complete $\;\Rightarrow\;$  left (right) $K$-sequentially complete.}

\vspace*{2mm}

\end{remark}

In spite of the obvious fact that left $\rho$-Cauchy is equivalent
to right $\bar \rho$-Cauchy, left $\rho$- and right
$\bar\rho$-completeness do not agree, due to the fact that right
$\bar \rho$-completeness means that every left $\rho$-Cauchy
sequence converges in $(X,\bar \rho),$ while left
$\rho$-completeness means the convergence of such sequences in the
space $(X,\rho).$ For concrete examples and counterexamples, see
\cite{reily-subram82}.  In fact, as remarked Mennucci \cite[\S 3.ii.2]{menuci04},  starting from these 7 notions of Cauchy sequence, one can obtain (taking into account the symmetry between $\rho$ and $\bar\rho$)
14 different notions of completeness,
  by asking that every sequence which is Cauchy in some sense for $\rho$
converges with respect to one of the topologies $\tau(\rho),\, \tau(\bar\rho)$ or $\, \tau(\rho^s)$.
Mennucci  \cite{menuci04} works with the following notion of completeness: any left $K$-Cauchy sequence
is  $\tau_{\rho^s}$-convergent.

We mention the following example from \cite{reily-subram82}. %
\begin{example}\label{ex.Cauchy} %
{\rm On the set $X=\Nat$ define the quasi-metric $\rho$ by $\rho(m,n)=0$ if $m=n,\, \rho(m,n)=n^{-1}$ if $m>n,\, m $ even, $n$ odd, and $\rho(m,n)=1$ otherwise. Since there are no non-trivial right $K$-Cauchy sequences,  $X$ is right $K$-sequentially complete. The quasi-metric  space $X$ is not right $\rho$-sequentially complete because the sequence
$\{2,4,6,\dots\}$ is right $\rho$-Cauchy but not convergent.  This sequence is left $\bar\rho$-Cauchy but not weakly left $K$-Cauchy in $(X,\bar\rho).$ Also, the space $(X,\bar\rho)$ is left $K$-sequentially complete but not left $\bar \rho$-sequentially complete.} %
\end{example}

 The following simple remarks concerning sequences in quasi-semimetric spaces are true. %
 \begin{prop}\label{p.conv-seq}%
 Let $(x_n)$ be as sequence in a quasi-semimetric space $(X,\rho).$ %
 \begin{itemize} %
\item[{\rm 1.}] If $(x_n)$ is $\tau_\rho$-convergent to $x$ and $\tau_{\bar\rho}$-convergent to $y,$ then $\rho(x,y)=0.$
\item[{\rm 2.}]  If $(x_n)$ is $\tau_\rho$-convergent to $x$ and $\rho(y,x)=0,$  then $(x_n)$ is  also $\tau_\rho$-convergent to $y.$
\item[{\rm 3.}]  If $(x_n)$  is left $K$-Cauchy and has a subsequence which is $\tau_\rho$-convergent to to $x,$   then $(x_n)$ is $\tau_\rho$-convergent to $x.$
\item[{\rm 4.}]    If $(x_n)$  is left $K$-Cauchy and has a subsequence which is $\tau_{\bar\rho}$-convergent to to $x,$   then $(x_n)$ is $\tau_{\bar\rho}$-convergent to $x.$
\end{itemize} %
\end{prop} %
\begin{proof}  %
1.\;Letting $n\to \infty$ in the inequality $\rho(x,y)\leq\rho(x,x_n)+\rho(x_n,y),$ one obtains $\rho(x,y)=0.$

2.\;Follows from the relations $\rho(y,x_n)\leq \rho(y,x)+\rho(x,x_n)=\rho(x,x_n)\to 0$ as $n\to \infty.$

3.\;Suppose that $(x_n)$ is left $K$-Cauchy and $(x_{n_k})$ is a subsequence of $(x_n)$ such that
$\lim_k\rho(x,x_{n_k})$ $= 0.$ For $\epsi>0$ choose $n_0$ such that $n_0\leq m\leq n$ implies $\rho(x_m,x_n)<\epsi,$ and let $k_0\in \Nat$ be such that $n_{k_0}\geq n_0$ and $\rho(x,x_{n_k})<\epsi$ for all $k\geq k_0.$ Then, for
$n\geq n_{k_0}),\; \rho(x,x_n)\leq\rho(x,x_{n_{k_0}}+\rho(x_{n_{k_0}},x_n)< 2\epsi.$

4.\;Reasoning similarly, for $n\geq n_{k_0}$ let $k\in \Nat$ such that $n_k\geq n. $ Then
$\rho(x_n,x)\leq \rho(x_n,x_{n_k})+\rho(x_{n_k},x)< 2\epsi. $%
\end{proof} %

Concerning Baire's characterization of completeness in terms of descending sequences of closed sets
we mention the following result from \cite[Th. 10]{reily-subram82}.   The {\it diameter} of a subset $A$
of $X$ is defined by %
\bequ\label{def.diam}%
\diam (A) =\sup\{\rho(x,y) : x,y\in A\}. %
\eequ%

It is clear that the diameter, as defined, is in fact the diameter with respect to the associated  semimetric
$\rho^s.$

\begin{theo}[\cite{reily-subram82}]\label{t.compl-char} %
A quasi-semimetric space $(X,\rho)$ is $\rho$-sequentially complete if and only if each decreasing sequence $F_1\supset F_2\dots$
of nonempty closed sets with $\diam F_n\to 0$ as $n\to \infty$ has nonempty intersection, which is a singleton if $\rho$ is a quasi-metric. %
\end{theo}

\subsection{Compactness, total boundedness and precompactness}

A subset $Y$ of a quasi-metric space $(X,\rho)$ is called {\it
precompact} if for every $\epsilon > 0$ there exists a finite
subset $Z$ of $Y$ such that %
\bequ\label{def.tot-bd}%
Y\subset \cup\{B_\rho(z,\epsilon) : z\in Z\}. %
\eequ%

If for every $\epsi>0$ there exists a finite subset $Z$ of $X$ such that \eqref{def.tot-bd} holds,
then the set $Y$ is called {\it outside precompact}. One obtains the  same notions if one works with closed balls $B_\rho[z,\epsi],\, z\in Z.$
Obviously that a precompact set is outside precompact, but
the converse is not true, even in asymmetric normed spaces, see \cite{aleg-fer-raf08}. %
\begin{prop}[\cite{aleg-fer-raf08}]\label{p.seq-precomp} Let $(X,\rho)$ be a quasi-semimetric space and $(x_n)$ a sequence in $X.$.%
\begin{itemize}%
\item[{\rm 1.}]  If $(x_n)$ is weakly left $K$-Cauchy, then $(x_n)$ is precompact.
\item[{\rm 2.}] If $(x_n)$ is $\rho$-convergent, then it is outside precompact. If $x$ is a limit
of $(x_n)$, then $\{x\}\cup\{x_n : n\in \Nat\}$ is precompact.
\item[{\rm 3.}] There exists $\rho$-convergent sequences which are not precompact.%
\end{itemize} %
\end{prop} %
\begin{proof} %
1. For $\epsi > 0$ there exists $k$ such that $\rho(x_k,x_n)<\epsi,$ for every $n\geq k, $ implying
$\{x_n : n\in \Nat\} \subset \cup\{B_\rho(x_i,\epsi) : 1\leq i\leq k\}.$

2. If $x$ is a limit of $(x_n)$, then for every $\epsi >0$ there exists $k\in \Nat$ such that $p(x_n-x)<\epsi,$ for all $n > k,$ implying  $\{x\}\cup\{x_n : n\in \Nat\}\subset B_\rho(x,\epsi)\cup\{B_\rho(x_i,\epsi) : 1\leq i\leq k\}.$

 If $(x_n)$  is $\rho$-convergent to some $x\in X,$
then the above reasoning shows that $\{x_n : n\in \Nat\}$ is outside precompact.

3. We shall present,  a counterexample in the space $\ell_\infty$ of all bounded real sequences
equipped with the asymmetric norm $p(x)=\sup_i x_i^+, $ for $x=(x_i)_{i=1}^\infty\in \ell^\infty.$
Consider the sequence $ x_n=(1,1,\dots,\underset{n}{1},0,0,\dots),\,n\geq 1. $ If $z=(1,1,\dots),$ then
$p(x_n-z)=p(0,\dots,\underset{n}{0},-1,-1,\dots))=0,  $ for all $n,$ so that $(x_n)$ converges to $z$ with respect to $p$.
Let $\epsi=1/2$ and let  $n_1<n_2<\dots<n_k$ be an arbitrary finite subset of $\Nat.$  Then, for every $n>n_k,\, p(x_n-x_{n_i})=1,\, i=1,\dots,k,\,$  showing that $\{x_n : n\in \Nat\}$ is not contained in $\cup\{B_\rho(x_{n_i},\epsi) : 1\leq i\leq k\}.$   %
\end{proof}

The set $Y$ is called {\it totally bounded } if for every
$\epsilon> 0,\; Y$ can be covered by a finite family of sets of
diameter less than $\epsilon,$  where the diameter of a set is defined by \eqref{def.diam}.
Total boundedness implies precompactness. Indeed, if $Y\subset \cup\{A_i : 1\leq i\leq n\}$ where $\diam(A_i)<\epsi$  and $A_i\cap Y\neq \emptyset,\,$ for $i=1,\dots,n,\,$ then, taking $z_i\in A_i\cap Y, \,1\leq i\leq n,\,$ it follows $Y\subset \cup\{B_\rho(z_i,\epsi) : 1\leq i\leq n\}.$

As it is known, in metric spaces  the precompactness, the outside precompactness  and the total
boundedness are equivalent notions, a result that is not longer
true in quasi-metric spaces, where outside  precompactness is strictly
weaker than precompactness, which, in its turn is strictly weaker than total boundedness, see \cite{lambrin77} or \cite{MN}.
Also, $p^s$-precompactness implies $p$-precompactness and $\bar p$-precompactness,   but the converse is not true, as it is shown by the following example.

\begin{example}[\cite{aleg-fer-raf08}]\label{ex.precomp} %
{\rm There exists a set that is both $p$- and $\bar p$-precompact, but it is not $p^s$-precompact. } %
\end{example} %

Consider the space $\ell^\infty$ with the asymmetric norm $p(x)=\sup_ix_i^+,\, x=(x_i)\in \ell^\infty.$ Then $p^s(x)=\sup_i|x_i|$ is the usual sup-norm on $\ell^\infty.$ Let $x^0=(1,1,\dots).$ Because for $x\in B_{p^s}, \, x_i-1\leq |x_i|-1\leq 0,$ it follows $p(x-x^0)=0, $ so that $x\in x^0+\epsi\, B_p$ for every $\epsi >0,$  showing that $B_{p^s}$ is $p$-precompact. The relations $B_{p^s}=-B_{p^s}\subset -x^0+\epsi(-B_p)=-x^0+\epsi\,B_{\bar p}$ show that $B_{p^s}$ is also $\bar p$-precompact. Since any normed space with precompact unit ball is finite dimensional, it follows that $B_{p^s}$ is nor $p^s$-precompact.

Notice also that in quasi-metric spaces compactness, countable
compactness  and sequential compactness are different notions (see
\cite{fer-greg83},  \cite{kunz83} and \cite{salb-romag90}).

In spite of these peculiarities there are some positive results
concerning Baire theorem and compactness.

We mention the following results concerning compactness. %
\begin{theo}\label{t.comp-precomp}\hfill %
\begin{itemize} %
\item[{\rm 1.}] A sequentially compact quasi-semimetric space is left $K$-sequentially complete and precompact.
\item[{\rm 2.}] A precompact and left $\rho$-sequentially complete quasi-semimetric space is sequentially compact.
\item[{\rm 3.}] A quasi-semimetric space is compact if and only if it is precompact and left $K$-sequentially complete.
\item[{\rm 4.}] A  countably compact quasi-metrizable space is compact, but there exists countably compact quasi-semimetrizable spaces that are not compact.
\item[{\rm 5.}] A precompact  countably compact quasi-semimetric space is compact.
\item[{\rm 6.}] A precompact  sequentially compact quasi-semimetric space is compact.
\end{itemize}%
\end{theo}    %

These results as well as some related ones  on the relations between
completeness, compactness, precompactness, total boundedness  and
other related notions in quasi-metric and quasi-uniform spaces can be found in the papers
 \cite{alem-romag97,greg-romag00,kunz92,kunz-reily93,kunz-romag96,perez-romag96,perez-romag99,
 reily-subram82,romag92,romag-gutier86,salb-romag90}.

 Fixed point theorems in quasi-metric spaces were proved by Hicks, Huffman and Carlson \cite{hiks71,hiks88,hiks80}, Romaguera \cite{romag93} and Romaguera and Checa \cite{romag-checa90}. Chen et al \cite{chen06b,chen07} proved fixed point  theorems using a slightly different notion of convergence (see \cite{chen05}). In \cite{chen06a}
 some optimization problems in quasi-metric and in asymmetric normed spaces are discussed.

\subsection{Completeness in quasi-uniform spaces}

\par The considered completeness notions can be extended to
quasi-uniform spaces by replacing sequences by filters or nets. Let $(X,\rondu)$ be a
quasi-uniform space, $\rondu^{-1}=\{U^{-1} : U\in \rondu\}$ the
conjugate quasi-uniformity on $X,$ and $\rondu^s=\rondu \vee
\rondu^{-1}$ the coarsest uniformity finer than $\rondu$ and
$\rondu^{-1}.$ The quasi-uniform space $(X,\rondu)$ is called {\it
bicomplete} if $(X,\rondu^s)$ is a complete uniform space. This
notion is useful and easy to handle, because one can appeal to
well known results  from the theory of uniform spaces.

Another notion of completeness is that considered  by Sieber and
Pervin \cite{sieb-perv65}. A filter $\mathcal F$ in a
quasi-uniform space $(X,\rondu)$ is called $\rondu$-{\it Cauchy}
if for every $U\in \rondu$ there exists $x\in X$ such that
$U(x)\in \mathcal F.$ In terms of nets, a net $(x_\alpha,\alpha\in
D)$  is called $\rondu$-{\it Cauchy} if for every $U\in \rondu$
there exists $x\in X$ and $\alpha_0\in D$ such that
$(x,x_\alpha)\in U$ for all $\alpha\geq \alpha_0.$ The
quasi-uniform space $(X,\rondu)$ is called $\rondu$-{\it complete}
if every $\rondu$-Cauchy filter (equivalently, every
$\rondu$-Cauchy net) has a cluster point. If every such filter
(net) is convergent, then the quasi-uniform space $(X,\rondu)$ is
called $\rondu$-{\it convergence complete}. Obviously that
convergence complete implies complete, but the converse is not
true. It is clear that this notion corresponds to that of
$\rho$-completeness of a quasi-metric space. It is worth to notify
that the $\rondu_\rho$-completeness of the associated
quasi-uniform space $(X,\rondu_\rho)$ implies the
$\rho$-sequential completeness of the quasi-metric space
$(X,\rho),$ but the converse is not true (see \cite{kunz-reily93}).
The equivalence holds for the notion of left $K$-completeness
(which  will be defined immediately): a quasi-metric space is left
$K$-sequentially complete if and only if its induced
quasi-uniformity $\rondu_\rho$ is left $K$-complete
(\cite{romag92}).

A filter $\mathcal F$ in a
 quasi-uniform space $(X,\rondu)$ is
called {\it left $K$-Cauchy} provided for every $U\in \rondu$
there exists $F\in \mathcal F$ such that $U(x)\in F$ for all $x\in
F$. A net $(x_\alpha,\alpha\in D)$ in $X$ is called {\it left
$K$-Cauchy} provided for every $U\in \rondu$ there exists
$\alpha_0\in D$ such that $(x_\alpha,x_\beta)\in U$ for all
$\beta\geq \alpha\geq \alpha_0.$ The quasi-uniform space
$(X,\rondu)$ is called {\it left $ K$-complete} if every left
$K$-Cauchy filter (equivalently, every left $K$-Cauchy net)
converges.  If every left $K$-Cauchy filter converges with respect
to the uniformity $\rondu^s,$  then the quasi-uniform space
$(X,\rondu)$ is called {\it Smyth complete} (see \cite{kunz95}
and \cite{smyth94}). This notion of completeness has applications
to computer science, see \cite{shelek95}. In fact, there are a lot
of applications of quasi-metric spaces, asymmetric normed spaces
and quasi-uniform spaces to computer science, abstract languages,
the analysis of the complexity of programs, see, for instance,
\cite{rafi-rom-per02b,rafi-rom-per04,shelek04,romag-shelek99,romag-shelek00a,romag-shelek00b,romag-shelek02}.

Another useful notion of completeness was considered  by
Doitchinov \cite{doichin88a,doichin91b,doichin92}.  A
filter $\mathcal F$ in a quasi-uniform space $(X,\rondu)$ is
called $D$-{\it Cauchy} provided there exists a so called co-filter
$\mathcal G$ in $X$ such that for every $U\in \rondu$ there are
$G\in \mathcal G$ and $F\in \mathcal F$ with $F\times
G\subset U. $ The quasi-uniform space $(X,\rondu)$ is called
$D$-{\it complete} provided every $D$-Cauchy filter converges. A
related notion of completeness was  considered by Andrikopoulos
\cite{andrik04a}. For a comparative study of the completeness
notions defined by  filters and nets see Andrikopoulos \cite{andrik04b}, De\'{a}k \cite{deak91,deak95,deak96} and  S\"{u}nderhauf \cite{sunderh95,sunderh97}.
   Doitchinov  \cite{doichin88b,doichin88c,doichin91a} defined a  similar notion of completeness for sequences in
quasi-metric spaces, which will be discussed in Theorem \ref{t.sL-compl2} from Section \ref{Semi-lip} on semi-Lipschitz functions.

Two difficult problems in the study of quasi-metric and quasi-uniform spaces are those of completion and compactification for such spaces. As it is mentioned in the paper \cite{reily-subram82}, the authors did not succeed to obtain a satisfactory theory for these problems. Some progress in this direction was obtained by Alemany and Romaguera \cite{alem-romag96}, Doichinov \cite{doichin88b, doichin92}, Gregori and Romaguera \cite{greg-romag00}, Romaguera and S\'{a}nchez-Granero \cite{romag-gran03}. We shall discuss in Section \ref{Lin-ops} the existence of a bicompletion for an asymmetric normed space, following the paper \cite{rafi-rom-per02a}. The existence of a bicompletion of a normed cone was proved by Oltra and Valero \cite{oltra-valero04b}.

Notice also that these notions of completeness can be considered
within the framework of bitopological spaces in the sense of Kelly
\cite{kelly63}, since a quasi-metric space is a bitopological
space with respect to the topologies $\tau(\rho)$ and $\tau(\bar
\rho).$ For this approach see the papers by De\' ak
\cite{deak95,deak96}. It seems that  $K$ in the definition of left
$K$-completeness comes from Kelly who considered first this notion
(see \cite{zimmer07}).

The notions of total boundedness and precompactness can be also
extended to quasi-uniform spaces.

The study of some bitopological and quasi-uniform structures on concrete spaces of semi-continuous or continuous functions was done in the papers \cite{fer-greg-reily95} and \cite{kunz-romag96}, respectively.

\subsection{Baire category}

As it is well known, Baire category theorem plays a fundamental role in the proofs of
two fundamental principles of functional analysis: the uniform boundedness principle
(called also the Banach-Steinhaus theorem) and the open mapping theorem. With the aim to
check for possible extensions of these principles to the asymmetric case, we shall present some
Baire category results in quasi-metric spaces. The first one was proved by Kelly \cite{kelly63} (see also \cite{reily-subram82}).
The notion of pairwise Baire bitopological space is considered in the paper \cite{aleg-fer-greg99a}. %
\begin{theo}\label{t.Baire1}  %
Let $(X,\rho)$ be a quasi-semimetric space. If $X$ is right $\bar\rho$-sequentially complete, then $(X,\tau_\rho)$
is of second category in itself. %
\end{theo} %

Other Baire type results were proved by Gregori and Ferrer \cite{fer-greg85}. A topological
space $(T,\tau)$ is called {\it quasi-regular} if every nonempty open subset of $T$ contains
the closure of some nonempty open set. %
\begin{theo}\label{t.Baire2} %
If the quasi-semimetric  space  $(X,\rho)$ is quasi-regular and left $\rho$-sequentially complete,
then it is a Baire space. %
\end{theo} %

Taking into account the implications from Remark \ref{rem.completeness}, it follows %
\begin{corol}\label{c.Baire3} %
A quasi-regular quasi-semimetric space is a Baire  space in any of the following cases %
\begin{itemize} %
\item[{\rm (a)}] \;$(X,\,\rho)$ is left $\rho$-sequentially complete;
\item[{\rm (b)}] \;$(X,\,\rho)$ is weakly left $K$-sequentially complete;
\item[{\rm (c)}] \;$(X,\,\rho)$ is left $K$-sequentially complete. %
\end{itemize} %
\end{corol} %

\section{Continuous linear operators between asymmetric normed spaces}\label{Lin-ops}

\subsection{The asymmetric norm of a continuous linear operator}

Let $(X,p)$ and $(Y,q)$ be two asymmetric normed spaces. Denote by $L_a(X,Y)$ the space of all linear operators
from $X$ to $Y$.  A linear operator $A:X\to Y$ is called $(p,q)$-continuous if it is continuous with respect to the topologies $\tau_p$ on $X$ and $\tau_q$ on $Y.$ The set of all $(p,q)$-continuous linear operators from $X$ to $Y$ is denoted by $L_{p,q}(X,Y).$
For $\mu\in \{p,\bar p,p^s\}$ and $\nu\in \{q,\bar q,q^s\},$ the $(\mu,\nu)$-continuity and the set $L_{\mu,\nu}(X,Y)$ are defined similarly.  The space of all continuous linear operators between  the associated  normed spaces $(X,p^s)$ and $(Y,q^s)$ is denoted by $L(X,Y).$

In the case of linear functionals, i.e., when $Y=(\Real,u),$ we put $X^\flat_p=L_{p,u}(X,\Real)$ and
$X^*=L((X,p^s),(\Real,|\cdot|))$. The meaning of  $X^\flat_{\bar p}$ is clear.

A linear operator $A:(X,p)\to (Y,q)$    is called $(p,q)$-{\it semi-Lipschitz} (or $(p,q)$-{\it bounded}) if there exists a number $\beta\geq 0$ such that %
\bequ\label{def.Lip} %
q(Ax)\leq \beta p(x), %
\eequ %
for all $x\in X.$

The following proposition contains characterization of continuity, similar to those known
in the symmetric case. %
\begin{prop}\label{p.cont-lin1} %
For a linear operator $A$ between two asymmetric normed spaces $(X,p),(Y,q)$,
the following are equivalent. %
\begin{itemize} %
\item[{\rm 1.}] The operator $A$ is continuous on $X.$
\item[{\rm 2.}] The operator $A$ is continuous at $0\in X$ (or at an arbitrary point $x_0\in X$).
\item[{\rm 3.}] The operator $A$ is $(p,q)$-semi-Lipschitz. %
\item[{\rm 4.}] The operator $A$ is $(\rondu_p,\rondu_q)$-quasi-uniformly continuous on $X.$
\end{itemize} %
\end{prop}%
\begin{proof} %
The equivalence 1 $\iff$ 2 holds for any additive operator (see Proposition \ref{p.top-asym-n1}).

2 $\Rightarrow$ 3 \; By the continuity of $A$ at $0\in X$ there exists $r>0$ such that $A(B_p[0,r])\subset B_q[0,1],$  meaning that %
\bequ\label{eq1.cont-lin1}  %
q(Ax)\leq 1, %
\eequ %
for every $x\in X$ with $p(x)\leq r. $ If $p(x)>0, $ then $p(\frac{r}{p(x)}x)=r,$ so that $\frac{r}{p(x)}q(Ax)=q(A(\frac{r}{p(x)}x))\leq 1, $ or, equivalently, $q(Ax)\leq r^{-1}p(x). $  Consequently,
\eqref{def.Lip} holds, with $\beta=r^{-1},$ for every $x\in X$ with $p(x)>0.$
 If $p(x)=0, $ then $p(nx)=np(x)=0$ for all $n\in \Nat,$ so that, by \eqref{eq1.cont-lin1}, $nq(Ax)=q(A(nx))\leq 1, $ for all $n\in \Nat, $ implying $q(Ax)=0,$  so that \eqref{def.Lip} holds with $\beta=r^{-1}$
for all $x \in X.$

3 $\Rightarrow$ 4\; Suppose that \eqref{def.Lip} holds with some $\beta >0$ and let $\epsi >0.$ For  $V=\{(y_1,y_2)\in Y\times Y : q(y_2-y_1)<\epsi\}\in \rondu_q,$  let $U:=\{(x_1,x_2)\in X\times X : p(x_2-x_1)<\epsi/\beta\}\in \rondu_p.$ Then for every $(x_1,x_2)\in U,\, q(Ax_2-Ax_1)=q(A(x_2-x_1))\leq\beta q(x_2-x_1)<\epsi,$ showing that $(Ax_1,Ax_2)\in V.$

The implication 4 $\Rightarrow  $ 1 is a general result: any quasi-uniformly continuous mapping between two quasi-uniform spaces is continuous with respect to the topologies induced by the quasi-uniformities. %
\end{proof}

In the case of linear functionals we have the following characterization.  Note that in this case the fact that
a linear functional $\vphi:(X,p)\to \Real$ is $(p,u)$-semi-Lipschitz is equivalent to %
\bequ\label{def.Lip2} %
\vphi(x)\leq \beta p(x), %
\eequ %
for all $x\in X.$

\begin{prop}\label{p.cont-lin2} %
Let $(X,p)$ be an asymmetric normed space and $\vphi:X\to \Real$ a linear functional. The following are equivalent. %
\begin{itemize} %
\item[{\rm 1.}] The functional $\vphi$ is continuous from $(X,p)$ to $(\Real,u).$
\item[{\rm 2.}] The functional $\vphi$ is continuous at $0\in X$ (or at an arbitrary point $x_0\in X$).
\item[{\rm 3.}] The functional $\vphi$ is $(p,u)$-semi-Lipschitz. %
\item[{\rm 4.}] The functional $\vphi$ is $(\rondu_p,\rondu_u)$-quasi-uniformly continuous on $X.$
\item[{\rm 5.}] The functional $\vphi$ is upper semicontinuous from $(X,p)$ to $(\Real,|\cdot|).$
\end{itemize} %
\end{prop} %
\begin{proof} %
We have to prove only the equivalence 1 $\iff$ 5. Suppose that 1 holds and let $x_0\in X$ and $\epsi > 0.$ Then $V=(-\infty;p(x_0)+\epsi)$ is a $\tau_u$-neighborhood of $p(x_0),$ so that there exists a $\tau_p$-neighborhood $U$ of $x_0$ such that
$p(x)\in V$ for all $x\in U,$ which proves the upper semicontinuity of $\vphi $ at $x_0.$  The converse also follows in this way. %
\end{proof}

The following proposition shows that continuous linear operators between asymmetric normed spaces are continuous with respect to the associated norm topologies. %
\begin{prop}\label{p.cont-lin3} %
Let $(X,p)$ and $(Y,q)$ be asymmetric normed spaces. Any $(p,q)$-continuous linear operator  $A:X\to Y$ is also
$(p^s,q^s)$-continuous and the set $L_{p,q}(X,Y)$ is a convex cone in $L(X,Y).$

In particular, every $(p,u)$-continuous linear functional is $(p^s,|\cdot|)$-continuous and $X^\flat_p$ is a cone in the space  $X^*  $ of all continuous linear functionals on the normed space $(X,p^s).$ %
\end{prop} %
\begin{proof}%
If $A$ is $(p,q)$-continuous, then $A$ satisfies \eqref{def.Lip}, so that %
$$%
q(Ax)\leq \beta p(x) \leq \beta p^s(x), %
$$%
for all $x\in X.$ Replacing $x$ by $-x,$  one obtains %
$$%
q(-Ax)=q(A(-x))\leq \beta p^s(-x) = \beta p^s(x), %
$$%
implying %
$$%
q^s(Ax)=\max\{q(Ax),q(A(-x))\}\leq \beta p^s(x), %
$$%
for all $x\in X$ which is equivalent to the $(p^s,q^s)$-continuity of $A.$

It is easy to check that $A,B\in L_{p,q}(X,Y)$ and $\lambda \geq 0$ implies $A+B,\lambda A\in L_{p,q}(X,Y),$
 so that $L_{p,q}(X,Y)$ is a convex cone in $L(X,Y).$ %
\end{proof} %

The following example shows that $ L_{p,q}(X,Y)$ is not  a subspace of $L(X,Y)$. %
\begin{example}\label{ex.borodin2} %
{\rm On the space $X=C_0[0;1]$  from Example \ref{ex.borodin} consider the functional $\vphi(f)=f(1),\, f\in C_0[0;1].$  Then $\vphi$ is continuous  on $(X,p),$ but $-\vphi$ is not continuous. }

{\rm Other example is furnished by the functional $\id:(\Real,u)\to \Real$ which is $(p,u)$-continuous, but $-\id$ is not.}
\end{example} %

Indeed,  $\vphi(f)=f(1)\leq \max\vphi([0;1])=p(f),\, f\in C_0[0;1],$ so that $\vphi$ is $(p,u)$-continuous. Taking $f_n(t)=1-nt,\, t\in [0;1],$  it follows $p(f_n)=1$ and $-\vphi(f_n)= -f_n(1)= n-1,$ so that $-\vphi$ is not bounded on the unit ball of $(X,p),$ so it is not continuous.

Based on this properties one can introduce an asymmetric norm on the cone $L_{p,q}(X,Y)$ by %
\bequ\label{def.norm-op1} %
\|A|_{p,q}=\sup\{q(Ax) : x\in X, \, p(x)\leq 1\}, %
\eequ %
for every $A\in L_{p,q}(X,Y).$ %

The norm $\|\cdot\|:=\|\cdot|_{p^s,q^s}$ is the usual operator norm on the space $L(X,Y)=L_{p^s,q^s}(X,Y)$ given for $A\in L(X,Y)$ by %
\bequ\label{def.norm-op2} %
\|A\|=\sup\{q^s(Ax) : x\in X,\, p^s(x)\leq 1\}. %
\eequ

Following \cite{rafi-rom-per03a}, we can consider an extended asymmetric norm on the space $L(X,Y)$ of all
linear continuous operators from $(X,p^s)$ to $(Y,q^s),$ defined by the same formula: %
\bequ\label{def.norm-op3} %
\|A|^*_{p,q}=\sup\{q(Ax) : x\in X, \, p(x)\leq 1\}=\sup\{q(Ax) : x\in B_p\}, %
\eequ %
for every $A\in L(X,Y).$  If $A\in L_{p,q}(X.Y),$ then $A\in L(X,Y),$
and so $-A\in L(X,Y),$ but, as  the above examples show, it is possible that $\|-A|^*_{p,q}=\infty,$ so that $\|\cdot|^*_{p,q}$ could be  effectively an extended asymmetric norm.

With the asymmetric norm $\|\cdot|^*_{p,q}$ one associates a (symmetric) extended norm on $L(X,Y)$ defined by %
\bequ\label{def.norm-op4}
\|A\|^*_{p,q}=\max\{\|A|^*_{p,q}\,,\|-A|^*_{p,q}\}. %
\eequ

Since %
$$%
\|-A|^*_{p,q} =\sup\{q(-Ax) : p(x)\leq 1\}=\sup\{q(Ax) : p(-x)\leq 1\}=\sup\{q(Ax) : x\in B_{\bar p}\}, %
$$%
it follows %
\bequ\label{eq1.norm-op} %
\|A\|^*_{p,q}=\sup\{q(Ax) : x\in B_p\cup B_{\bar p}\}. %
\eequ

The norm $\|A\|^*_{p,q}$ can be calculated also by the formula %
\bequ\label{eq1b.norm-op} %
\|A\|^*_{p,q}=\sup\{q^s(Ax) : x\in B_p\}. %
\eequ %

Indeed, denoting by $\lambda$ the right side member of the above equality, we have $q(Ax)\leq q^s(Ax)\leq \lambda$ for every $x\in B_p,$ so that $\|A|^*_{p,q}\leq \lambda.\,$ Similarly $q(-Ax)\leq q^s(Ax)\leq \lambda$ for every $x\in B_p$ implies $\|-A|^*_{p,q}\leq \lambda,$ so that $\|A\|^*_{p,q}\leq \lambda.$ Also,
$q(Ax)\leq \|A|^*_{p,q}\leq\|A\|^*_{p,q}$ and  $q(-Ax)\leq \|-A|^*_{p,q}\leq\|-A\|^*_{p,q}=\|A\|^*_{p,q}$    for every $x\in B_p,\,$ implies $\lambda\leq \|A\|^*_{p,q}.$

A {\it cone} (in fact, a {\it convex cone}) is a subset $Z$ of a linear space $X$ such that $w+z\in Z$ and $\lambda z\in Z$ for all $w,z\in Z$ and $\lambda \geq 0.$ A cone is called also a {\it  semilinear space}. In order to study spaces of linear operators between asymmetric  normed spaces and the duals of such spaces, we shall consider asymmetric norms   on cones.  %
\begin{remark}\label{rem.cones} %
{\rm In fact, one can define an abstract notion of cone as a set $K$ with two operation: addition + with respect to which $(K,+)$ is a monoid with null element 0, and  multiplication by nonnegative scalars satisfying the properties $(\lambda\mu)a=\lambda\,(\mu a),\, \lambda (a+b)=\lambda a+\lambda b, \,(\lambda+\mu)a =\lambda a+\mu a,\, 1a=1,$ and $0a=0.$ The cone  $K$ is cancellative (i.e, $a+c=b+c\Rightarrow a=b$) if and only if $K$ can be embedded in a vector space, as in our  case.  The theory of locally convex cones, with applications to Korovkin approximation theory and to vector-measure theory, is developed in the books by Keimel and Roth \cite{K-Roth92} and Roth \cite{Roth09}.}%
\end{remark}

We mention the following results, whose proofs are similar to that for normed spaces.

\begin{prop}\label{p.norm-op1} %
Let $(X,p)$ and $(Y,q)$ be asymmetric normed spaces and $A\in L_{p,q}(X,Y)$.
Then the number $\|A|_{p,q}$ is the smallest semi-Lipschitz constant for $A$
 and $\asnorm$ is an asymmetric norm on the cone $L_{p,q}(X,Y).$

 The asymmetric norm $\|A|_{p,q}$ can be calculated also by the formula %
 \bequ\label{eq1.norm-op1} %
 \|A|_{p,q}=\sup\{q(Ax)/p(x) : x\in X,\, p(x)>0\}. %
 \eequ %
 \end{prop} %

 \subsection{The normed cone  of continuous linear  operators - completeness}

 Recall that an asymmetric normed space $(X,p)$ is called biBanach if the associated normed space $(X,p^s)$ is a Banach space (i.e., a complete normed space). %
\begin{prop}[\cite{rafi-rom-per03a}]\label{p.norm-op2} %
Let $(X,p)$ and $(Y,q)$ be asymmetric normed spaces. %
\begin{itemize} %

\item[{\rm 1.}] The functional   $\asnorm^*$  is an extended asymmetric norm on the space $L(X,Y)$ and $\|A\|\leq \|A\|^*_{p,q}$ for all $A\in L(X,Y).$  An operator $A\in L(X,Y)$ belongs to $L_{p,q}(X,Y)$ if and only if $\|A|^*_{p,q}<\infty.$ Also $\|A|^*_{p,q}=\|A|_{p,q}$ for $A\in L_{p,q}(X,Y).$
\item[{\rm 2.}]  If  the space $(Y,q)$ is biBanach, then the space  $(L(X,Y),\|\cdot\|^*)$ is complete.
\item[{\rm 3.}] The set $L_{p,q}(X,Y)$ is closed in $(L(X,Y),\|\cdot\|_{p,q}^*), $ so it is complete  with respect to the restriction of the extended norm $\|\cdot\|_{p,q}^*$  to $L_{p,q}(X,Y). $
\end{itemize} %
\end{prop} %
\begin{proof}
We shall omit the subscripts $p,q $ in what follows.

1. It is easy to check that $\|\cdot|^*$ is an extended asymmetric norm on  $L(X,Y).$

We can suppose $\|A\|^*<\infty.$ Then %
\begin{align*}%
q(Ax)\leq& \|A|p(x)\leq \|A\|^*p^s(x),\quad\mbox{and} \\
q(-Ax)\leq& \|-A|p(x)\leq \|A\|^*p^s(x), %
\end{align*} %
 so that $q^s(Ax)\leq \|A\|^*p^s(x), $ for all $x\in X,$ implying  $\|A\|\leq \|A\|^*.  $

 If $A\in L_{p,q}(X,Y), $ then $\, q(Ax)\leq \|A|p(x)\leq \|A|p^s(x),\, x\in X,\,$ implies $\|A\|\leq \|A|.$

2.   Let $(A_n)$ be a $\|\cdot\|^*$-Cauchy sequence in $L(X,Y),$  that is for every $\epsi>0$ there exists
$n_0\in \Nat$ such that %
\bequ\label{eq2.norm-op2} %
\|A_m-A_n\|^*\leq\epsi%
\eequ %
holds for all $m,n\geq n_0.$
By the first point of the theorem,
$\|A_m-A_n\|\leq \|A_m-A_n\|^*,$ so that $(A_n)$ is a Cauchy sequence in the  Banach space $(L(X,Y),\|\cdot\|),$
and so $(A_n)$ has a $\|\cdot\|$-limit  $A\in L(X,Y).$

It remains to show that $(A_n)$ converges to $A$ with respect to the norm $\norm^*.$  The inequality $q^s(A_nx-Ax)\leq \|A_n-A\|p^s(x)$ implies %
\bequ\label{eq3.norm-op2} %
\lim_{n\to\infty}q^s(A_nx-Ax)=0,%
\eequ %
that is the sequence $(A_nx)$ converges to $Ax$ in the normed space $(Y,q^s),$ for every $x\in X.$

For $\epsi >0$ let $n_0$ be such that \eqref{eq2.norm-op2} holds.  Then, by \eqref{eq1b.norm-op}, for every
$x\in B_p$%
$$%
q^s(A_mx-A_nx)\leq \|A_m-A_n\|^*\leq\epsi, %
$$%

By \eqref{eq3.norm-op2}, the inequality $\, q^s(A_mx-A_nx)\leq\epsi\,$ yields for  $n\to \infty$  %
$$%
q^s(A_mx-Ax)\leq\epsi, %
$$%
for every $x\in B_p$ and $m\geq n_0.$
   Taking into account \eqref{eq1b.norm-op}, it follows
 $\|A_m-A\|^*\leq\epsi$ for every $m\geq n_0.$

3. To show that $L_{p,q}(X,Y)$ is closed in $(L(X,Y),\|\cdot\|^*)$ let $(A_n)$  be a sequence in
$L_{p,q}(X,Y)$ which is $\norm^*$-convergent to some $A\in L(X,Y).$ For $\epsi =1$ let
$n_0\in \Nat$ be such that $\|A_{n}-A\|^*\leq 1$ for all $n\geq n_0.$ Then $\|A\|^*\leq\|A-A_{n_0}\|^*+\|A_{n_0}\|^*\leq 1+\|A_{n_0}\|^*,$ which implies that both  $A $ and $-A$ belong to $L_{p,q}(X,Y).$ %
 \end{proof} %
 \begin{remarks} {\rm 1. If a sequence $(A_n)$  in $L_{p,q}(X,Y)$ converges to $A\in L(X,Y)$ with respect to the conjugate norm $\,(\|\cdot|^*_{p,q})^{-1}$  of $\,\|\cdot|^*_{p,q},$ then } $A\in L_{p,q}(X,Y).$   %

 {\rm 2. As it is known, in the classical case, the completeness of $(L(X,Y),\norm)$
 implies the completeness of  the normed space $Y.$ We do not know if a similar
 result holds for the extended norm} $\norm^*.$  %
 \end{remarks} %

 Let $n_0$ be such that $\|A-A_n|^*\leq 1$ for all $n\geq n_0.$ Then $\|A|^*\leq \|A-A_{n_0}|^*+\|A_{n_0}|^*\leq 1+\|A_{n_0}|^*<\infty,$ shows that
 $A\in L_{p,q}(X,Y)$, proving the validity of the assertion from 1.

 \subsection {The bicompletion of an asymmetric normed space}

 As it is well known, any normed space $(X,\norm)$ has a completion, meaning that there exists a Banach
 space $(\tilde{X},\norm\tilde{\;})$ such that $(X,\norm)$ is isometrically isomorphic to a dense subspace
 of   $(\tilde{X},\norm\tilde{\;}).$  The Banach space  $(\tilde{X},\norm\tilde{\;})$, called the {\it completion}
  of $X$ is uniquely determined, in the sense that any Banach space $Z$ such that
  $X$ is isometrically isomorphic to a dense subspace of $Z$ is isometrically
   isomorphic to    $(\tilde{X},\norm\tilde{\;}).$

A {\it bicompletion}  of an asymmetric normed space $(X,p) $ is a
bicomplete  asymmetric normed space $(Y,q)$ such that $X$ is
isometrically isomorphic to a dense subspace of $Y.$  An {\it
isometry} between two asymmetric normed  spaces $(X,p),(Y,q)$ is a mapping  $T:X\to
Y$ such that
\bequ\label{def.isometry} %
q(Tx-Ty)=p(x-y),\, \mbox{ for all}\quad  x,y \in X. %
\eequ

If $T$ is linear, then \eqref{def.isometry} is equivalent to %
\bequ\label{def.isometry2} %
q(Tx)=p(x),\, \mbox{ for all}\quad x \in X. %
\eequ

Note that, as defined, the isometry $T$ is in fact an isometry between
the associated normed  spaces $(X,p^s),(Y,q^s),$ because %
$$%
q^s(Tx-Ty)=q(Tx-Ty)\lor q(Ty-Tx)=p(x-y)\lor p(y-x)=p^s(x-y), %
$$%
for all $\,x,y\in X.$

The construction of the bicompletion of an asymmetric normed space was done in \cite{rafi-rom-per02a} (see  \cite{oltra-valero04b} for the case of normed cones), following the ideas from the normed case, adapted to
the asymmetric one. We only sketch the construction, referring for  details to the mentioned papers.

Let $(X,p)$ be an asymmetric normed space. In the set of all
$p^s$-Cauchy sequences in $X$ define an equivalence
relation by %
\bequ\label{def.equiv-Cauchy} %
(x_n)\sim (y_n) \iff \lim_{n\to \infty}p^s(x_n-y_n)=0. %
\eequ %

\begin{lemma}\label{le.bicompl1} \hfill%
\begin{itemize} %
\item[{\rm 1.}] The relation $\sim$ is an equivalence relation on the
set of all $p^s$-Cauchy sequences in $X.$
\item[{\rm 2.}] For every $p^s$-Cauchy sequence there exists the limit
$\,\lim_n p(x_n)$ and  $\,\lim_n p(x_n) = \lim_n p(y_n)  $ if $(x_n)$
and
$(y_n)$ are equivalent $p^s$-Cauchy sequences.%
    \end{itemize}
\end{lemma} %
\begin{proof} %
The verification of 1 is routine.

2.  The inequalities $p(x_n)-p(x_m)\leq p(x_n-x_m)\leq p^s(x_n-x_m)$
valid for all $m,n\in \Nat,$ and the fact that  $(x_n)$  is
$p^s$-Cauchy, imply that $(p(x_n))$ is a Cauchy sequence in $\Real,
$ so it converges to some $\alpha\in \Real.$ If $(y_n)$ is another
$p^s$-Cauchy sequence equivalent to $(x_n)$, then the inequalities
\begin{align*} %
p(x_n)-p(y_n)\leq& p(x_n-y_n)\leq p^s(x_n-y_n),\\ p(y_n)-p(x_n)\leq&
p(y_n-x_n)\leq p^s(y_n-x_n)=p^s(x_n-y_n)\end{align*}
and the condition $\lim_np^s(x_n-y_n)=0$ imply $\lim_np(x_n)=\lim_np(y_n). $ %
\end{proof} %

Denote by $\tilde{X}$ the linear space of all equivalence classes of
$p^s$-Cauchy sequences with addition and multiplication by scalars
defined, as usual, by $[(x_n)]+[(y_n)=[(x_n+y_n)] $ and $\lambda
[(x_n)]=[(\lambda x_n)],$ where $[(x_n)]$ denotes the equivalence
class containing the $p^s$-Cauchy sequence $(x_n)$. Based on Lemma
\ref{le.bicompl1}, one can define on the  space  $\tilde{X}$ an
asymmetric norm $\tilde p$ by  %
\bequ\label{def.hat-p}
\tilde p([(x_n)])=\lim_np(x_n),%
\eequ
for any $p^s$-Cauchy sequence $(x_n)$ in $X.  $

As it known, equipped with the norm  %
\bequ\label{def.hat-ps}%
\widetilde{p^s}([(x_n)])=\lim_np^s(x_n),
\eequ %
 the space $(\tilde{X},\widetilde{p^s}) $ is a Banach space and the mapping $i:X\to\tilde X$  defined  by
 \bequ\label{def.embed} %
 i(x)=[(x_n)]\quad\mbox{ where}\quad x_n=x,\, \forall n\in \Nat,
 \eequ    %
 is a linear isometry of $X$ into $(\tilde X,\widetilde{p^s})$ such that $i(X)$ is $\widetilde{p^s}$-dense in $\tilde X.$
  Hence, the fact that $(X,p) $
is biBanach, meaning  that  $(\tilde{X},(\tilde p)^s) $ is Banach,
will follow once we prove the  equality  %
\bequ\label{eq1.bicompl1} %
(\tilde p)^s=\widetilde{p^s}, %
\eequ %

For  a $p^s$-Cauchy sequence  $(x_n)$ in $X,$ let $\alpha_n=p(x_n)$
and $\beta_n=p(-x_n), \, n\in \Nat.$ Then %
$$%
 \widetilde{p^s}([x_n])=\lim_np^s(x_n)=\lim_n(\alpha_n\lor \beta_n)=(\lim_n\alpha_n)\lor(\lim_n\beta_n)=
 \tilde p([(x_n)])\lor \tilde p(-[(x_n)]) =(\tilde p)^s([(x_n)]). %
 $$%

\begin{remark}\label{rem.bicompl1} %
{\rm The fact that $\alpha_n\to \alpha$ and $\beta_n\to \beta$ implies $\alpha_n\lor \beta_n\to\alpha\lor\beta, $ follows from the relations} %
$$%
\alpha_n\lor\beta_n=\frac{\alpha_n+\beta_n+|\alpha_n-\beta_n|}{2}\to \frac{\alpha+\beta+|\alpha-\beta|}{2}=\alpha\lor\beta. %
$$%
\end{remark}

  We summarize the results in the following theorem.%
 \begin{theo}\label{t.bicompl2} %
 Let $(X,p)$ be an asymmetric normed space, $\,\tilde X$ the space constructed above and $\,\tilde p, \,\widetilde{p^s}\,$
 the norms on $\tilde X$ given by \eqref{def.hat-p} and \eqref{def.hat-ps}, respectively. %
 \begin{itemize} %
 \item[{\rm 1.}] The space $(\tilde X,\tilde p)$ is  biBanach, or, equivalently,  $(\tilde X,(\tilde p)^s)$ is a Banach  space. %
 \item[{\rm 2.}]  The mapping $i:X\to \tilde X,$ defined by \eqref{def.embed}, is a linear isometry of $(X,p)$ into $(\tilde X,\tilde p)$ and the space $i(X)$ is $\tilde p$-dense in $\tilde X.$
 \item[{\rm 3.}]  If $(Y,q)$ is an symmetric biBanach space such that $(X,p)$ is isometrically isomorphic to
 a $q$-dense subspace of $Y,$  then $(Y,q)$ is isometrically isomorphic to $(\tilde X,\tilde p).$ %
 \end{itemize} %
 \end{theo} %

 \subsection {Open mapping and closed graph theorems for asymmetric normed spaces}

 As it is known, the proofs of two fundamental principles of functional analysis - the Open Mapping Theorem and the Closed Graph Theorem for Banach spaces -  rely on Baire's category theorem. Based on Theorem \ref{t.Baire1}, C. Alegre \cite{alegre09} extended these principles to  asymmetric normed spaces. %

 \begin{theo}[The Open Mapping Theorem, \cite{alegre09}]\label{t.open-map-th1} %
 Let $(X,p)$ and $(Y,q)$ be two asymmetric normed spaces. If $(X,p)$ is right $K$-sequentially complete and  $(Y,q)$ is right $K$-sequentially complete and Hausdorff, then any surjective continuous linear mapping $A:(X,\tau_p)\to (Y,\tau_q)$    is an open mapping. %
 \end{theo} %

 A consequence of this deep result is the inverse mapping theorem which, in essence, is an equivalent  form of the open mapping theorem. %
 \begin{theo}\label{t.inv-map} %
 Let $(X,p)$ and $(Y,q)$ be two asymmetric normed spaces. If $(X,p)$ is right $K$-sequentially complete and  $(Y,q)$ is right $K$-sequentially complete and Hausdorff, then the inverse of any bijective continuous linear mapping $A:(X,\tau_p)\to (Y,\tau_q)$    is continuous. %
 \end{theo} %

 For two asymmetric normed spaces $(X,p)$ and $(Y,q)$ consider $X\times Y$ endowed with the asymmetric norm
 \bequ\label{def.prod-norm} %
r(x,y)=p(x)+q(y),\, (x,y)\in X\times Y.        %
 \eequ %
 The proof of the results from the following lemma are similar to those in the symmetric case.%
\begin{lemma}\label{le.prod-top} %
 Let $(X,p),\, (Y,q)$ be asymmetric normed spaces. %
 \begin{itemize} %
 \item[{\rm 1.}] The norm $r$ defined by \eqref{def.prod-norm} generates the product topology $\tau_p\times\tau_q$ on $X\times Y.$
 \item[{\rm 2.}]   If $(X,p)$ and $(Y,q)$ a right(left) $K$-sequentially complete, then  $(X\times Y,r)$ is
 right(left) $K$-sequentially complete. %
  \item[{\rm 3.}] A closed subset of a right(left) $K$-sequentially complete normed space is  right(left) $K$-sequentially complete.
      \end{itemize}
 \end{lemma} %

 As in the case of Banach spaces, the closed graph theorem can easily derived from the open mapping theorem.
 The graph $G_f$ of a mapping $f:X\to Y$ is the subset of $X\times Y$ given by  $G_f=\{(x,y)\in X\times Y : y=f(x)\}.$  %
 \begin{prop}\label{p.closed-graph} %
 If $(X,\tau),\,(Y,\sigma)$  are topological spaces and $f:X\to Y$ is continuous, then the graph $G_f$ of $f$ is closed in $X\times Y$ with respect to the product topology $\tau\times \sigma.$ %
 \end{prop} %
 \begin{proof} %
 Let $(x_i,y_i),\, i\in I,$ a net in $ X\times Y$ converging  to some $(x,y)\in X\times Y$ with respect to the product topology. This is equivalent to $x_i\to x$ in $X$ and $y_i\to y$ in $Y$. Since $y_i=f(x_i),\,i\in I,$ the continuity of $f$ implies $ f(x)=y,$  that is $(x,y)\in G_f.$ %
 \end{proof} %

 If $X$ and $Y$ are topological spaces, then a mapping $f:X\to Y$ is called {\it with closed graph} provided $G_f$ is closed in $X\times Y$ with respect to the product topology.

 \begin{theo}[The Closed Graph Theorem, \cite{alegre09}]\label{t.closed-graph} %
 Let $(X,p)$  and $(Y,q)$ be asymmetric normed spaces. If  $(X,p)$ is right $K$-sequentially complete and Hausdorff and
$(Y,q)$ is  right $K$-sequentially complete, then any linear mapping $A:X\to Y$ with closed graph is continuous. %
\end{theo} %
\begin{proof} %
By Lemma \ref{le.prod-top}, the   product topology $\tau_p\times\tau_q$ is generated by the asymmetric norm
$r(x,y)=p(x)+q(y)$ and the space $(X\times Y,r)$ is right $K$-sequentially complete.

By hypothesis, the graph $G_A$   of $A$ is a closed subspace of $(X\times Y,r)$ so it is also right $K$-sequentially complete with respect to $r.$ The projections $P_1,P_2:G_A\to X$ defined by $P_1(x,y)=x$ and $P_2(x,y)\,$,\, for $(x,y)\in G_A,\,$ are linear and continuous mappings. The projection $P_1$ is also
 bijective, so that, by the inverse mapping theorem (Theorem \ref{t.inv-map}),
$P_1^{-1}:X\to G_A$ is also continuous. Observe that $P_1^{-1}$ is given by $P^{-1}_1(x)=(x,Ax),\, x\in X, $ so that $A=P_2\circ P^{-1}_1$ will be continuous. %
\end{proof} %

\subsection {Normed cones}

As we did mention in Remark \ref{rem.cones},  the study of the duals of asymmetric normed spaces requires the consideration of normed cones. We shall briefly present some results.

Recall that a {\it cone} is a set $K$ equipped with   two operations: addition + with respect to which $(X,+)$ is a commutative  monoid (that is, the operation + is associative and commutative) with null element 0, and  multiplication by nonnegative scalars satisfying the properties %
\bequ\label{def.cone}%
\begin{aligned}%
&\mbox{(i)}\; (\lambda\mu)a=(\lambda)(\mu a),\quad \mbox{(ii)}\; \lambda (a+b)=\lambda a+\lambda b, \quad \mbox{(iii)} (\lambda+\mu)a =\lambda a+\mu a,\\ &\mbox{(iv)} \; 1\cdot a=1,\quad\mbox{ and} \quad \mbox{(v)}\; 0\cdot a=0. %
\end{aligned} %
 \eequ %

 The cone  $X$ is called {\it cancellative} if  $a+c=b+c\Rightarrow a=b$ for all $a,b,c\in X.$ A cone  $X$ is cancellative if and only if it can be embedded in a vector space, which is  the  case for the dual of an asymmetric normed space.  The theory of locally convex cones, with applications to Korovkin type approximation theory for positive operators and to vector-measure theory, is developed in the books by Keimel and Roth \cite{K-Roth92} and Roth \cite{Roth09}, respectively. A recent paper by Galanis
 \cite{galanis09} discusses G\^{a}teaux and Hukuhara differentiability on topological cones (called by him topological semilinear spaces, meaning  cones for which the operations of addition and  multiplication by positive scalars are continuous).

A {\it linear mapping}  between two cones $X,Y$ is an additive and positively homogeneous mapping $A:X\to Y.$
An {\it asymmetric seminorm} on a cone $X$ is a mapping $p:X\to \Real^+$ such that %
$$%
\begin{aligned} %
\mbox{(i)}\quad  &p(0)= 0 \;   \mbox{and} \;  (x,-x\in X \land p(x)=p(-x)=0 )\;\Rightarrow x=0\\
\mbox{(ii)}\quad &p(\alpha x)=\alpha p(x); \\
\mbox{(iii)}\quad &p(x+y)\leq  p(x) + p(y), %
\end{aligned} %
$$%
for all $x,y\in X$ and $\alpha \geq 0.$  If %
$$%
\mbox{(iv)}\quad p(x)=0 \iff x=0, %
$$%
then $p$ is called an {\it asymmetric norm.}

Starting from an     asymmetric seminorm $p$ on a cone $X$ one can define an extended quasi-semimetric $e_p$
on $X$ by the formula %
\bequ\label{def.cone-metric} %
e_p(x,y)=\begin{cases} \inf\{p(z) : z\in X,\,  y=x+z\} \quad \mbox{if}\quad y\in x+X,\\
\infty\qquad\qquad \mbox{otherwise}. %
\end{cases} %
\eequ %

An extended quasi-semimetric  $d:X\times X\to [0;\infty]$  on a cone $X$ is called {\it subinvariant} provided %
\bequ\label{def.subinv}%
\begin{aligned} %
\mbox{(i)}\quad &d(x+z,y+z))\leq  d(x,y),\; \mbox\;{and} \;
\mbox{(ii)}\quad &d(\alpha x,\alpha y)=\alpha d(x,y), \\
\end{aligned} %
\eequ %
for all $x,y,z\in X$ and $\alpha\geq 0.$

For instance, $\Real^+$ is a cancellative cone and $u(\alpha)=\alpha$ is an asymmetric norm  on $\Real^+.$ The associated extended quasi-metric,  given by $e_u(x,y)=y-x$ if $x\leq y$ and $e_u(x,y)=\infty,$ otherwise,  induces the Sorgenfrey topology on $\Real^+$ (see Example \ref{ex.Sorgenfrey}).

The topological notions for a normed  cone $(X,p)$ will be considered with respect to this extended quasi-semimetric. As before, one associates to $e_p$ the conjugate quasi-semimetric $\bar e_p(x,y)=e_p(y,x)$ and the (symmetric)
extended semimetric $e^s_p(x,y)=\max\{ e_p(x,y),\bar e_p(x,y)).$

Some  properties of this quasi-metric are collected in the following proposition. %
\begin{prop}\label{p.cone1} %
Let $(X,p)$ be an asymmetric normed cone. %
\begin{itemize} %
\item[{\rm 1.}]
The function $e_p$ defined by \eqref{def.cone-metric} is a subinvariant  extended   quasi-semimetric  on $X$.
\item[{\rm 2.}]
The  equality %
$$%
r\cdot B_{e_p}(x,\epsi)=rx+\{y\in X : p(y)\leq \epsi r\}, %
$$%
holds for every $x\in X\,$ and $ \,r,\epsi >0.$
\item[{\rm 3.}]
The translations with respect to + and $\cdot$ are $\tau(e_p)$-open,
that is, if $Z\subset X$ is $\tau(e_p)$-open, then both $ x+Z $ and $r\cdot Z$ are   $\tau(e_p)$-open. %
\end{itemize} %
\end{prop} %

We shall present now some closed graph and open mapping results for normed cones  proved by Valero \cite{valero08a}.
A uniform boundedness principle for locally convex cones was proved by Roth \cite{roth98} (see also Roth \cite{Roth09}).

An asymmetric normed cone is called {\it bicomplete} if it  is complete with respect to the extended metric $e^s_p.$  As it was shown
by an example in \cite{valero08a}, the Closed Graph and the Open Mapping Theorems do not hold for bicomplete asymmetric normed cones, some supplementary  hypotheses being necessary.

A mapping $f$ between two topological spaces $(S,\sigma)$ and $(T,\tau)$ is called {\it almost continuous} at $s\in S$  if for every open subset $V$ of $T$ such that $f(s)\in V,$ the set  $\clos_\sigma(f^{-1}(V))$ is a $\sigma$-neighborhood of $s.$  A subset $A$ of a bitopological space $(T,\tau_1,\tau_2)$ is called $(\tau_1,\tau_2)$-{\it preopen} if $A\subset \tau_1\mbox{-}\inter\,(\tau_2\mbox{-}\clos\, A).$ A mapping  $f$ from a topological space $(S,\sigma)$ to a bitopological space  $(T,\tau_1,\tau_2)$ is called {\it almost open} if
$f(U)$ is $(\tau_1,\tau_2)$-preopen for every $\sigma$-open subset $U$ of $S.$

The closed graph theorem proved by Valero \cite{valero08a} is the following. %
\begin{theo}\label{t.closed-gr-cone} %
Let $(X,p)$ and $(Y,q)$ be two asymmetric normed cones such that the cone $Y$ is right $K$-sequentially complete with respect to the conjugate extended quasi-metric  $\bar e_q.$
If $A:X\to Y$ is a linear mapping with closed graph in $(X\times Y,\bar e_p\times\bar e_q)$ which is $(\bar e_p,e_q)$-almost continuous at 0, then $A$   is continuous. %
\end{theo} %

An open mapping theorem holds in similar conditions. %
\begin{theo}\label{t.open-map-cone} %
Let $(X,p)$ and $(Y,q)$ be two asymmetric normed cones such that the cone $Y$ is right $K$-sequentially complete with respect to the conjugate extended quasi-metric  $\bar e_q.$
If $A:X\to Y$ is a linear mapping with closed graph in $(X\times Y,\bar e_p\times\bar e_q)$ which is almost open as a mapping from $(X,\tau(e_p))$ to $(T,\tau(e_q),\tau_(\bar e_q))$, then $A$   is continuous. %
\end{theo} %

There are also other results on normed cones: the paper \cite{rafi-rom-per-val04} discusses the metrizability
of the unit ball of the dual of a normed cone, Oltra and Valero \cite{oltra-valero04b} study the isometries and bicompletions of normed cones and Valero \cite{valero06} defines and studies the properties of quotient normed cones (the study of quotient spaces of asymmetric normed spaces is done in \cite{aleg-fer07}).  Other properties are investigated in a series of papers Romaguera, S\'{a}nchez P\'{e}rez and Valero: in \cite{romag-per-valero07} one considers generalized monotone normed cones, quasi-normed monoids are discussed in \cite{romag-per-valero03b}, the dominated extension of functionals, of $V$-convex functions  and duality on cancellative and noncancellative normed cones are treated in \cite{romag-per-valero03a} and \cite{romag-per-valero06}, respectively.

\subsection {The $w^\flat$-topology of the dual space   $X^\flat_p$}

This is the analog of the $w^*$-topology on the dual of a normed space, which we shall present following \cite{rafi-rom-per03a}.

Let $(X,p)$ be a space with asymmetric norm and $X^\flat_p$ its asymmetric dual. The
$w^\flat$-topology on $X^\flat_p$ is  the topology admitting as a neighborhood
basis of a point $\vphi \in X^\flat_p$ the sets %
\bequ\label{def.w-flat}%
V_{\epsilon,x_1,...x_n}(\vphi) = \{\psi \in X^\flat_p : \psi(x_k) - \vphi (x_k) <
\epsilon,
\; k=1,2,...,n\}, %
\eequ %
for all $\epsilon >0,$ and all $  n\in \mathbb N$ and $x_1,...,x_n\in X$.

 The topology $w^\flat$ is derived from a quasi-uniformity
 $\mathcal W^\flat_p$ on $X^\flat_p$ with a basis formed of the sets%
 \bequ\label{def.w-flat-unif} %
 V_{x_1,...,x_n;\,\epsilon}=\{(\vphi_1,\vphi_2)\in X^\flat_p \times X^\flat_p :
\vphi_2(x_i)-\vphi_1(x_i)\leq \epsilon,\; i=1,...,n\},  %
\eequ%
for $n\in \Nat,\, x_1,...,x_n\in X$ and $\epsilon >0.$ Note that,
for fixed $\vphi_1=\vphi,$ one obtains the neighborhoods from
\eqref{def.w-flat}.

By the definition of the topology $w^\flat$,
the $w^\flat$-convergence of a net $(\vphi_\gamma)$ in $X^\flat_p$ to $\vphi \in X^\flat_p$ is
equivalent to %
$$%
\forall x\in X, \quad \vphi_\gamma(x) \to \vphi(x) \quad\mbox{in}\quad (\mathbb R,u). %
$$%

 The following proposition shows that, in fact, a stronger result holds.  %
 \begin{prop}\label{p.w-flat1} %
 Let $(X,p)$ be an asymmetric normed space with dual $X^\flat_p.$ %

  The $w^\flat$-topology is the restriction to $X^\flat_p$ of the
$w^*$-topology of the space $X^* = (X,p^s)^*.$ Consequently,  it is Hausdorff and the
$w^\flat$-convergence of a net $(\vphi_\gamma)$ in $X^\flat_p$ to $\vphi \in X^\flat_p$ is
equivalent to %
$$%
\forall x\in X, \quad \vphi_\gamma(x) \to \vphi(x) \quad\mbox{in}\quad (\mathbb R,|\cdot|). %
$$%
\end{prop} %
\begin{proof} %
The first assertion is a direct consequence of the definition of the topology $w^\flat.$
 The second assertion follows from the remarks that  %
$$%
V_{\epsilon,x} \cap  V_{\epsilon,-x} =\{\psi \in X^\flat_p : |\psi(x)-\vphi(x)| <
\epsilon\} %
$$%
is a $w^\flat$-neighborhood of $0\in X$  and $X^\flat_p\subset X^*.$ %
\end{proof}%

   A deep result in Banach space theory  is the Alaoglu-Bourbaki Theorem:\\

   {\sc Theorem} (Alaoglu-Bourbaki) {\it  The closed unit ball $B_{X^*}$
   of the dual of a normed space $X$ is $w^*$-compact.}\\

    The analog of this theorem  for asymmetric normed spaces was proved in
    \cite[Theorem 4]{rafi-rom-per03a}. We  include a slightly simpler proof of this result.

 \begin{prop} \label{p.w-comp}%
 The closed unit ball $B^\flat_p=B_{X^\flat_p}$ of the space $X^\flat_p$ is a $w^*$-closed subset of the
closed unit ball  $B^*$ of the space $X^* = (X,p^s)^*.$
\end{prop} %
\begin{proof}%
Let $(\vphi_\gamma)$ be a net in $B^\flat_p$ that is $w^*$-convergent to an element
  $\vphi \in X^*$, i.e.  for every $x\in X$ the net $(\vphi_\gamma(x))$
  converges to $\vphi(x)$ in $(\mathbb R,|\cdot|).\,$ Because $\forall x\in X,\;
  \vphi_\gamma(x) \leq p(x),\,$ it follows $\vphi(x)\leq p(x)$ for all $x\in X$, showing
  that $\vphi \in B^\flat_p$.
\end{proof} %
\begin{theo}\label{t.Alaoglu-Bourb}%
The closed unit ball  $B^\flat_p$ of the dual $X^\flat_p$ of an asymmetric normed space $(X,p)$ is $w^\flat$-compact. %
\end{theo} %
\begin{proof} %
By Alaoglu-Bourbaki theorem the ball $B^*$ is $w^*$-compact, so that, as a $w^*$-closed
subset of $B^*$, the ball $B^\flat_p$ will be $w^*$-compact too. Since the
$w^\flat$-topology is the restriction of $w^*$ to $X^\flat_p$, it follows that the set
$B^\flat_p$ is also $w^\flat$-compact. %
\end{proof}%

\begin{remark} {\rm In \cite[Proposition 2.11]{cobz05a}  Alaoglu-Bourbaki Theorem was extended to asymmetric locally convex spaces: the polar set of any neighborhood of 0 is $w^\flat$-compact convex subset of
the asymmetric dual cone $X^\flat$.} %
\end{remark}

\subsection{ Compact subsets of asymmetric normed spaces}

At this point we shall present, following \cite{rafi05} and \cite{aleg-fer-raf08}, some results on compactness specific to asymmetric normed normed spaces.

Let $(X,p)$ be an asymmetric normed space. For $x\in X$ put %
\bequ\label{def.theta} %
\theta(x)=\{y\in X : p(y-x)=0\}.
\eequ %

It is clear that $\theta(x)=x+\theta(0)$ and  $Y+\theta(0)=\cup\{\theta(y) : y\in Y\}. $ Also $\theta(x)=\tau(\bar p)$-$\clos (\{x\}),$ as  can be seen from the equivalences %
\begin{align*} %
y\in \theta(x)\;\iff& p(y-x)=0\iff \bar p(x-y)=0\\
\iff& \forall \epsi > 0,\; \bar p(x-y)<\epsi\iff \forall \epsi >0,\; x\in B_{\bar p}(y,\epsi)\\
\iff& y\in \tau(\bar p)\mbox{-}\clos (\{x\}). %
\end{align*}

The following properties hold.
\begin{prop}[\cite{rafi05}]\label{p.comp1} %
 Let $(X,p)$ be an asymmetric normed normed space, $x\in X$ and $\epsi >0.$ Then $B_p(x,\epsi)=\theta(0)+B_p(x,\epsi) +\theta(0).$
 Also, if $Y$ is a $\tau(p)$-open subset of $X,$ then $Y=Y+\theta(0). $%
  \end{prop}

As it is shown by Garc\'{\i}a-Raffi \cite{rafi05} the sets $\theta(x)$ are involved in the study of compactness in asymmetric normed spaces. %
\begin{prop}\label{p.comp2} %
Let $(X,p)$ be an asymmetric normed space and $K\subset X$. %
\begin{itemize}
\item[{\rm 1.}] The set  $K$  is $\tau_p$-compact if and only if $K+\theta(0)$ is $\tau_p$-compact.\\
If $K+\theta(0)$ is $\tau_p$-compact, $K_0\subset K+\theta(0)$   and $K_0+\theta(0)=K+\theta(0), $ then $K_0$ is $\tau_p$-compact.
\item[{\rm 2}] A finite sum and a finite union of $p$-precompact sets is $p$-precompact
\item[{\rm 3}] The convex hull of a $p$-precompact set is $p$-precompact.
\item[{\rm 4}] The set $K$ is $p$-precompact if and only if the $\tau(\bar p)$-closure of $K$ is $p$-precompact.
\item[{\rm 5}] If $K_0\subset K\subset K_0+\theta(0)$ and $K_0$-is $\tau(p^s)$-compact, then $K$ is $p$-compact.
\end{itemize}
\end{prop}

Garc\'{\i}a-Raffi \cite{rafi05} obtained characterizations of finite dimensional normed spaces similar to those known for normed spaces.
In the following proposition all topological notions refer to $\tau(p).
$

\begin{theo}\label{t.comp3} %
  Let $(X,p)$be an asymmetric normed normed space.
\begin{itemize}
  \item[{\rm 1.}] If $X$ is finite dimensional, of dimension $n\geq 1,$ and  $T_1,$ then $X$ is topologically isomorphic to the Euclidean space $\Real ^n$.
  \item[{\rm 2.}] If $(X,p)$   is $T_1$, then $X$ is finite dimensional if and only if its closed unit ball $B_p$ is $\tau(p)$-compact.
    \item[{\rm 3.}] Suppose that $X$ is finite dimensional. Then $X$ is $T_1$ if and only if every $\tau(p)$-compact subset of $X$ is $\tau(p)$-closed.
\end{itemize}
\end{theo}

As it is shown  in \cite[Example 12]{aleg-fer-raf08}, the property 5 from Proposition \ref{p.comp2} does not characterize the $p$-compactness. This paper contains also some further results  on the relations between the $\tau(p^s)$-compactness of $K_0$ and the $p$-compactness of $K,$  involving a notion of boundedness called right-boundedness. The unit closed ball $B_p$ of $(X,p)$ is called {\it right-bounded} if there exists $r>0$ such that %
$$%
r\,B_p\subset B_{p^s}+\theta(0). %
$$%
Note that the inclusion $B_{p^s}+\theta(0)\subset B_p$ is always true. %
\begin{theo}\label{t.comp4} %
  Let $(X,p)$ be an asymmetric normed normed space.
\begin{itemize}
  \item[{\rm 1.}]\, {\rm (\cite{rafi05})} If $X$ is finite dimensional  the unit closed ball $B_p$ is  right-bounded, then $B_p$ is $\tau(p)$-compact.
  \item[{\rm 2.}]\,{\rm(\cite{aleg-fer-raf08})}  Suppose that  $(X,p)$ is biBanach with $B_p$ right-bounded with $r=1.$ If $K\subset X$ is $p$-precompact, then there exists a $p^s$-compact subset $K_)$ of $K$ such that $K\subset K_0+\theta(0).$
  \item[{\rm 3.}]\,{\rm (\cite{aleg-fer-raf08})}     If $K_0$ is $p^s$-precompact and $K\subset K_0+\theta(0),$ then $K$ is outside $p$-precompact.
\end{itemize}
\end{theo}  %

\subsection{The conjugate operator, precompact operators between asymmetric normed spaces and a Schauder type theorem}

 A Schauder type theorem on the compactness of conjugate of a compact linear operator on an asymmetric normed space was proved in \cite{cobz06}. We shall briefly present this result, referring for the details to the mentioned paper.

 For a continuous linear operator $A:(X,p)\to (Y,q)$ between two asymmetric normed spaces, one defines the {\it conjugate operator} $A^\flat:Y^\flat_q\to X_p^\flat$ by the formula %
 \bequ\label{def.conj-op} %
 A^\flat\psi=\psi\circ A,\;\; \psi\in Y_p^\flat. %
\eequ%
Concerning the continuity we have.

\begin{prop}\label{p.conj-op1} %
Let $(X,p),(Y,q)$ be asymmetric normed spaces and $A:X\to Y$ a continuous linear operator.
\begin{itemize}%
\item[{\rm 1.}]
The operator $A^\flat:(Y_q^\flat,\|\cdot|_q)\to (X_p^\flat,\|\cdot|_p)$ is additive, positively homogeneous and continuous. So it is also quasi-uniformly continuous with respect to
the quasi-uniformities  $\rondu_q^\flat$ and $\rondu_p^\flat$
on $Y^\flat_q\,$ and $X^\flat_p\,,$ respectively.
\item[{\rm 2.}]
 The operator $A^\flat$ is also
quasi-uniformly continuous with respect to the
$w^\flat$-quasi-uniformities $\rondw^\flat_q$ on $Y^\flat_q$ and $\rondw^\flat_p$ on
$X^\flat_p\,.$
\end{itemize} %
\end{prop} %
\begin{proof}
  1.  Obviously that $A^\flat$ is properly defined, additive and
positively homogeneous.

  For every $\psi\in Y^\flat_q,$  %
  $$%
  \|A^\flat\psi|_q=\|\psi\circ A^\flat|_q\leq \|\psi|_q\,\|A|_{p,q}, %
  $$%
  implying the continuity of $A^\flat,$  which, in its turn, by the linearity of $A$, implies the quasi-uniform continuity with respect to the quasi-uniformities $\rondu^\flat_q$ and $\rondu^\flat_p.$

  2.   For   $x_1,...,x_n\in X$ and $\epsilon >0$   let %
$$%
V=\{(\vphi_1,\vphi_2)\in X^\flat_p\times X^\flat_p :
\vphi_2(x_i)-\vphi_1(x_i)\leq \epsilon,\; i=1,..., n\} %
$$%
be a $w^\flat$-entourage in $X^\flat_p.$ Then %
$$%
U=\{(\psi_1,\psi_2)\in Y^\flat_q\times Y^\flat_q :
\psi_2(Ax_i)-\psi_1(Ax_i)\leq \epsilon,\; i=1,...,n\}, %
$$%
is a $w^\flat$-entourage in $  Y^\flat_q$ and %
$(A^\flat\psi_1,A^\flat\psi_2)\in V$ for every $(\psi_1,\psi_2)\in
U, $ proving the quasi-uniform continuity of $A^\flat$ with
respect to the $w^\flat$-quasi-uniformities on $Y_q^\flat$ and
$X^\flat_p\,$\,. %
 \end{proof}

 A linear operator $A:(X,p)\to (Y,q)$ between two asymmetric normed spaces is called $(p,q)$-{\it precompact} if the image $A(B_p)$ of the closed unit ball
$B_p$ of $X$ by the operator $A$ is a $q$-precompact subset of $Y.$ We shall denote by $(X,Y)^k_{p,q}$ the set of all $(p,q)$-precompact operators from $X$ to $Y.$
A subset $Y$ of a quasi-uniform space $(X,\rondu)$ is called $\rondu$-{\it precompact} if for every $\epsi > 0$ there exists a finite subset $Z$ of $Y$ such that
$Y\subset U[Z].$ If there exists  set $Z\subset X$ such that $Y\subset U[Z],$ then $Y$ is called {\it outside $\rondu$-precompact}. It is clear that a subset $Y$ of a asymmetric normed space $(X,p)$ is (outside)  $p$-precompact if and only if it is (outside) $\rondu_p$-precompact.

For $\mu\in \{p,\bar p,p^s\}$ and
$\nu\in \{q,\bar q,q_s\}$ let $\rondu_{\mu,\nu}$ denote by $(X,Y)^\flat_{\mu,\nu}$ the cone of all continuous linear operators from $(X,\mu)$ to $(Y,\nu).$ The space
$(X,Y)^*_s:=(X,Y)^\flat_{p^s,q^s}$  is the space of all continuous linear operators between the associated normed spaces $(X,p^s)$ and $(Y,p^s),$ which was   denoted also by $L(X,Y).$

On the space $(X,Y)^*_s$ we shall consider several
quasi-uniformities.
For $\mu\in \{p,\bar p,p^s\}$ and
$\nu\in \{q,\bar q,q^s\}$ let $\rondu_{\mu,\nu}$
 be the
quasi-uniformity
generated by the basis %
\bequ\label{def.q-u} %
U_{\mu,\nu;\,\epsilon} =\{(A,B) ; A,B\in (X,Y)^*_s, \;
\nu(Bx-Ax)\leq
\epsilon,\;\forall x\in B_\mu,\}, \;\;\epsilon >0,%
\eequ %
where $B_\mu=\{x\in X : \mu(x)\leq 1\}$ denotes the unit ball of
$(X,\mu).$ The induced quasi-uniformity on the semilinear subspace
$(X,Y)^\flat_{\mu,\nu}$ of $(X,Y)^*_s$ is denoted also by
$\rondu_{\mu,\nu}$ and the corresponding topologies  by
$\tau(\mu,\nu).$ The uniformity $\rondu_{p^s,q^s}$ and the
topology $\tau(p^s,q^s)$ are those  corresponding to the norm
\eqref{def.norm-op2} on the space $(X,Y)^*_s\,$, while the quasi-uniformity  $\rondu_{p,q}$ corresponds to the extended asymmetric  norm
$\|\cdot|^*_{p,q}\,$ given by \eqref{def.norm-op3}.
In the case of the dual space $X^\flat_\mu$ we shall use the
notation $\rondu^\flat_\mu$ for the quasi-uniformity $\rondu_{\mu,u}\,.$

Notice that, for $\,\mu=p^s\,$ and $\,\nu=q^s,\,$ the space
$(X,Y)^\flat_{p^s,q^s}$ agrees with $(X,Y)^*_s,$ the
$(p^s,q^s)$-compact operators are the usual linear compact
operators between the normed spaces $(X,p^s)$ and $(Y,q^s),$ so
the following proposition contains some  well
known results for compact operators on normed spaces.%

\begin{prop}\label{p.k-op1} %
Let  $(X,p),(Y,q)$ be asymmetric normed spaces. The following
assertions hold. %
\begin{itemize} %
\item[{\rm 1.}]
 $(X,Y)^k_{\mu,\nu}$ is a subcone  of the cone $\,(X,Y)^\flat_{\mu,\nu} $ of all continuous linear operators from $X$ to $Y.$
\item[{\rm 2.}]
 $(X,Y)^k_{p,q}$ is $\tau(p, q^s)$-closed in
$\,(X,Y)^\flat_{p,q}. $
\end{itemize} %
\end{prop} %

Now we are prepared to enounce and prove  the analog of the Schauder theorem. %
\begin{theo}\label{t.Sch} %
Let $(X,p), (Y,q)$ be  asymmetric normed spaces. If the linear
operator $A:X\to Y$ is $(p,q)$-compact, then $A^\flat(B^\flat_q)$
is precompact with respect to the quasi-uniformity $\rondu_p^\flat$ on $X^\flat_p.$ %
\end{theo} %
\begin{proof} %
For  $\epsilon > 0 $ let %
$$%
U_\epsilon=\{(\vphi_1,\vphi_2)\in X^\flat_p\times X^\flat_p :
\vphi_2(x)-\vphi_1(x)\leq \epsilon, \; \forall x\in B_p\}, %
$$%
be an entourage in $X^\flat_p$ for the quasi-uniformity
$\rondu_p^\flat\,.$

Since $A$ is $(p,q)$-compact, there exist
$x_1,...,x_n \in B_p\,$ such that %
\bequ\label{eq1.Sch} %
\forall x\in B_p\,,\; \exists i\in \{1,...,n\},\;\quad
q(Ax-Ax_i)\leq \epsilon. %
\eequ %

By the Alaoglu-Bourbaki theorem, Theorem \ref{t.Alaoglu-Bourb}, the set $B^\flat_q$ is $w^\flat$-compact, so
by the $(w^\flat,w^\flat)$-continuity of the operator $A^\flat$
(Proposition \ref{p.conj-op1}),  the set $A^\flat(B^\flat_q)$ is
$w^\flat$-compact in $X^\flat_p.$ Consequently, the $w^\flat$-open
cover of $A^\flat(B^\flat_q),$%
$$%
V_\psi=\{\vphi\in X^\flat_p : \vphi(x_i)-A^\flat\psi(x_i)
<\epsilon,\; i=1,...,n\},\;\psi\in B^\flat_q, %
$$%
contains a finite subcover $V_{\psi_k},\, 1\leq k\leq m, $ i.e,
\bequ\label{eq2.Sch} %
A^\flat(B^\flat_q)\subset \bigcup\{V_{\psi_k} : 1\leq k\leq m\}, %
\eequ %
for some $m\in \mathbb N$ and $\psi_k\in B^\flat_q,\; 1\leq k\leq
m.$

Now let  $\psi \in B^\flat_q\,.$  By \eqref{eq2.Sch} there exists
$k\in \{1,...,m\}$ such that %
$$%
A^\flat\psi(x_i)-A^\flat\psi_k(x_i) < \epsilon,\; i=1,...,n. %
$$%
 If  $x\in B_p,$ then, by \eqref{eq1.Sch}, there  exists $i\in
 \{1,...,n\}, $ such that %
 $$%
 q(Ax-Ax_i)\leq \epsilon. %
 $$%

 It follows  %
 $$%
 \begin{aligned} %
 &\psi(Ax)-\psi_k(Ax) =\\
 & =
 \psi(Ax)-\psi(Ax_i)+\psi(Ax_i)-\psi_k(Ax_i)+\psi_k(Ax_i)-\psi(Ax_i)\\
&\leq 2q(Ax-Ax_i)+\epsilon \leq 3\epsilon. %
\end{aligned} %
$$%

Consequently, %
$$%
\forall x\in B_p\,,\; \quad (A^\flat\psi-A^\flat\psi_k)(x)\leq 3\epsilon, %
$$%
proving that %
$$%
A^\flat(B^\flat_q)\subset U_{3\epsilon}[\{A^\flat\psi_1,...,A^\flat\psi_m\}]. %
$$%
\end{proof} %
\begin{remark}\label{rem.comp-op} {\rm  As a measure of precaution, we have restricted our study to precompact linear operators $A$ on an asymmetric normed space $(X,p)$ with values in another asymmetric normed space $(Y,q)$, meaning that the image $A(B_p)$ of the unit ball of $X$ by $A$ is a $q$-precompact subset $Y.$  A compact linear operator  should be defined by the condition that $A(B_p)$ is a relatively compact subset of   $Y,$
that is the $\tau_q$-closure of  $A(B_p)$ is $\tau_q$-compact subset of  $Y,$  in concordance to the definition of  compact linear operators between normed spaces.  }

{\rm    As can be seen from Section \ref{Complete-Comp}, the relations
between precompactness, total boundedness and completeness are
considerably more complicated in the asymmetric case than in the
symmetric one. The compactness properties  of the set
$A(B_p)$  need a study of the completeness of the space
$(X,Y)^\flat_{\mu,\nu}$ with respect to various quasi-uniformities
and various notions of completeness, which could be a topic of
further investigation.} %
\end{remark} %

\subsection{ Asymmetric moduli of smoothness and rotundity  }

A convex body in a normed space $(X,\norm)$ is a bounded closed convex set with nonempty interior.
A {\it rooted body} in a normed space $X$ is a pair $(K,z)$ where $K$ is a convex body and $z$ is a fixed point in the interior  $K,$ called a {\it root} for $K.$   If $(K,z)$ is a rooted body, then $0\in \inter(K-z)=\inter(K) -z,$ so that the Minkowski functional $p_{K-z}$ is well defined and it is an asymmetric norm on $X$ satisfying $p(x)>0$ for  $x\neq 0.$  The topology defined by $p_{K-z}$ is equivalent  to the norm-topology $\tau_{\norm}$ of $X.$ This follows from the facts that there exist $0<r<R$ such that $B_{\norm}(0,r)\subset K-z\subset B_{\norm}(0,R)$ and $K-z=\{x\in X : p_{K-z}(x)\leq 1\}$ is the closed unit ball of the asymmetric normed space $(X,p_{K-z}).$

Starting from this definition one can introduce geometric properties of convex bodies, as rotundity, local uniform rotundity, smoothness, uniform smoothness, expressed in terms of Minkowski functionals,  by analogy with those from Banach space geometry (see, for instance, the book by Megginson \cite{Meg98}). This was done in a series of papers (see  \cite{klee-maluta-zanco86,klee-maluta-zanco91,klee-vesel-zanco96}) by V. Klee, E. Maluta, C. Zanco and L. Vesely, mainly in connection with the existence of good tilings of normed spaces.   A {\it tiling} of a normed linear space $X$ is a covering $\mathcal T$ of $X$ such that each {\it tile} (member of $\mathcal T$) is a convex body  and no point of $X$ is interior to more than one tile. In contrast to the genuine theory of tilings in finite-dimensional spaces and an extensive theory for the plane,  few significant examples  were known and not even a rudimentary theory of tilings in infinite-dimensional spaces was built, see \cite{klee86}. In the above mentioned papers  substantial  progress was done in the study of tilings in infinite dimensional Banach spaces.  Zanco and Zucchi \cite{zanco-zuchi93} defined and studied some properties of the asymmetric moduli of smoothness, quantitative expressions (in terms of Minkowski functionals) of the rotundity and smoothness properties of convex bodies.

Some uniform  geometric properties for families of convex bodies were defined in \cite{klee-maluta-zanco91}. We shall restrict these properties to a single convex body, a situation that fits better our needs. A convex body $K$ in
a normed space $(X,\norm)$ is called {\it smooth} if for any point $y\in \partial K$ (the boundary of $K$) there exists exactly one hyperplane $H_y=\{x\in X : x^*(x)=c\},$ for some $x^*\in X^*$ and $c\in \Real,$ supporting $K$ at $x.$ That is $x^*(y)=c$ and $x^*(x)\leq c$ for all $x\in K.$ The convex body $K$ is called {\it rotund} if its boundary $\partial K$ does not contain nontrivial line segments, or, equivalently, if $d(\frac{1}{2}(x+y),\partial K) >0$ (the distance with respect to $\norm$) for every pair of distinct elements $x,y\in \partial K.$  The convex body $K$ is called {\it uniformly rotund} if for every $\epsi> 0$ there exists $\delta(\epsi)>0$ such that $d(\frac{1}{2}(x+y),\partial K) \geq \delta(\epsi)$
for every $x,y\in K$ with $\|x-y\|\geq \epsi.$ Obviously that one obtains the same notion if we require that
$x,y\in \partial K.$

A rooted body $(K,z)$ is smooth if and only if the Minkowski functional $p=p_{K,z}$ is G\^{a}teaux differentiable on $X\setminus\{z\}.$  This means that for any point $x\in X\setminus\{z\}$  there exists a continuous linear functional $p'(x;\cdot)\in X^*$ such that for every $h\in X$ %
\bequ\label{def.G-diff} %
\lim_{t\to 0}\frac{p(x+th)-p(x)}{t}=p'(x;h). %
\eequ %

 With a rooted convex body $(K,z)$ one can associate two duality mappings from $\partial K$ to $X^*$ defined for $y\in \partial K$ as follows:

$\bullet$ \quad $J_{(K,z)}(y)=$ the  unique functional $x^*\in X^*$ such that $\|x^*\|=1$ and  $H_y=\{x\in X : x^*(x)=c\},$ and

$\bullet$ \quad $J_{(K,z)}^1(y)=$ the  unique functional $x^*\in X^*$ such that  $H_y=\{x\in X : x^*(x)=1\},$
   for some  $c\in \Real.$

If $K=B_X$ (the closed unit ball of $(X,\norm)$) and $z=0$, then $J=J^1$ is the ordinary duality mapping, a very important notion in the  geometry of Banach spaces and its applications to nonlinear operator theory, see, for instance, the book by Cior\u{a}nescu \cite{Cioran}.   Also, it is clear that $J_{K,z_1}=J_{K,z_2}$ for every $z_1,z_2\in \inter K,$  so that the duality mapping  $J_{K,z_1}$ does not depend on the root $z,$  consequently it can be denoted simply by $J_K.$

A normed space $(X,\norm)$ is called {\it uniformly smooth} if the norm is uniformly Fr\'{e}chet differentiable on the unit sphere $S_X$. This means that, putting $f(x)=\|x\|,\,x\in X,$ for every $\epsi > 0$ there exists $\delta=\delta(\epsi) >0$ such that for every $x\in S_X$ and every $h\in X$ with $\|h\|\leq \delta$ %
$$%
|f(x+h)-f(x)-f'(x)h|\leq \epsi \|h\|, %
$$%
 where $f'(x)\in X^*$ denotes the Fr\'{e}chet derivative of $f$ at $x.$   The modulus of smoothness of the space $X$ is defined for $\tau\geq 0$ by%
 \bequ\label{def.mod-sm} %
 \rho_X(\tau)=\sup\{\frac{1}{2}\,(\|x+y\|+\|x-y\| -2 : x\in S_X,\, \|y\|=\tau. %
 \eequ

 The uniform smoothness of $X$ can be characterized by the condition %
 $$%
 \lim_{\tau\searrow 0}\frac{\rho_X(\tau)}{\tau}=0.\, %
 $$%
see, e.g., \cite{Meg98}. Another characterization can be done in terms of the duality mapping $J,$ namely, the space $X$ is uniformly smooth if and only if the duality mapping  $J:S_X\to S_{X^*}$ is norm-to-norm continuous (see \cite{Cioran}).

Starting from this property one can define the {\it uniform smoothness} of a rooted body $(K,z)$ by asking that
the duality mapping $J_{K,z}:\partial K\to S_{X^*}$  is norm-to-norm continuous (see \cite{klee-maluta-zanco91}). The norm-to-norm continuity of the duality mapping  $J_{K,z}$ is equivalent to the norm-to-norm continuity of the duality mapping $J^1_{K,z}\,,$ so that one obtains the same notion of uniform smoothness by working with the duality mapping $J^1_{K,z}\,.$

A normed space $(X,\norm)$ is called {\it uniformly rotund}  ({\it uniformly convex} by some authors) if for every $\epsi>0$ there exists $\delta=\delta(\epsi)> 0$ such that  $\|x+y\|\geq 2(1-\delta)$ for every $x,y\in S_X$ with $\|x-y\|\leq \epsi.$

This property  admits also quantitative  characterizations in terms of some moduli. We mention  two of them.

{\it Clarkson's modulus of uniform  rotundity}\;$ \delta_X:[0;2]\to[0;1]$ defined by %
\bequ\label{def.Clark-mod} %
\begin{aligned}
\delta_X(\epsi)=&\inf\{1-\frac{1}{2}\|x+y\| : x,y\in S_X,\, \|x-y\|\geq \epsi\} \\
=&\inf\{1-\frac{1}{2}\|x+y\| : x,y\in S_X,\, \|x-y\|\geq \epsi\}. %
 \end{aligned}
\eequ

 {\it Gurarii's   modulus of uniform rotundity } $\gamma_X:[0;2]\to[0;1]$ defined by %
 \bequ\label{Gurari-mod} %
 \gamma_X(\epsi)=\inf\{\max_{0\leq t\leq 1}(1-\|tx+(1-t)y\|) : x,y\in S_X,\, \|x-y\|=\epsi\}.
 \eequ %

 One has $\delta_X\leq\gamma_X,$ and a 2-dimensional example given in \cite{zanco-zuchi93} shows that the inequality can be strict. Also it is unknown whether the condition $\|x-y\|=\epsi$ in the definition of Gurarii's modulus can be replaced by the condition $\|x-y\|\geq \epsi,$ as in the definition of Clarkson's
 modulus. The uniform rotundity of the space $X$ can be characterized in the following way: %
$$%
 \; X\;  \mbox{is uniformly rotund} \;\iff \; \forall \epsi\in (0;2],\, \delta_X(\epsi)>0 \;
 \iff \forall \epsi\in (0;2],\, \gamma_X(\epsi)>0. %
$$%

The {\it characteristic of convexity} corresponding to the moduli $\delta_X$ and $\gamma_X$ are $\epsi^0_X=\sup\{\epsi\in [0;2] : \delta_X(\epsi)=0\}=\{\epsi\in [0;2] : \gamma_X(\epsi)=0\}. $  In terms of this characteristic of convexity the uniform rotundity of $X$ is characterized by the condition $\epsi^0_X=0$ and the rotundity by the condition $\delta_X(2)=\gamma_X(2)=1.$ Other geometric properties of the normed space $X$, as, for instance, the uniform non-squareness can also be expressed in terms of these moduli.

The analogs of these moduli (Gurarii's variant in the case of uniform rotundity) and their relevance for the smoothness and rotundity properties of
 rooted bodies were given in the paper by Zanco and Zucchi \cite{zanco-zuchi93}.

 The modulus of smoothness of a rooted body $(K,z),\, \rho_{(K,z)}:[0;\infty) \to [01)$ can be defined by replacing the norm in \eqref{def.mod-sm} by the Minkowski functional :
 \bequ\label{def.M-mod-sm}%
 \rho_{(K,z)}(\tau)=\sup\{\frac{1}{2}(p_{K-z}(x+y)+p_{K-z}(x-y))-1 : x\in \partial(K-z),\, y\in X,\, p_{(K,z)}(y)=\tau\}. %
 \eequ%

 For the sake of simplicity suppose $z=0$ and put $\rho_K=\rho_{(K,0)}.$
 The modulus of smoothness is a continuous convex function such that $\rho_K(\tau)\leq \lambda_K \tau, $ where $$%
 \lambda_K=\sup\{p_K(-y) : y\in K\}=\sup\{p_{_K\cap K}(y) : y\in K\}, %
 $$%
 could be taken as a possible measure of eccentricity of $K$ with respect to 0.  By analogy by the normed case,  the asymmetric norm $p_K$ is called $K$-uniformly Fr\'{e}chet differentiable if \eqref{def.G-diff} holds uniformly with respect to $h\in K$ and $x\in \partial K.$     One shows (\cite[Th. 4.2]{zanco-zuchi93}) that the following are equivalent:

$\bullet$\;  the rooted body $(K,z)$ is uniformly smooth;

 $\bullet$\;  the Minkowski functional $p_{(K,z)}$ is $K$-uniformly Fr\'{e}chet differentiable on $\partial K$;

$\bullet$\;    $\lim_{\tau\searrow 0}\big(\rho_{(K,z)}(\tau)/\tau\big) =0 $ for every $z\in \inter K.$

 The definition of the  modulus of uniform rotundity is more involved and needs to consider the existence   of diametral points in $K.$  The {\it Minkowski diameter} ($M$-diameter) of a rooted body $(K,z)$  is defined by %
 \bequ\label{def.M-diam} %
 \diam_M(K,z)=\sup\{p_{K-z}(x-y) : x,y\in K\}.
 \eequ

A pair of points $x,y\in K$ is called $M$-{\it diametral} if $\diam_M(K,z)=\max\{p_{(K,z)}(x-y),p_{(K,z)}(y-x)\}.$
The problem of the existence of diametral points of convex bodies with respect to the norm was studied by Garkavi \cite{gark81} in connection with some minimax and maximin problems.  The center   of the largest ball (whose radius is denoted by $r_K$) contained in a convex body $K$ is called a $H$-center for $K,$ while the center  of the smallest ball containing $K$ is called a Chebyshev center of $K.$ Both centers could not exist, and if exist they need not be unique. If $x_0$ is a $H$-center for $K$ then the set $A_K(x_0)=\{y\in K : \|x^0-y\|= r_K\}$ is called the critical set of the $H$-center $x_0.$ Garkavi, {\it loc. cit.}, proved that a Banach space $X$ is reflexive if and only if every convex body in $X$ has a $H$-center. Also, the space $X$ is finite dimensional if and only if for any convex body $K$ in $X$ and any $H$-center $x_0$ of $K$ the set $A_K(x_0)$  of critical points is nonempty.

  By a compactness argument it follows that if $X$ is finite dimensional, then every rooted body admits $M$-diametral points.
In the case of $M$-diametral points the authors show in \cite{zanco-zuchi93} that, for every $p\in [1;\infty),$ the space $\ell_p$  contains
a rooted body $(K,0)$ whose $M$-diameter is not attained. The general problem of the validity of Garkavi's results for $M$-diameters  remains open, with the conjecture that the existence of $M$-diametral points for every rooted body is equivalent to the finite dimensionality of $X.$
For a rooted body $(K,z)$ in a normed space let %
 \bequ\label{def.Delta} %
\Delta_{(K,z)}(\epsi)=\inf\{\max_{t\in [0;1]}(1-p_{K-z}(tx+(1-t)y)) : x,y\in K-z,\, p_{K-z}(x-y)\geq \epsi\}.
\eequ

If $(K,z)$ has diametral points, then  $\Delta_{(K,z)}$ is defined on $[0;\diam_M(K,z)]$ with values in $ [0;1]$ and, if $(K,z)$ does not has diametral points, then  $\Delta_{(K,z)}$ is defined on $[0;\diam_M(K,z))$ with values in $ [0;1).$

Based on this notion, one defines the modulus of uniform rotundity $\gamma_{(K,z)}$ of a rooted body $(K,z)$ by the conditions:%

$%
\bullet \; \gamma_{(K,z)}(\epsi)= \Delta_{(K,z)}(\epsi), %
$ \hspace*{5cm}%
for $0\leq \epsi<\diam_M(K,z)$, %

$
\bullet \; \gamma_{(K,z)}(\diam_M(K,z))=  \Delta_{(K,z)}(\diam_M(K,z)), $\qquad
if $\dim X=2$,\\ and  %

$%
\bullet \gamma_{(K,z)}(\diam_M(K,z))=$

\quad $=\inf\{\gamma_{(Y\cap(K-z),0)}(\diam_M(Y\cap(K-z),0)) : Y $ is a 2-dimensional subspace of \, $X\}, $\\%
in general.

The last formula is justified by the fact that a similar result holds for both  moduli of uniform rotundity $\delta_X$ and $\gamma_X.$

One shows that the function $\gamma_{(K,z)}$ is continuous on some interval $[0;\beta),$ where the numeber $\beta>0$ depends on the geometric properties of the body $K,$ expressed in terms of the so called directional $M$-diameters of $K$. Also the convex body $K$ is rotund if and only if $\inf\{\gamma_{(K,z)}(\diam_M(K,z)) : z\in \inter K\} =1\}, $ and is uniformly rotund if and only if $\epsi^0_{(K,z)}=0$ for every $z\in \inter K,$  where $\epsi^0_{(K,z)}=\sup\{\epsi\in [0;\diam_M(K,z)) : \gamma_{(K,z)}(\epsi)=0\}$ is the characteristic of convexity corresponding to the asymmetric modulus $\gamma_{(K,z)}.$

If a Banach space $X$ contains a rooted body $(K,0)$ with $\epsi^0_{(K,0)}<2,$ then $X$ is supperreflexive.  Recall that a Banach space $X$ is called superreflexive if it admits an equivalent uniformly rotund   renorming (it is known that any uniformly rotund Banach space is reflexive).

If $0\in \inter K,$ then one defines  the polar set of $K$ by $K^\pi=\{x^*\in X^* : \forall y\in K,\; x^*(y)\leq 1\}.$ If $B_X$ is the unit ball of a normed space $(X,\norm),$ then $B^\pi_X$ is the unit ball of the dual space. Based on Corollary  \ref{c.HB-thm1} from the next section, it follows that this relation  is also true in the asymmetric case: $B_{X^\flat}=B_p^\pi$.  The well-known duality relation between uniform smoothness (US)  and uniform rotundity (UR) holds in this case too:
\vspace*{2mm}

$\bullet$ \; the convex body $K$ is  UR \; $\iff$\; the polar set $K^\pi$ is US.

\begin{remark}\label{rem.asym-mod} %
{\rm The results presented above stand in a normed space. It would be of interest to study these properties in an asymmetric normed space, and to see their significance for the properties of the corresponding asymmetric normed space. For instance, is there any connection between the asymmetric uniform rotundity and the reflexivity of the asymmetric normed space (as defined at the end of the next section), like in the case of Banach spaces?} %

{\rm The approximation properties of subsets of a Banach space heavily depends on the geometric properties of
the underlying space, see, for instance, the survey \cite{cobz05b}. It would be interesting to see to what extent can these properties  be extended to asymmetric normed spaces?}   %
\end{remark} %

\subsection{  Asymmetric topologies on normed lattices}

Alegre, Ferrer and Gregori,  \cite{fer-greg-aleg93} and \cite{aleg-fer-greg98,aleg-fer-greg99b}, introduced an asymmetric norm on a normed lattice and studied the properties of the induced quasi-uniformity and topology, in connection with the usual properties of normed lattices.

For a normed  lattice $(X,\norm,\leq)$ define $p:X\to [0;\infty)$ by $p(x)=\|x^+\|, $ where $x^+=x\vee 0.$ Then $p$ is an asymmetric norm on $X$ generating a quasi-uniformity $\rondu_p$ and a topology $\tau_p=\tau(\rondu_p).$

A normed lattice $(X,\norm,\leq)$ is called an $L$-space, $M$-space or an $E$-space provided %
\begin{align*}
 {\rm (L)}&\quad \|x+y\|=\|x\|+\|y\|,\\
 {\rm (M)}&\quad \|x\vee y\|=\|x\|\vee\|y\|,\\
 {\rm (E)}&\quad \|x+y\|^2=\|x\|^2+\|y\|^2,
  \end{align*}
for every positive $x,y\in X.$

To the asymmetric norm $p$ one can associate the norms, defined for $x\in X$ by the equalities:
\begin{align*}
     p^s_L(x)=&p^s_L(x)=p(x)+ p(-x),\\
      p^s_M(x)=&p^s(x)=p(x)\vee p(-x),\\
      p^s_E(x)=&\big(p(x)^2+ p(-x)^2\big)^{1/2}.
\end{align*}

A subset $Z$ of a normed lattice $X$ is called {\it increasing} if $z\leq y$ implies $y\in Z,$   for every $z\in Z$ and $y\in X.$ It is called {\it decreasing} if
$z\geq y$ implies $y\in Z,$   for every $z\in Z$ and $y\in X.$

 The following proposition contains some properties of this asymmetric norm and the corresponding topology and quasi-uniformity.%
 \begin{prop}\label{p.B-latt1} %
 Let $(X,\norm,\leq)$ be a normed lattice and $p(x)=\|x\vee 0\|$ the asymmetric norm on $X.$ %
 \begin{itemize} %
 \item[{\rm 1.}]
The norms  $p^s_M,p^s_L$ and $p^s_E$ are mutually equivalent norms on $X$ which are also equivalent to the original norm.  Further, if $X$ is an $M$-space, an $L$-space, or an $E$-space, then $p^s_M,p^s_L,$ respectively  $p^s_E$ agree with the original norm $\norm$.
 \item[{\rm 2.}]
 The quasi-uniformity $\rondu_p$ {\it determines} the normed lattice structure in the sense that $\tau_{\norm}=\tau(\rondu_p^s)$ and $\graph (\leq)=\bigcap\rondu_p,$ where $\graph (\leq) =\{(x,y)\in X\times X : x\leq y\}. $
  \item[{\rm 3.}]
A linear functional $\vphi:X\to \Real$ is $(p,u)$-continuous if and only if it is positive, that is $p(x)\geq 0$ whenever $x\geq 0.$  \\ %
There  are no continuous linear functionals from $(X,p)$ to $(\Real,|\cdot|).$
 \item[{\rm 4.}]
 A subset $Y$ of $X$ is $p$-open ($\bar p$-open) if and only if it is $\norm$-open and decreasing (resp. increasing).
\end{itemize} %
 \end{prop}

 Now we shall present some Baire properties of the asymmetric topology. As it is remarked in \cite[Proposition 1]{aleg-fer-greg98} the asymmetric topology $\tau_p$ of a normed lattice is never Baire. For this reason the authors defined in {\it loc cit} another property: a normed lattice is called {\it quasi-Baire} if the intersection of any sequence of monotonic (all of the same kind) $\norm$-dense sets is $\norm$-dense. By a monotonic set one understands a set that is increasing or decreasing. A bitopological space $(T,\tau,\sigma)$ is called {\it pairwise Baire} provided the intersection of any sequence of $\tau$-open $\sigma$-dense sets is $\sigma$-dense, and the intersection of any sequence of $\sigma$-open $\tau$-dense sets is $\tau$-dense (see \cite{aleg-fer-greg99a}).

 Let $(X,\norm,\leq)$ be a normed lattice and $X^*$ the dual space of $(X,\norm).$ A subset $F\subset X^*$ is called {\it order determining} if %
 $$%
 x\leq y \;\iff\; \forall \vphi\in F,\; \vphi(x)\leq \vphi (y).%
 $$%

 The following proposition puts in evidence the relevance of this notion for the quasi-Baire  property. %
 \begin{prop}\label{p.B-latt2} %
 Let $(X,\norm,\leq)$ be a normed lattice and $F$ an order determining subset of $X^*.$  If %
 $$%
 \sup_{x\in X}\inf_{\vphi\in F}\vphi(x) > 0, %
 $$%
 then $X$ is a quasi-Baire space. %
 \end{prop}

 In fact, the proof given in \cite{aleg-fer-greg98} shows that $X$ does not contain decreasing proper dense subsets, so it is quasi-Baire in a trivial way.

 \begin{theo}[\cite{aleg-fer-greg98}] %
 A normed lattice $(X,\norm,\leq)$ is quasi-Baire if and only if the associated bitopological space $(X,\tau_p,\tau_{\bar p})$ is pairwise Baire. %
   \end{theo}

\section{ Linear functionals on an asymmetric normed space} \label{Lin-fcs}

\subsection{Some properties of continuous linear functionals}

As in the case of normed spaces we shall be particularly interested in the duals of asymmetric normed spaces.
Let $(X,p)$ be  a space with asymmetric norm, $\bar p$ the conjugate norm and
$p^s$ the (symmetric) norm corresponding to $p$. Recall that we denote by  $X_p^\flat$ and $ X^\flat_{\bar
p}\,$   the cones of $p$-continuous, respectively $\bar p$-continuous linear functionals on
$X$.

When it is no danger of confusion the space $X^\flat_p$ will be denoted simply by
$X^\flat$ and we shall call it {\it the asymmetric dual} of  the space $(X,p).$

The functionals %
\bequ\label{def1.norm-fct}%
\|\vphi|_p = \sup \vphi(B_p) \quad\mbox{and}\quad \|\psi|_{\bar p} = \sup \psi (B_{\bar
p}) \eequ%
defined for $\vphi \in X^\flat_p$ and $\psi \in X^\flat_{\bar p}\,$, are asymmetric
norms on $X^\flat_p,\,$ and $X^\flat_{\bar p}$, respectively. Let also $X^* =
(X,p^s)^*$ be the dual space of the normed space $(X,p^s)\,$  and, for $x^*\in X^*$, let %
\bequ\label{def2.norm-fct}%
\|x^*\| = \sup x^*(B_{p^s}) = \sup\{x^*(x) : x\in X,\; p^s(x) \leq 1\}.\eequ%
Then $\norm$ is a norm on the space $X^*$ and the space $X^*$ is complete with respect
to this norm, i.e.,  a Banach space.

The following  proposition emphasizes the relations holding between these spaces.%
\begin{prop}\label{p.lin-func1}%
Let $(X,p)$ be a space with asymmetric norm.  %
\begin{itemize} %
\item[{\rm 1.}] The norm $\|\vphi|_p$ of a functional $\vphi\in X^\flat_p$ is the smallest semi-Lipschitz constant for $\vphi$ and it can be also calculated by the formula %
    \bequs %
    \|\vphi|_p=\sup\{\vphi(x)/p(x) : x\in X,\, p(x)>0\}. %
    \eequs %
\item[{\rm 2.}]  The cones $X^\flat_p$ and $X^\flat_{\bar p}$ are contained in $X^*$ and %
$$%
\|\vphi\|\leq\|\vphi|_p  \quad \mbox{and} \quad \|\psi\|\leq\|\psi|_{\bar p}\,,$$%
for  every $\vphi \in X^\flat_p$ and every  $\psi\in X^\flat_{\bar p}$.
\item[{\rm 3.}] The linear spans of $X^\flat_p$ and $X^\flat_{\bar p}$ are given by %
$$%
\hull(X_p^\flat)=X_p^\flat-X_p^\flat \quad\mbox{and}\quad  \hull(X_{\bar p}^\flat)=X_{\bar p}^\flat-X_{\bar p}^\flat. %
$$%

\item[{\rm 4.}] We have $\|\vphi|_p = \|-\vphi|_{\bar p},$ so that %
  $$\vphi\in X^\flat_p \;\mbox{ and} \; \|\vphi|_p \leq r \; \iff \; -\vphi \in
X^\flat_{\bar p} \; \mbox{and} \; \|-\vphi|_{\bar p}\leq r.$$%
\end{itemize} %
\end{prop}%
\begin{proof}%
The assertions from 1 hold for arbitrary continuous linear mappings, see Proposition \ref{p.norm-op1}.

2. Let  $\vphi \in X^\flat_p$. The inequalities %
$$%
\vphi(x) \leq \|\vphi|_pp(x) \leq \|\vphi|_pp^s(x),\; x\in X,$$ %
  imply $\vphi \in X^*$ and
$\|\vphi\|\leq \|\vphi|_p.$  The situation
for the conjugate norm $\bar p$ is similar.

3. In fact the property is true for every convex cone $K$ in a vector space $X,$ that is $\hull (K)=K-K.$

Obviously that  $\,K-K\subset \hull(K). $  An arbitrary element $x\in \hull(K)$ can be written in the form $x=\sum_{i=1}^m\alpha_i y_i- \sum_{i=m+1}^{m+n}\alpha_iy_i,$
with $y_i\in K$ and $\alpha_i\geq 0,\, i=1,\dots,m+n.\,$ As a convex cone, $K$ is closed with respect to the linear combinations with positive coefficients, so that $x\in K-K.$

4. Let now $\vphi :X\to \mathbb R$ be a linear functional. The assertion  will be a
consequence of the following equalities:%
$$%
\|-\vphi|_{\bar p} = \sup\{-\vphi(x) : \bar p(x) \leq 1\} = \sup\{\vphi(-x) :  p(-x)
\leq 1\} = \|\vphi|_p.$$%
\end{proof}%

In the following proposition we collect some simple properties of the norm $\|\,|_p$
we shall need in the proofs of separation theorems. %
\begin{prop} \label{p.lin-fct2} %
If $\vphi$ is a continuous linear functional on a space with asymmetric norm $(X,p),\;
p\neq 0, $ then the following assertions hold.%
\begin{itemize}%
\item[{\rm 1.}] We have %
$$%
\begin{aligned}%
\|\vphi|_p = & \sup \{\vphi (x) : x\in X\; p(x) < 1\} \\
         = &  \sup \{\vphi (x) : x\in X\; p(x) = 1\}.
\end{aligned} %
$$%

\item[{\rm 2.}]   If $\vphi \neq 0$, then $\|\vphi |_p > 0.$ %
Also, if $\vphi \neq 0$ and $\vphi(x_0) = \|\vphi|_p$ for some $x_0\in B_p,$ then
$p(x_0) = 1$.

\item[{\rm 3.}] If $\vphi$ is both $p$-  and $ \bar p$-continuous (i.e, $\vphi\in X^\flat_{p}\cap  X^\flat_{\bar p}$), then %
 $$%
 \vphi(rB'_p) = (-r\|\vphi|_{\bar p};r\|\vphi|_p) %
 \quad\mbox{and}\quad %
 \vphi(rB'_{\bar p}) = (-r\|\vphi|_p;r\|\vphi|_{\bar p}) %
 $$%
 where $B'_p = \{x\in X : p(x) < 1\},\; B'_{\bar p} = \{x\in X : \bar p(x) <
 1\}\,$ and $r > 0.$

\item[{\rm 4.}]If $\vphi$ is $p$-continuous but not $\bar p$-continuous (i.e, $\vphi\in X^\flat_{p}\setminus  X^\flat_{\bar p}$), then %
 $$%
 \vphi(rB'_p) = (-\infty, r\|\vphi|_p).$$%
\end{itemize} %
\end{prop} %
\begin{proof} %
1. If $c:=\sup\{\vphi(x) : p(x)<1\}<\|\vphi|, $ then there exists $x_0\in X$ with $p(x_0)=1,$ such that
$c<\vphi(x_0)\leq\|\vphi|.$ In its turn, the first of these inequalities implies the existence of $0<\alpha<1$ such that $\vphi(\alpha x_0)=\alpha\vphi(x_0)>c. $ But then $p(\alpha x_0)=\alpha p(x_0)<1, $ in contradiction to the definition of $c.$

To prove the second equality, suppose again that $b:=\sup\{\vphi(x) : p(x)=1\}<\|\vphi|,$  and let $x_0\in X$ with $p(x_0)<1$ such that $\vphi(x_0)>b.$ Taking $x_1:=x_0/p(x_0), $ it follows $p(x_1)=1$ and %
$$%
\vphi(x_1)=\vphi(x_0)/p(x_0)>\vphi(x_0)>b, $$%
in contradiction to the definition of $b.$

2. If $\vphi\in X^\flat_p\setminus\{0\},$ then there exists $x_0\in X$ such that $\vphi(x_0)>0.$ By the semi-Lipschitz   condition  $0<\vphi(x_0)\leq\|\vphi|_p\,p(x_0),$ implying $\|\vphi|_p>0$ and $p(x_0)>0.$

If there exists $x_0\in X$ with $p(x_0)<1$ such that $\vphi(x_0)=\|\vphi|_p,$ then $x_1=[p(x_0)]^{-1}x_0$ satisfies $p(x_1) =1$ and $\vphi(x_1)=\vphi(x_0)/p(x_0)>|\vphi|_p,$ a contradiction.

3.     It is obvious that it is sufficient to prove the equality for $r=1.$
By the first assertion of the proposition, $\sup\vphi(B_p')=\|\vphi|_p$ and %
\begin{align*}%
\inf\vphi(B'_p)=&-\sup\{-\vphi(x) : p(x)<1\}=-\sup\{\vphi(-x) : p(x)<1\}\\=&-\sup\{\vphi(x') : \bar p(x')<1\}=-\|\vphi|_{\bar p}.%
\end{align*}%

By the assertion 2, $-\|\vphi|_{\bar p}<\vphi(x) < \|\vphi|_{p} $ for every $x\in B'_p.$
The convexity of $B_p'$ and the linearity of $\vphi$ imply that $\vphi(B_p')$ is convex, that is an interval in $\Real.$ Consequently %
 $$%
 \vphi(B_p')=(\inf \vphi(B_p');\sup \vphi(B_p'))=(-\|\vphi|_{\bar p};\|\vphi|_{p}). %
 $$%

4. If $\vphi\in X_p^\flat\setminus X^\flat_{\bar p},$   then %
$$%
\infty =\sup\{\vphi(x) : p(-x)<1\}=\sup\{-\vphi(x) : p(x)<1\}=-\inf \{\vphi(x) : p(x)<1\}.$$%

Reasoning as above, it follows $\vphi(B'_p)=(-\infty;\|\vphi|_p). $%
\end{proof}

\subsection{ Hahn-Banach type theorems}

Other fundamental principle of functional analysis is the Hahn-Banach extension theorem.
Similar results hold in the asymmetric case, based on the classical Hahn-Banach extension theorem for linear functionals dominated by sublinear functionals.

Let $X$ be a real vector space. Recall that a functional $p:X\to \Real$ is called {\it sublinear} if  %
$$%
{\rm (i)}\; p(\lambda x)=\lambda p(x)\quad \mbox{and}\quad {\rm (ii)}\quad p(x+y)\leq p(x)+p(y),   %
$$%
for all $x,y\in X$ and $\lambda \geq 0.$ Notice that, as defined, a sublinear functional need not to be positive. A sublinear functional is called a {\it seminorm} if instead of (i) it satisfies %
$$%
p(\lambda x)=|\lambda|p(x), %
$$%
for all $x\in X$ and $\lambda \in \Real.$ A seminorm is necessarily positive, that is $p(x)\geq 0$ for all $x\in X.$  A seminorm is called a norm if %
$$%
{\rm (iii)}\quad p(x)=0\;\iff\; x=0.%
$$%

A function $f:X\to \Real$ is said to be  {\it dominated} by a function $g:X\to\Real$ if %
$$%
f(x)\leq g(x), %
$$%
for all $x\in X.$

\begin{theo}[Hahn-Banach Extension Theorem]\label{t.HB-thm0} %
Let $X$ be a real vector space and $p:X\to \Real$ a sublinear functional. If $Y$ is a subspace of $X$ and $f:Y\to \Real $ is a linear functional dominated by $p$ on $Y$ then there exists a linear functional $F:X\to \Real$ dominated by $p$ on $X$ such that $F|_Y=f.$ %
\end{theo}

We present now the extension results in the asymmetric case. %
\begin{theo}\label{t.HB-thm1}%
Let $(X,p)$ be a space with asymmetric norm.
\begin{itemize}

\item[{\rm 1.}] If $Y$ is  a subspace of $X$ and $\vphi_0:Y\to \mathbb R$ is  a bounded linear
functional, then there exists a continuous linear functional $\vphi :X\to \mathbb R$ such that%
$$%
\vphi|_Y = \vphi \quad \mbox{and}\quad \|\vphi|_p = \|\vphi_0|_p.$$%
\item[{\rm 2.}] If $x_0$ is a point in $X$ with $p(x_0) > 0$, then there exists a continuous
linear functional $\vphi : X\to \mathbb R$ such that %
$$%
\|\vphi|_p = 1\quad \mbox{and}\quad \vphi(x_0) = p(x_0).$$%
\end{itemize}
\end{theo} %

We agree to call a functional $\vphi$ satisfying the conclusions of the first point of the theorem above a
{\it norm preserving extension} of $\vphi_0.$
Proofs of these results, all relying on the Hahn-Banach extension
theorem, can be found in \cite{cob-mus04} (see also
\cite{aleg-fer-greg99b,borod01,rafi-rom-per03a}). By studying quasi-uniformities on real vector spaces,
Alegre, Ferrer and Gregori \cite{aleg-fer-greg97} were able to prove a Hahn-Banach type extension theorem
for pseudo-topological vector spaces.

From the second point of the theorem one obtains as corollary a well known and useful results in normed spaces.
\begin{corol}\label{c.HB-thm1}%
Let $(X,p)$ be an asymmetric normed space, $\,x_0\in X$  and $X^\flat_p$ its dual.
If $p(x_0) > 0$ then %
$$%
p(x_0) = \sup \{ \vphi(x_0) : \vphi \in X^\flat,\; \|\vphi|_p \leq 1\}.$$%
\end{corol}%
\begin{proof}%
Denote by $s$ the supremum in the right hand side of the above
formula. Since $\vphi(x_0) \leq \|\vphi|_p\,p(x_0) \leq p(x_0)$ for
every  $\vphi\in X^\flat, \; \|\vphi|_p\leq 1,\;$  it follows $s\leq
p(x_0).$ Choosing $\vphi \in X^\flat$ as in
Theorem \ref{t.HB-thm1}.2, it follows $p(x_0) = \vphi(x_0) \leq s.$%
\end{proof}%

Now we shall present  the existence of a norm preserving extension that preserves also the
extremality of the original functional. In the case of normed spaces the result was
obtained by Singer \cite{sing62} and in the asymmetric case in \cite{cobz04}.

Let $Y$ be a convex subset of a vector space $X$. A point $e$ of
$Y$ is called an {\it extreme point} of $Y$ provided $(1-t)x + ty
= e$,  for $x,y\in Y$ and some $t,\; 0< t < 1,\,$ implies $x = y =
e.$ A nonempty convex subset $Z$ of $Y$ is called an {\it extremal
subset} of $Y$ provided $(1-t)x+ty\in Z,$ for some $x,y\in Y$ and
$0<t<1,$ implies $x,y\in Z$. Obviously that a one-point subset
$Z=\{e\}$ of $Y$ is extremal if and only if $e$ is an extreme
point of $Y$. Also, if $Z$ is an extremal subset of $Y$ and $W$ is
an extremal subset of $Z$, then $W$ is an extremal subset of $Y$.
In particular, an extreme point of an extremal subset of $Y$ is an
extreme point of $Y$.

In the following theorem, the symbols $B_{Y^\flat_p}$ and $B_{X^\flat_p}$   stand for the closed
unit balls of the dual spaces  $Y^\flat_p$ and $X^\flat_p, \, B_{Y^\flat_p}=\{\vphi\in Y^\flat_p :  \|\vphi|_p\leq 1\}$  and
$B_{X^\flat_p}=\{\psi\in X^\flat_p : \|\psi|_p\leq 1\}.$

\begin{theo}\label{t.HB-extreme}%
Let $(X,p)$ be a space with asymmetric norm and $Y$  a
subspace of $X$.

If $\vphi_0$ is an extreme point of the closed  ball
$\|\vphi_0|\cdot B_{Y^\flat_p}\,$ then there exists a norm
preserving extension $\vphi$ of $\vphi_0$ which is an extreme
point of the
ball $\|\vphi_0|\cdot B_{X^\flat_p}$ of $X^\flat_p$. %
\end{theo}%

\subsection{ The bidual  space, reflexivity and Goldstine theorem}

The bidual space was introduced in \cite{rafi-rom-per03a}, including the definition of the canonical embedding of an asymmetric normed space in its bidual and the definition of a reflexive asymmetric normed space. Further properties were proved in \cite{rafi-rom-per09}.

Let $(X,p)$ be an asymmetric normed space. Denote by $X^\flat=X^\flat_p$ the cone of all continuous linear functionals $f:(X,p)\to(\Real,u)$ and by $X^*$ the dual space of the associated normed space $(X,p^s).$ On $X^*$  one considers the extended asymmetric norm %
$$%
\|f|^*=\sup f(B_p),
$$%
and the associated extended norm $\|f\|^*=\max\{\|f|^*,\|-f|^*\}. $ The restriction of $\|\cdot|^*$ to $X^\flat$ is denoted by $\|\cdot|.$
As we have seen, $\hull X^\flat =X^\flat-X^\flat\subset X^*$   (Proposition \ref{p.lin-func1})
 and a linear functional $f\in X^*$ belongs to $X^\flat$ if and only if $\|f|^*<\infty$  (Proposition \ref{p.norm-op2}).
Let %
\bequ\label{def.bid1} %
X_e^{**} =\{\vphi:(X^*,\|\cdot\|^*)\to (\Real,|\cdot|) : \vphi \;\mbox{is linear and continuous}\}, %
\eequ %
and
\bequ\label{def.bid2} %
X_e^{*\flat} =\{\vphi:(X^*,\|\cdot|^*)\to (\Real,u) : \vphi \;\mbox{is linear and continuous}\}. %
\eequ %
(Here the subscript "e" comes from "extended").

The set $X_e^{**}$ is a linear space and $X_e^{*\flat}$ is a cone in $X_e^{**}.$  For $\vphi\in X^{*\flat}_e$ let
\bequ\label{def.norm-bid} %
\|\vphi|^{*\flat}=\sup\{\vphi(f) : f\in B^\flat_p\}. %
\eequ %

Then $\|\cdot|^{*\flat}$ is an asymmetric norm and the asymmetric  normed space  $(X_e^{*\flat},\|\cdot|^{*\flat})$  is called the {\it bidual space} of $(X,p).$

For $x\in X$ define $\Lambda(x):X^*\to \Real$ by %
\bequ\label{def.Lambda}%
\Lambda(x)(f)=f(x),\; f\in X^*. %
\eequ%

\begin{prop}\label{p.bidual1} %
Let $(X,p)$ be an asymmetric normed space and let the mapping $\Lambda$  be defined  by \eqref{def.Lambda}. Then  %
$$%
\Lambda(x)\in X_e^{*\flat} \quad\mbox{and}\quad \|\Lambda(x)|^{*\flat}=p(x),
$$%
for every $x\in X.$ Moreover, the mapping $\Lambda$ is linear, so that it defines a linear isometric embedding of $(X,p)$ into the semilinear  space $(X_e^{*\flat},\|\cdot|^{*\flat}).$ %
\end{prop} %
\begin{proof}
  Obviously that $\Lambda(x)$ is a linear functional on $X^*.$ The inequality
  $$%
  \Lambda(x)(f)=f(x)\leq\|f|^*p(x),
  $$%
  valid for every $f\in X^*,$ implies $\|\Lambda(x)|^{*\flat}\leq p(x). $ Since, by Theorem \ref{t.HB-thm1}.2, there exists $f_0\in B^\flat_p$ such that $f_0(x)=p(x),$
  it follows that $\|\Lambda(x)|^{*\flat}= p(x). $

  The linearity of the mapping $\Lambda:X\to X^{*\flat}$ is easily verified. %
  \end{proof}

Based on this proposition one can define the reflexivity of an asymmetric normed space: an asymmetric normed space $(X,p)$ is called {\it reflexive} if $\Lambda(X)=X_e^{*\flat}.$
In spite of the fact that this definition imposes a strong condition on the cone $X^{*\flat},$ namely to be a linear space, there are many interesting examples
justifying this definition, see \cite{rafi-rom-per03a}.

The  space $X^*$ induces a topology on $\hull (X^{*\flat})$ denoted by $\sigma(X^{*\flat},X^*)$ having as basis of neighborhood of 0 the sets %
\bequ\label{def.bid-top} %
V_{f_1,\dots,f_n;\epsi}=\{\vphi\in X^{*\flat} : |\vphi(f_k)|\leq \epsi,\, k=1,\dots,n\}, %
\eequ %
for $\epsi>0,\, f_1,\dots,f_n\in X^{*\flat},\, n\in \Nat.$ The neighborhoods of an arbitrary point $\vphi\in X^{*\flat}$ are obtained by translating the neighborhoods of 0.

 The following proposition says that, in essence, the spaces $X^*$ and $X^{*\flat}$ form a dual pair. %
\begin{prop}[\cite{rafi-rom-per09}]\label{p.bidual2} %
Let $(X,p)$ be an asymmetric normed space. Then the following hold. %
\begin{itemize}
  \item[{\rm 1.}]
  For each $f\in X^*$ the linear functional $e_f:\hull X^{*\flat}\to \Real$ defined by $e_f(\vphi)=\vphi(f),\, \vphi\in X^{*\flat},\,$ is continuous from $(X^{*\flat},\sigma(X^{*\flat},X^*))$ to $(\Real,|\cdot|). $
  \item[{\rm 2.}]
  If $\Psi:X^{*\flat}\to \Real$ is a linear functional, continuous from $(X^{*\flat},\sigma(X^{*\flat},X^*))$ to $(\Real,|\cdot|),$ then there exists $f\in X^*$ such that $\Psi =e_f.$
\end{itemize} %
\end{prop} %

The remarkable theorem of Goldstine of the weak density of $B_X$ in $B_{X^{**}}$ was extended in \cite{rafi-rom-per09} to asymmetric normed spaces. %
\begin{theo}[Goldstine Theorem for asymmetric normed spaces]\label{t. Goldstine} %
Let $(X,p)$ be an asymmetric normed space. The unit ball $B_p$ is $\sigma(X^{*\flat},X^{*})$-dense in the unit ball $B_{X^{*\flat}}$ of $X^{*\flat}.$ %
\end{theo} %

The paper \cite{rafi-rom-per09} contains also a characterization of reflexivity of an asymmetric normed space $(X,p)$  in terms of the completeness of the unit ball $B_p$ with respect to the weak uniformity induced by the space $X^*$ on $X.$    \\

\section{The Minkowski functional and the separation of convex sets}  \label{Sep}

\subsection{The Minkowski gauge functional - definition and properties}

A subset $Y$ of a vector space $X$ is called {\it absorbing } if %
$$%
\forall x\in X, \; \exists t> 0, \quad \mbox{such that} \quad x\in tY.$$%
If $Y$ is absorbing, then the {\it Minkowski functional} (or the {\it gauge function})
$p_Y$ of the set $Y$ is defined by  %
$$%
p_Y(x) = \inf \{ t>0 : x\in tY\}. $$%
It follows that $p_Y$ is a positive and positively homogeneous functional, and %
$$%
Y \subset \{ x\in X : p_Y(x) \leq 1\}. $$%
If Y is convex and absorbing, then $p_Y$ is a positive sublinear functional and %
$$%
\{x\in X : p_Y(x) < 1\} \subset Y \subset \{x\in X : p_Y(x) \leq 1\}.$$%

Now suppose that $(X,p) $ is a space with asymmetric norm  and look for conditions
on
the set $Y$ ensuring the $(p,u)$-continuity of $p_Y$. %
\begin{prop}\label{p.Mink}%
Let $Y$ be a convex absorbing subset of a space with asymmetric norm $(X,p)$.
\begin{itemize}
\item [{\rm 1.}]
The gauge function $p_Y$ of the set $Y$ is $(p,u)$-continuous if and only if 0 is a
$\tau_p$-interior point of $Y.$
 \item [{\rm 2.}]
 If $p_Y$ is $(p,u)$-continuous, then %
 \begin{equation}\label{eq1.Mink}%
 \tau_p\mbox{-}\inter Y = \{ x\in X : p_Y(x) < 1\}. %
 \end{equation}%
 \end{itemize}
 \end{prop}%
 \begin{proof}%
 Suppose that $0$ is a $\tau_p$-interior point of the set $Y$, and let $r>0$ be such
 that $B_p[0,r] \subset Y.$ Because $Y \subset \{x\in X : p_Y(x) \leq 1\}, $ it follows
 that %
 $$%
 p(x) \leq r \; \Rightarrow \; p_Y(x) \leq 1 $$%
 for every $x\in X.$ As in the proof of Proposition \ref{p.cont-lin1}, one shows that
 this condition implies %
 $$%
 \forall x\in X, \quad p_Y(x) \leq \frac{1}{r} p(x),$$%
 which, by the same  proposition, is equivalent to the $(p,u)$-continuity of the
 sublinear functional $p_Y$.

 Conversely, suppose that $p_Y$ is $(p,u)$-continuous and let $L > 0$ be such that%
 $$%
 \forall x\in X, \quad p_Y(x) \leq Lp(x).$$%

 But then $p(x) < 1/L$ implies $p_Y(x) < 1, $ showing that %
 $$%
 B_p(0,1/L) \subset \{x\in X : p_Y(x) < 1\} \subset Y,$$%
 i.e. 0 is a $\tau_p$-interior point of $Y$.

  Suppose now that $p_Y$ is $(p,u)$-continuous and show that%
  $$%
  \tau_p\mbox{-}\inter Y = \{ x\in X : p_Y(x) < 1\}.$$%

  Since $p_Y$ is $(p,u)$-continuous and $(-\infty;1)$ is a $\tau_u$-open subset
   of $\mathbb   R_u$, it follows that%
      $$%
\{x\in X : p_Y(x) < 1\} = p_Y^{-1}((-\infty, 1))$$%
is a $\tau_p$-open subset of $Y$, so that %
$$%
\{x\in X : p_Y(x) < 1\} \subset \tau_p\mbox{-}\inter Y.$$%

Let now $x_0\in \tau_p\mbox{-}\inter Y$ and let $r> 0$ be such that $B_p(x_0,r)\subset Y$. For
$\alpha > 0$ put $x=(1+\alpha)x_0.$  If  $p(x_0) > 0$, then $p(x-x_0)=\alpha p(x_0) <
r,\,$ whenever
$0<\alpha < r/p(x_0)\;$   so that $(1+\alpha )x_0 \in Y$ and %
$$%
p_Y(x) \leq 1/(1+\alpha) < 1.$$%

If $p(x_0) = 0$ then, since  for some $L>0,\; p_Y(x)\leq L p(x)$ for all $x\in X $, it
follows $p_Y(x_0) = 0 < 1.$ %
\end{proof} %

\subsection{ The separation of convex sets}

The separation theorems for convex sets are very efficient tools in the treatment of
optimization problems in Banach or locally convex spaces.  The so far developed
machinery allows us to prove  the asymmetric analogs of the Eidelheit and Tukey
separation theorems (see, e.g., \cite[pages 179-180]{Meg98}). %
\begin{theo}\label{t.sep-cv1}%
Let $(X,p)$ be a space with asymmetric norm, $Y_1,Y_2$ two disjoint nonempty convex
subsets of $X$ with $Y_1\; \tau_p$-open.

Then there exists a bounded linear functional $\vphi : X\to \mathbb R$ such that %
\begin{equation}\label{eq4.3}%
\forall y_1\in Y_1 \; \forall y_2\in Y_2, \quad \vphi (y_1) < \vphi (y_2). %
\end{equation} %
\end{theo}%
\begin{proof}%
For $y^0_i\in Y_i,\; i=1,2,$ put $ x_0:=y_2^0-y_1^0.$ Since the set $Y_1$ is
$\tau_p$-open it follows that the set %
$$%
Y:= x_0+Y_1-Y_2 = \bigcup \{x_0-y_2+Y_1 : y_2\in Y_2\} $$%
is $\tau_p$-open too and convex and $0 \in Y$ and $x_0\notin Y.$

The Minkowski functional  $p_Y$ of the set $Y$ is sublinear and continuous and
$p_Y(x_0) \geq 1$ because $x_0\notin Y$.

Let $Z:= \mathbb Rx_0$ and let $\vphi_0: Z\to \mathbb R$ be defined by
$\vphi_0(tx_0)=t,\; t\in \mathbb R.$  It follows that $\vphi_0$ is linear, %
$$%
\vphi_0(tx_0) = t\leq tp_Y(x_0) = p_Y(tx_0) $$%
for $t>0$,  and %
$$%
\vphi_0(tx_0) = t \leq 0\leq p_Y(tx_0),$$%
for $t\leq 0.$

By the Hahn-Banach extension theorem there is a linear functional $\vphi :X\to \mathbb
R$ such that %
$$%
\vphi |_Z=  \vphi_0 \quad \mbox{and}\quad  \vphi (x) \leq p_Y(x) \; \mbox{for all } \;
x\in X.$$%
The set $Y$ being $\tau_p$-open with $0\in Y$ it follows $Y = \{ x\in X : p_Y(x) < 1\}$.
Taking into account that  $\vphi (x_0) =1,$ we have %
$$%
\forall y_1\in Y_1, \; \forall y_2 \in Y_2, \quad 1+\vphi(y_1) - \vphi (y_2) = \vphi
(x_0+y_1-y_2) \leq p_Y(x_0+y_1-y_2) < 1, $$%
implying %
$$%
\forall y_1\in Y_1 \; \forall y_2 \in Y_2, \quad  \vphi(y_1) < \vphi (y_2),$$%
i.e. \eqref{eq4.3} holds. \end{proof} %

\begin{remark}\label{rem4.4}{\rm  The inequality in \eqref{eq4.3} can not be reversed,
in the sense that
we can not prove the existence of a $p$-bounded linear functional $\psi$ such that%
$$%
\forall y_1\in Y_1 \; \forall y_2 \in Y_2, \quad  \psi(y_1) < \psi (y_2).$$%

An examination of the proof shows that we should work with the set $Y'  = x_0+ Y_2-Y_1.$
The fact that the set $Y_1$ is $\tau_p$-open does  imply that $-Y_1$ is $\tau_{\bar
p}$-open, so that, in this case, there exists a $\bar p$-bounded linear functional $\psi$
such that }%
$$%
\forall y_2\in Y_2 \; \forall y_1\in Y_1, \quad \psi(y_2) < \psi(y_1).$$%
\end{remark}%

The second classical separation theorem can be transposed to the asymmetric case too. %
\begin{theo}\label{t4.5}%
Let $(X,p)$ be a space with asymmetric norm and $Y_1, Y_2$ two nonempty convex
subsets of $X$, with $Y_1\; \tau_p$-compact and $Y_2 \; \tau_p$-closed.

If $Y_1\cap Y_2 = \emptyset $, then there exists $\vphi\in X^\flat_p$ such that %
$$%
\sup\vphi(Y_1) < \inf \vphi(Y_2).$$%
\end{theo}%
\begin{proof}%
 Since $Y_1\cap Y_2 = \emptyset$ and $Y_2$ is closed, it follows that for every $y\in
 Y_1$ there exists $r_y > 0$ such that %
 $$%
 (y+2r_yB'_p)\cap Y_2 = \emptyset.$$%
 By the compactness of $Y_1$, the open cover $\{y+r_yB'_p : y\in Y_1\}$ contains a
 finite subcover $\{y_k+r_kB'_p : k=1,2,...,n\}\},\;$ where $r_k = r_{y_k}$. Putting %
 $r = \min \{r_k : 1\leq k \leq n\}$ it follows %
 $$%
 (Y_1+rB'_p)\cap Y_2 = \emptyset.$$%

 The set $Z := Y_1+rB'_p $ is open convex and $Z\cap Y_2 = \emptyset,$ so that, by
 Theorem \ref{t.sep-cv1}, there exists $\vphi \in X^\flat_p$ such that %
 $$%
 \forall z\in Z\; \forall y'\in Y_2, \quad \vphi(z) < \vphi
 (y'),$$%
 that is %
 $$%
 \forall y\in Y_1\; \forall u\in B'_p\; \forall y'\in Y_2, \quad \vphi(y) + r\vphi(u) <
 \vphi(y'). $$%
 Passing to supremum with respect to $u\in B'_p$ and taking into account Proposition
 \ref{p.lin-fct2}.1, we get%
 $$%
 \forall y\in Y_1\; \forall y'\in Y_2,\quad \vphi(y) + r\|\vphi|_p \leq \vphi(y'),$$%
 which implies %
 $$%
\sup \vphi(Y_1) +  r\|\vphi|_p  \leq \inf \vphi(Y_2), $$%
so that %
$$ \sup\vphi(Y_1) < \inf\vphi(Y_2).$$%
\end{proof} %

\subsection{ Extremal points and the Krein-Milman theorem}

Using the second  separation theorem, one can  prove a
Krein-Milman type theorem for spaces with asymmetric norms, following the ideas from
the symmetric case.  The
proof  is based on the fact that the intersection of an arbitrary
family of extremal subsets of a convex set $Y$ is an extremal
subset of $Y$, provided it is nonempty. We refer to \cite{cobz04} for details.

 Denote by  $\extreme (Y)$ the set of extreme points of a set $Y$
and by $\conv (Z),\; \clco (Z)$  the convex, respectively closed
convex, hull of a subset $Z$ of a space with asymmetric norm.

\begin{theo}\label{t4.6}%
Let $(X,p)$ be an asymmetric normed space such that the topology $\tau_p$ is Hausdorff.

Then any nonempty $\tau_p$-compact convex subset of $X$ agrees with the $\tau_p$-closed
convex hull of the set of its extreme points. %
\end{theo}%
The following question remains open.

  \vspace*{2mm}
 {\bf Problem.}\; It is known that in locally convex spaces  a kind of converse of
 the Krein-Milman theorem holds: If $Y$ is convex compact and $Y =
 \clco(Z)$ for some subset $Z$ of $Y$, then $\extreme (Y) \subset \bar
 Z.\,$ Is this result true in the asymmetric case too?

 \begin{remark} {\rm Some of the results in this section (the separation of convex sets and Krein-Milman theorem) were
 extended in \cite{cobz05a} to asymmetric locally convex spaces. An asymmetric locally convex space is defined
 through a family of asymmetric seminorms, as in the usual case of locally convex spaces.
 Compactness in asymmetric locally convex spaces is studied in \cite{cobz09} extending some results proved by Garc{\'{\i}}a-Raffi a.o. in \cite{aleg-fer-raf08,rafi05}} %
 \end{remark}

\section{ Applications to best approximation} \label{B-app}

\subsection{Characterizations of nearest points in convex sets and duality }

As it is known, the natural framework for treating the problem of best approximation is that of normed spaces, see the books by Singer \cite{Singer70,Singer74}, so that it is very natural to consider the corresponding problem in asymmetric normed spaces. Some problems of best approximation with respect to an asymmetric norm, including  approximation in spaces of continuous or integrable functions,  were considered by Duffin and Karlovitz \cite{duff-karlovitz68}. Dunham \cite{dunham89} treated the problem of best approximation by elements of a finite dimensional subspace of an asymmetric normed normed space and proved  existence results, uniqueness results (guaranteed by the rotundity of the asymmetric norm), and found some conditions ensuring the continuity of the metric projection.  Pfankuche-Winkler \cite{winkler88} considered   the best approximation problem in some asymmetric normed spaces of Orlicz  type. De Blasi and Myjak \cite{debl-myj98} proved some generic existence results for  the problem of best approximation with respect to an asymmetric norm in a Banach space.  Similar problems were considered by Li and Ni \cite{li-ni02} and Ni \cite{ni03}.

As it is well known any closed convex subset of a Hilbert space is Chebyshev. A famous problem in best approximation theory is that of the convexity of Chebyshev sets: must be any Chebyshev subsets of a Hilbert space be convex? There are a lot of results in this direction as presented, for instance, in the survey paper by Balaganskii and Vlasov \cite{balag-vlasov96}, but the general problem is still unsolved. In some of the papers dealing with this problem one works with asymmetric norms as, for instance, in Alimov \cite{alimov05}.

 As in the normed case, linear functionals are useful in  characterizing the nearest points and for the  duality results in best approximation in asymmetric normed spaces. In the following we shall present some results were obtained in the papers \cite{cobz04,cob-mus04,cob-mus06}.
Let $(X,p)$ be an asymmetric normed space, $Y$ a nonempty subset of $X$ and $x\in X.$  Due to the asymmetry of the norm we have to consider two {\it distances}  from $x$ to $Y$ %
\bequ\label{def.dist}%
{\rm (i)}\;\; d_p(x,Y) = \inf \{p(y-x) : y\in Y\}\quad\mbox{and}\quad {\rm (ii)}\;\;d_p(Y,x) =  \inf \{p(x-y) : y\in Y\}. %
\eequ

Observe that $d_p(Y,x) = d_{\bar p}(x,Y),$ where $\bar p$ is the norm conjugate to
$p$.  Let also %
\bequs %
P_Y=\{y\in Y : p(y-x)=d_p(x,Y)\}\quad\mbox{and}\quad \bar{P}_Y=\{y\in Y : p(x-y)=d_p(Y,x)\}, %
\eequs %
denote the metric projections on $Y.$  An element $y$ in $P_Y(x)$ is called a $p$-{\it nearest point} to $x\,$
in $Y,$     while an element $\bar y$ in $\bar P_Y(x)$ is called a $\bar p$-{\it nearest point} to $x\,$
in $Y.$  The set  $Y$ is called $p$-{\it proximinal}  if $P_Y(x)\neq\emptyset$ for every $x\in X,\,$
$p$-{\it semi-Chebyshev} if $\#P_Y(x)\leq 1$ for every $x\in X$ (i.e., every $x\in X$ has at most one $p$-nearest point in $Y$), and $p$-{\it Chebyshev} if $\#P_Y(x)= 1$ for every $x\in X$ (i.e., every $x\in X$ has exactly one $p$-nearest point in $Y$). The corresponding notions for the conjugate norm $\bar p$ are defined similarly.

\begin{theo}\label{t.b-app1} %
Let $Y$ be  a subspace of a space with asymmetric norm $(X,p)$ and  $x_0 \in X$.
Denote by $  d = d_p(x_0,Y)$ and   and $\bar d = d_p(Y,x_0)$ suppose $d>0$ and $\bar d >
0$.

Then there exist $p$-bounded linear functionals $\vphi, \psi: X\to \mathbb R$ such that %
$$%
(i)\; \vphi|_Y =0,\quad (ii)\; \|\vphi|_p = 1, \quad\mbox{and}\quad (iii)\; \vphi(-x_0)
= d,$$%
respectively %
$$%
(j)\; \psi|_Y =0,\quad (jj)\; \|\psi|_p = 1, \quad\mbox{and}\quad (iii)\; \psi(x_0) =
\bar d.$$%
\end{theo}%

A consequence of Theorem \ref{t.b-app1} is the following characterization of
nearest points.%
\begin{theo}\label{t.b-app2}%
Let $(X,p)$ be  a space with asymmetric norm, $Y$  a subspace of $X$ and $x_0$ a
point in  $X$ such that $d = d_p(x_0,Y)>0$ and $\bar d  = d_{\bar p}(x_0,Y)>0.$
\begin{itemize} %
\item[{\rm 1.}] An element $y_0\in Y$ is a $ p$-nearest point to $x_0$ in $Y$ if and only if
there exists
a $p$-bounded linear functional $\vphi :X\to \mathbb R$ such that%
$$%
(i)\; \vphi|_Y =0, \quad (ii) \; \|\vphi|_p = 1,\quad (iii)\; \vphi(-x_0) = p(y_0-x_0).$$%
and %
\item[{\rm 2.}] An element $y_1\in Y$ is a $ \bar p$-nearest point to $x_0$ in $Y$ if and only if
there exists
a $p$-bounded linear functional $\psi :X\to \mathbb R$ such that%
$$%
(j)\; \psi|_Y =0, \quad (jj) \; \|\psi|_p = 1,\quad (jjj)\; \psi(x_0) = p(x_0-y_0).$$%
\end{itemize} %
\end{theo}%
\begin{proof}%
1. Suppose that $y_0\in Y$ is such that $p(y_0-x_0) = d = d_{ p}(x_0,Y)> 0.$ By Theorem
\ref{t.b-app1}, there exists $\vphi \in X^\flat_p,\; \|\vphi|_p = 1,\;$ such that $\vphi|_Y
= 0$ and $\vphi(-x_0) = d = p(y_0-x_0).$

Conversely, if for $y_0\in Y$ there exists $\vphi\in X^\flat$ satisfying the conditions
(i) -- (iii), then for every $y\in Y$, %
$$%
p(y-x_0) \geq \vphi (y-x_0) = \vphi(y_0-x_0) = p(y_0-x_0), $$%
implying $p(y_0-x_0) = d_{ p}(x_0,Y).$

The second assertion  can be proved in a similar way.
\end{proof}%

For a nonempty subset $Y$ of an asymmetric normed space $(X,p)$ put %
$$%
Y^\perp=Y^\perp_p=\{\vphi\in X^\flat_p : \vphi|_Y=0\}. %
$$%

Another consequence of Theorem  \ref{t.b-app1} is the following
 duality results for best approximation by
elements of convex sets in spaces with asymmetric seminorm. These extend results
obtained in the case of normed spaces by Nikolski \cite{nik63}, Garkavi
\cite{gark61,gark64},  Singer \cite{sing62} (see also Singer's book \cite[Appendix I]{Singer70}
and  \cite{Gol71}).  Some duality results in the asymmetric case were proved also by Babenko \cite{babenko84}.
The case of so called $p$-convex sets was considered in
\cite{cob-mun87} and \cite{cobz90a,cobz90b}.  %
\begin{theo}\label{t.b-app4}%
For a nonempty convex subset $Y$ of  a space $(X,p)$ with asymmetric norm and $x_0\in
X$, the following duality relations hold %
\begin{equation}\label{eq1.b-app4}%
d_p(x_0,Y) = \inf\{p(y-x_0) : y\in Y\} = \sup_{\|\vphi |_p\leq 1}\inf_{y\in Y}\vphi (y-x_0). %
\end{equation}%
and %
\begin{equation}\label{eq2.b-app4}
d_{p}(Y,x_0) = \inf\{p(x_0-y) : y\in Y\} = \sup_{\|\vphi |_p\leq 1}\inf_{y\in Y}\vphi (x_0-y). %
\end{equation}%

If $d_p(x,Y)>0$, then there exists $\vphi_0\in X^\flat_p \,,\; \|\vphi_0|_p =1,\,$ such
that $d_p(x_0,Y)=\inf\{\vphi_0(y-x_0) : y\in Y\},\,$  i.e. the supremum in the right hand
side of the relation \eqref{eq1.b-app4} is attained.

A similar result holds for the second
duality relation.
\end{theo}%

Based on this duality result one obtains the following characterization of
nearest points. %
\begin{theo}\label{t5.2}%
Let $(X,p)$ be  a space with asymmetric norm, $Y$ a nonempty subset of $X, \;x\in
X,\, $ and $y_0\in Y.$

If there exists a functional  $\vphi_0\in X^\flat_p$ such that%
$$%
(i)\; \|\vphi_0|_p =1,\quad (ii)\; \vphi_0(y_0-x) = p(y_0-x),\quad (iii)\; \vphi_0(y_0) =
\inf \vphi_0(Y),$$%
then $y_0$ is a $ p$-nearest point to $x$ in $Y$.

Similarly, if for $z_0 \in Y$, there exists a functional  $\psi_0\in X^\flat_p$ such that%
$$%
(j)\; \|\psi_0|_p =1,\quad (jj)\; \psi_0(x-z_0) = p(x-z_0),\quad (jjj)\; \psi_0(z_0) =
\sup \psi_0(Y),$$%
then $z_0$ is a $\bar  p$-nearest point to $x$ in $Y$.

Conversely, if $Y$ is convex, $d_p(x,Y)>0,\,$
an  $y_0$ is a  $p$-nearest point to $x$ in $Y$,   then there exists
a functional $\vphi_0\in X^\flat_p$ satisfying the conditions (i)--(iii) from above. %

Similarly, if  $z_0\in Y$ a $\bar p$-nearest point to $x$ in $Y$
with $p(x-z_0) > 0,$ then there exists a functional $\psi_0\in
X^\flat_p$ satisfying the conditions (j)--(jjj).
\end{theo}%

Based on Theorem \ref{t.HB-extreme} one can prove following characterization theorem in terms of the extreme points of
the unit ball of the dual space $X^\flat_p.$ In the case of a normed space $X$ the
result was obtained by Singer \cite{sing62} when $Y$ is a subspace of $X$  and by
Garkavi \cite{gark64} for convex sets.   The asymmetric case was treated in \cite{cobz04}.

 \begin{theo}\label{t5.3}%
 Let $(X,p)$ be a space with asymmetric norm, $Y$ a nonempty subset of $X,\; x\in X$
 and $y_0\in Y$.

 If for every $y\in Y$ there is a functional  $\vphi = \vphi_y$ in the unit ball
 $B^\flat_{p}$ of $X^\flat_p$ such that %
 $$%
 (i)\; \vphi(y_0-x) = p(y_0-x) \quad\mbox{and}\quad (ii)\; \vphi(y_0-y) \leq 0, $$%
 then $y_0$ is a $ p$-nearest point to $x$ in $Y$.

 Conversely, if $Y$ is convex and $y_0\in Y$ is  such
 that%
 $$%
  p(y_0 -x) = d_p(x,Y) > 0,$$%
  then for every $y\in Y$ there exists an extreme point $\vphi = \vphi_y$ of the unit
  ball $B^\flat_{p}$ of $X^\flat_p$, satisfying the conditions (i) and (ii) from above. %

Similarly, if $z_0\in Y$ is such that  for every $y\in Y$ there exists a functional
$\psi=\psi_y\in B^\flat_{p}$ such that %
$$%
(j)\; \psi(x-z_0) = p(x-z_0) \quad \mbox{and}\quad (jj)\; \psi(y-z_0) \leq 0,$$%
then $z_0$ is a $\bar p$-nearest point to $x$ in $Y$.

Conversely, if $z_0\in Y$ is such that $ p(x-z_0) = d_p(Y,x) > 0$, then for every $y\in
Y$ there exists an extreme point $\psi = \psi_y$ of the unit ball $B^\flat_{p}$ of $
X^\flat_p$ satisfying the conditions (j) and (jj).
  \end{theo}%

  \subsection{The  distance to a hyperplane}

 As it was shown in  \cite{cob-mus04}, the well known formula for the distance to a closed hyperplane in
 a normed space (the so called Arzel\'{a} formula) has an
analog in spaces with asymmetric norm. Remark that in this case we have to work
with both of the distances $d_p$ and $d_{\bar p}$ given by \eqref{def.dist}.  %
\begin{prop}\label{p.dist-hyperpl}%
Let $(X,p)$ be  a space with asymmetric norm, $\vphi \in X^\flat_p,\; \vphi\neq 0,\;
c\in \mathbb R$,  %
$$%
H = \{x\in X : \vphi (x) = c\}$$%
the hyperplane corresponding to $\vphi$ and $c$, and %
$$%
H^< = \{ x\in X : \vphi (x) < c\} \quad \mbox{and}\quad H^> = \{x\in X : \vphi (x) >
c\},$$%
the open half-spaces  determined by $H$.
\begin{itemize} %
\item[{\rm 1.}]
 We have %
\begin{equation}\label{eq1.dist-hyperpl}%
d_{\bar p}(x_0,H) = \frac{\vphi(x_0) - c}{\|\vphi|_p}%
\end{equation}%
for every $x_0\in H^>$, and%
\begin{equation}\label{eq2.dist-hyperpl}
d_{p}(x_0,H) = \frac{c - \vphi(x_0) }{\|\vphi|_p}%
\end{equation}%
for every $x_0\in H^<.$

\item[{\rm 2.}] If there exists an element $z_0\in X\;$ with $ p(z_0) = 1\,$ such that $\vphi
(z_0) = \|\vphi|_p$ then every element in $H^>$ has a $\bar p$-nearest point in $H$,
 and every element in $H^<$ has a $ p$-nearest point in $H$.

\item[{\rm 3.}]
If there is an element $x_0 \in H^>$ having a $\bar p$-nearest point in $H$,  or there
is an element $x'_0 \in H^<$ having a $ p$-nearest point in $H$, then there exists an
element $z_0\in X,\; p(z_0) =1,\,$ such that $\vphi(z_0) =\|\vphi|_p.$ It follows that, in
this case, every element in $H^>$ has a $\bar p$-nearest point in $H$, and every element
in $H^<$  has a $ p$-nearest point in $H$. %
 \end{itemize}
\end{prop}%

Remark that, according to the assertions 2 and 3 of the above proposition,
the hyperplane $H$ generated by a functional $\vphi\in X^\flat_p$
has some proximinality properties if and
only if the functional $\vphi$ attains its norm on the unit ball of $X,$
a situation similar to that in normed spaces.

\subsection{ Best approximation by elements of sets with convex complement}

Best approximation by elements of sets with convex complement was considered by Klee
\cite{klee61}, in connection with the still unsolved problem of convexity of Chebyshev
sets in Hilbert space (see the survey \cite{balag-vlasov96}). Klee conjectured that if a
Hilbert space contains a non convex Chebyshev set, then it contains a Chebyshev set
whose complement is convex and bounded. The conjecture was solved affirmatively by
Asplund \cite{aspl69} who proposed the term {\it Klee cavern} to designate a set whose
complement is convex and bounded.  This term was used by Franchetti and Singer
\cite{franc-sing80} who proved duality and characterization results for best
approximation by elements of caverns as well as some existence results. In
\cite{cob-mun87} some of these results were extended to sets with $p$-convex complement.
In the
paper \cite{cobz04} it was shown that  the duality result proved by Franchetti and  Singer, {\it loc cit.},
holds in spaces with asymmetric norm too. The proof is based on the formula for the
distance to a hyperplane, Proposition \ref{p.dist-hyperpl}.

We call a subset $Y$ of $(X,p),\;  p$-{\it bounded} if there exists $r>0$ and $x\in X$
such that $Y\subset B_p[x,r]\;$ or, equivalently, if $\;\sup p(Y) < \infty.$

The duality result is the following. %
\begin{theo}\label{t.b-app-cv-compl}     %
Let  $(X,p)$ be a space with asymmetric norm, $Z$ a $\tau_p$-open, $p$-bounded
convex subset of $X$ and $Y = X\setminus Z$.

Then for every $x\in Z$  the following duality relation holds: %
\begin{equation}\label{eqd3}%
d_p(x,Y) = \inf\{\sup \vphi(Y) -\vphi(x) : \vphi\in X^\flat_p,\; \|\vphi|_p =1\}.
\end{equation} %
\end{theo} %

\subsection{ Optimal points}

Garc\'{\i}a Raffi and S\' anchez P\' erez \cite{rafi-per08} propose a finer approach to the best approximation problem in asymmetric normed spaces.

Let $(X,p)$ be an asymmetric normed space and $p^s(x)=\max\{p(x),p(-x)\} $ the associated norm on $X.$ A norm $p^s_0$ on $X$ that is equivalent to $p^s$ is called a $p$-{\it associated norm} on $X.$
For $Y\subset X$ nonempty and $x\in X$ a point $y_0\in P_Y(x)$ is called $p^s_0$-{\it optimal distance point} provided %
\bequ\label{def.optim} %
p^s_0(y_0-x)\leq p^s_0(y-x), %
\eequ
for all $y\in P_Y(x).$  A $p^s$-optimal distance  point is called simply an optimal distance point. The set of all $p^s_0$-optimal distance  points to $x$ in $P_Y(x)$ is denoted by
$O_{Y,p^s_0}(x)$ and that of optimal distance points by $O_Y(x).$

As it is remarked in \cite{rafi-per08}, the size of the set $P_Y(x)$ of nearest points to $x$ in $Y$ depends on the set $\theta(x)$ given by
\eqref{def.theta}. %

\begin{prop}\label{p.optim1}
  Let $(X,p)$ be an asymmetric normed space, $ Y\subset X,\, x,y\in X.$  If $P_Y(x)\neq\emptyset,$ then %
  \begin{itemize}
    \item[{\rm 1.}] \quad$y\in P_Y(x) \;\Rightarrow \; \theta(y)\cap Y\subset P_Y(x).$%
    \item[{\rm 2.}] \quad $ \theta(y)\cap Y\neq \emptyset \; \Rightarrow\; p(y-x)\geq d_p(x,Y).$
    \item[{\rm 3.}] \quad $P_Y(x)=\bigcup\{\theta(y)\cap Y : y\in B_p[x,d]\},\;$ where $d=d(x,Y).$
  \end{itemize}
\end{prop}

The paper contains also the following existence results for optimal distance points. %
\begin{theo}\label{t.optim2} %
Let $(X,p)$ be an asymmetric normed space, $p^s_0$ a $p$-associated norm on $X,\, Y$ a nonempty convex subset of $X$ and $x\in X$ such that $P_Y(x)\neq \emptyset.$ %
\begin{itemize}
\item[{\rm 1.}]   If the normed space $(X,p_0^s)$ is strictly convex, then there is  at most one $p_0^s$-optimal distance point in $P_Y(x).$
\item[{\rm 2.}] In any of the following cases there is at least one   $p_0^s$-optimal distance point in $P_Y(x):$

{\rm (i)}\; the set $P_Y(x)$ is locally compact (in particular, if $Y$ is contained in a finite dimensional subspace of $X$), or

{\rm (ii)}\; $(X,p^s)$ is a reflexive Banach space and $Y$ is $p^s$-closed.
\end{itemize} %
\end{theo}%

\subsection{ Sign-sensitive approximation in spaces of continuous or integrable functions}

As we did mention asymmetric, normed spaces and the notation $\|\cdot|$ for an asymmetric norm were introduced and studied by Krein and Nudelman \cite[Ch. IX, \S5]{Krein-Nudelman} in their book on the moment problem. As an  example they considered some spaces of continuous functions with asymmetric norm given by a weight function. One considers a pair $\vphi=(\vphi+,\vphi_-)$ of continuous  strictly positive functions on an interval $[a;b]$ and one denote by $B(\vphi)$ the space of all continuous functions on  $[a;b]$ equipped with the norm %
\bequ\label{def.Krein-norm} %
\|f|=\max_{a\leq t\leq b}\big\{\frac{f^+(t)}{\vphi_+(t)}+\frac{f^-(t)}{\vphi_-(t)}\big\}, %
\eequ %
 for $f\in C[a;b],$ where $f^+(t)=\max\{f(t),0\}$ and $f^-(t)=\max\{-f(t),0\}.$ (It follows $f=f^+-f^-$ and $|f|=f^++f^-$).

 The asymmetric norm \eqref{def.Krein-norm} is topologically equivalent to the usual sup-norm on $C[a;b]$ (and coincides with it for $\vphi_+=\vphi_- \equiv 1$),  so that the dual of $B(\vphi)$ agrees with the dual of $C[a;b],$  that is with the space of all functions with bounded variation on $[a;b]$ with the total variation norm. The authors give in \cite{Krein-Nudelman} the expression of the asymmetric norm of continuous linear functionals in terms of $(\vphi_+,\vphi_-)$ and study some extremal problems in this  space.

 Later Dolzhenko and Sevastyanov  \cite{dolz-sevast98a,dolz-sevast99} (see also the survey \cite{dolz-sevast98b}) considered some best approximation problems in spaces of continuous functions on an interval $\Delta=[a;b]$ and studied the existence and uniqueness (finite dimensional Chebyshev subspaces) and gave characterizations of
the best approximation (alternance and  Kolmogorov type criteria). They consider a pair $w=(w_+,w_-)$  of nonnegative functions
and the asymmetric norm %
\bequ\label{def.Dolz-norm1} %
\|f|_w=\sup\{w_+(t)f^+(t)-w_-(t)f^-(t) : t\in \Delta\}. %
\eequ %

Asymmetric norms on spaces of integrable functions are defined analogously: %
\bequ\label{def.Dolz-norm2}   %
\|f|^p_{p,w}=\int_a^b|w_+(t)f^+(t)+w_-(t)f^-(t)|^p dt, %
\eequ %
for $1\leq p <\infty.$

The study of sign-sensitive approximation is considerably more complicated than the usual approximation (in the sup-norm or in $L^p$-norms) and requires a fine analysis, based of the  the properties of weight functions.
This analysis is done in several papers, mainly by Russian and Ukrainian  authors. Among these we mention  Dolzhenko and Sevastyanov with the papers quoted above, Sevastyanov \cite{sevast97}, Babenko \cite{babenko82,babenko83}, Kozko \cite{kozko97,kozko98a,kozko98b,kozko02}, Ramazanov \cite{ramaz95,ramaz96a,ramaz96b,ramaz98}, Simonov \cite{simonov03}, Simonov and Simonova \cite{simonov93},
and the references quoted in these papers.

\section{Spaces of semi-Lipschitz functions} \label{Semi-lip}

\subsection{Semi-Lipschitz functions - definition, the extension property, applications to best approximation in quasi-metric spaces}

Let $(X,\rho)$ be a quasi-metric  space. A function $f:X\to \Real$ is called {\it semi-Lipschitz} provided there
exists a number $L\geq 0$ such that %
\bequ\label{def.sL} %
f(x)-f(y)\leq L\,\rho(x,y), %
\eequ %
for all $x,y\in X.$  A number $L\geq 0$ for which \eqref{def.sL} holds is called a {\it semi-Lipschitz constant} for $f$ and we say that $f$ is $L$-semi-Lipschitz.
We  denote by $\slip(X)$ ($\slip_\rho(X)$ if more precision is necessary) the set of  all semi-Lipschitz functions on $X.$
The function $f$ is called $\leq_\rho$-{\it monotone} if $f(x)\leq f(y)$ whenever $\rho(x,y)=0.$ Obviously that a semi-Lipschitz function is $\leq_\rho$-monotone.
Since the topology $\tau_\rho$ is $T_1$ if and only if $\,\rho(x,y)=0 \iff x=y,\,$   it follows that any function on a $T_1$ quasi-metric  space is $\leq_\rho$-monotone.

For an arbitrary function $f:X\to \Real$ put %
\bequ\label{def.sL-norm} %
\|f|_\rho=\sup\big\{\frac{u(f(x)-f(y))}{\rho(x,y)} : x,y\in X,\,\rho(x,y) > 0\big\}. %
\eequ

\begin{prop}\label{p.sL1} %
Let $(X,\rho)$ be a quasi-metric  space and $f:X\to\Real.$ %
\begin{itemize} %
\item[{\rm 1.}]
The set $\slip(X)$ is  a cone in the linear space of all functions from $X$ to $\Real.$
\item[{\rm 2.}]
If $f$ is semi-Lipschitz, then  $\|f|_\rho$ is the smallest semi-Lipschitz constant for $f.$
\item[{\rm 3.}]
A function $f:X\to \Real $ is semi-Lipschitz if and only if $f$ is $\leq_\rho$-monotone and $\|f|_\rho<\infty.$
\end{itemize} %
\end{prop} %
\begin{proof}
1. It is clear that $f+g,\,\alpha f\in \slip (X)$ for all  $f,g\in \slip (X)$ and $\alpha \geq 0.$

2. If $\rho(x,y)>0,$ then $(f(x)-f(y))/\rho(x,y)\leq \|f|_\rho, $ that is $f(x)-f(y)\leq \|f|_\rho\,\rho(x,y). $ Since a semi-Lipschitz function is $\leq_\rho$-monotone,
 $\rho(x,y)=0,$ implies $f(x)-f(y)\leq 0 =\|f|_\rho\,\rho(x,y). $ Consequently, $\|f|_\rho$ is a semi-Lipschitz constant for $f.$

Suppose that $L$ is a semi-Lipschitz constant for $f.$ Then $(f(x)-f(y))/\rho(x,y)\leq L,$ whenever $\rho(x,y)>0,$ so that $\|f|_\rho\leq L,$ showing that $\|f|_\rho$ is the smallest semi-Lipschitz constant for $f.$

The above  reasonings show also the validity of 3. %
\end{proof}

The properties of the spaces of semi-Lipschitz functions were studied by Romaguera and Sanchis \cite{romag-sanchis00a,romag-sanchis05} and Romaguera, S\'{a}nchez-\'{A}lvarez and Sanchis \cite{romag-sanchis08}. The paper by  Must\u{a}\c{t}a \cite{must02c} is concerned with the behavior of the extreme points of the unit ball in spaces of semi-Lipschitz functions

An important result in the study of Lipschitz functions on metric spaces is the extension of Lipschitz functions, usually known as Kirszbraun's  extension theorem, see, for instance, the book \cite{We-Will}.

In the case of semi-Lipschitz functions a similar result was proved by Must\u a\c ta \cite{must02a} (see also \cite{must00b,must05}). The extension problem for semi-Lipschitz functions on quasi-metric spaces was considered  also by Matou\v{s}kov\'{a} \cite{matuska00}.  The paper \cite{rafi-rom-per01} discusses the existence  of an extension of a norm defined on a cone $K$ to a norm defined on the generated linear space $X=K-K.$

\begin{prop}\label{p.ext-sL} %
 Let $(X,\rho)$ be a quasi-metric  space, $\,Y$ a nonempty subset of $X$ and  $f:Y\to \Real$  an $L$-semi-Lipschitz function.
 \begin{itemize}
 \item[{\rm 1.}]
The functions defined for $x\in X$  by %
 \bequ\label{eq1.ext-sL} %
F(x)=\inf\{f(y)+L\rho(x,y) : y\in Y\} %
 \eequ %
 and
 \bequ\label{eq2.ext-sL} %
 G(x)=\sup\{f(y')-L\rho(y',x) : y'\in Y\} %
 \eequ %
are $L$-semi-Lipschitz extensions of $f.$
 \item[{\rm 2.}]
Any other $L$-semi-Lipschitz extension $H$ of $f$ satisfies the inequalities %
\bequ\label{eq3.ext-sL} %
G\leq H\leq F. %
\eequ %
\end{itemize} %
\end{prop} %
\begin{proof}
1. Let $y,y'\in Y$ and $x\in X.$ The inequalities $  f(y')-f(y)\leq L\,\rho(y',y)\leq L\,\rho(y',x) + L\,\rho(x,y)$ imply %
$$%
f(y')-L\,\rho(y',x) \leq f(y)+L\,\rho(x,y).%
$$ %

Passing to supremum with respect to $y'\in Y$ and to infimum with respect to $y\in Y,$ it follows that $G, F$  are well defined and $G\leq F.$

Let $x\in Y.$
  Then $f(x)\leq f(y)+L\,\rho(x,y)$ for every $y\in Y, $ implies $f(x)\leq F(x).
  $ Similarly, $f(y')-L\,\rho(y',x)\leq f(x)$ implies $G(x)\leq f(x).$ Taking $y=x$ in \eqref{eq1.ext-sL} and $y'=x$ in \eqref{eq2.ext-sL}, it follows
$  F(x)\geq f(x)$  and $G(x)\leq f(x),$ so that $G(x)=f(x)=F(x)$ for every $x\in Y.$

  To conclude, we have to show that the functions $F,G$  are semi-Lipschitz. Let $x,x'\in X.$  The inequalities %
  $$%
  F(x)\leq f(y)+L\,\rho(x,y)\leq f(y)+L\,\rho(x,x')+L\,\rho(x',y), %
  $$%
valid for all  $y\in Y,$ yield  $\, F(x)\leq F(x')+L\,\rho(x,x'), $ showing that  $F$ is semi-Lipschitz.

Similar reasonings show that $G$ is semi-Lipschitz too.

2. Let $H$ be an $L$-semi-Lipschitz extension of $f$ and $x\in X.$ Since %
$$%
H(x)\leq H(y)+L\,\rho(x,y)=f(y)+L\,\rho(x,y), %
$$%
 for  every $y\in Y,$ passing to infimum with respect to $y\in Y$ one obtains  $H(x)\leq F(x).$

 Similarly $f(y')-H(x)= H(y')-H(x)\leq L\,\rho(y',x,)$ implies %
 $$%
 f(y')-L\,\rho(y',x)\leq H(x), %
 $$%
 for every $y'\in Y.$ Passing to supremum with respect to $y'\in Y$ one obtains $G(x)\leq H(x).$ %
 \end{proof}

 The following corollary shows the existence of norm-preserving extentions of semi-Lipschitz functions. %
\begin{corol}\label{c.ext-sL2}
Let $X,Y$ and $f$ be  as in Proposition \ref{p.ext-sL} and %
$$%
\|f|_\rho=\sup\{u(f(y)-f(y'))/\rho(y,y') : y,y'\in Y,\, \rho(y,y')>0\}, %
$$%
and let $F,G$ be given by \eqref{eq1.ext-sL} and \eqref{eq2.ext-sL} for $L=\|f|_\rho.$ Then %
\begin{itemize}
  \item[{\rm 1.}]
  The functions $F$  and $G$ are semi-Lipschitz norm-preserving extensions of $f$, that is %
  $$%
  (i)\; F|_Y=G_Y=f  \quad\mbox{and}\quad (ii) \; \|F|_\rho=\|G|_\rho=\|f|_\rho.
  $$%
   \item[{\rm 2.}]
Any other semi-Lipschitz norm-preserving extension $H$ of $f$ satisfies the inequalities %
\bequs %
G\leq H\leq F. %
\eequs %
\end{itemize} %
\end{corol} %

The functional $\|\cdot|_\rho$ given by \eqref{def.sL-norm} is only an asymmetric semi-norm on $\slip (X)$ because for a constant function
$f(x)\equiv c$ with $c\neq 0, \, \|f|_\rho=\|-f|_\rho =0. $ In order  to obtain a norm we consider  the space %
$$%
\slip_{\rho,0}(X)=\slip_0(X)=\{f\in \slip (X) : f(x_0) =0\}, %
$$%
where $x_0$ is a fixed point in $X$, which is chosen to be 0 if $X$ is a vector space.

In this case if $\|f|_\rho=\|-f|_\rho =0,$ it follows $f(x)-f(y)\leq 0$ and $f(y)-f(x)\leq 0, $ so that the function $f$ is constant.
Since $f(x_0)=0$ it follows $f=0.$

For a nonempty subset $Y$ of a quasi-metric  space $(X,\rho)$ and $x\in X$ one defines the {\it distance} from $x $ to $Y$ by %
\bequ\label{def.rho-dist} %
d(x,Y)=d_\rho(x,Y)=\inf\{\rho(x,y) : y\in Y\} %
\eequ%
A point $y\in Y$ such that $\rho(x,Y)=d_\rho(x,Y)$ is called a nearest point to $x$ in $Y$ and the set of all these points is denoted by $P_Y(x).$
The operator $P_Y\to 2^Y$ is called the metric projection of $X$ onto $.Y$   %
\begin{prop}\label{p.sL-dist1} %
Let $(X,\rho)$ be a quasi-metric  space, $y\in X$ and $Y\subset X$ nonempty.
\begin{itemize}
  \item[{\rm 1.}]
The functions $\rho(\cdot,y) :X\to \Real$ and $d(\cdot,Y):X\to \Real$ are semi-Lipschitz with semi-Lipschitz constant 1.
\item[{\rm 2.}] For fixed $a\in X,$ the functions $f(x)=\rho(a,x_0)-\rho(a,x)$ and $g(x)=\rho(x,a)-\rho(x_0,a)$ belong to $\slip_{\rho,0}(X)$ and $\|f|_\rho,\,\|g|_\rho \leq 1.$
\end{itemize}
\end{prop} %
\begin{proof}
1. The inequality
  \bequ\label{eq1.sL-dist1}%
    \rho(x,y)\leq \rho(x,x')+\rho(x',y),
    \eequ
    valid   for $x,x'\in X,$     shows that the function $\rho(\cdot,y)$ is semi-Lipschitz. Since the inequality \eqref{eq1.sL-dist1} holds for all $y\in Y$ and fixed $x,x'$,
    passing to infimum with respect to $y\in Y$ one obtains $d(x,Y)\leq \rho(x,x')+d(x',Y), $ which means that the function $d(\cdot,Y)$ is semi-Lipschitz, too.

    The assertions from 2 follow from 1.
\end{proof}

 For $Y\subset X$ put %
$$%
Y^\perp=\{f\in \slip_0 (X) : f|_Y=0\} %
$$%

The following result is the semi-Lipschitz analog of the Theorem \ref{t.b-app2}. %
\begin{prop}\label{p.ext-sL3} %
Let $(X,\rho)$ be a quasi-metric  space, $Y\subset X$ nonempty,  $x_0\in X$ such that $d(x_0,Y)>0$ and $y_0\in Y.$ Then $y_0$ is a nearest point to $x_0$ in $Y$ if and only if there exists $f\in Y^\perp$ such that %
$$%
{\rm (i)}\quad \|f|_\rho =1\quad\mbox{and}\quad {\rm (ii)}\quad \rho(x_0,y_0)=f(x_0)-f(y_0). %
$$%
\end{prop}

In the case of metric spaces and Lipschitz functions similar results were obtained by Must\u a\c ta \cite{must75,must78}, who obtained also many results on the characterizing the approximation properties in a quasi-metric space in terms on the semi-Lipschitz functions defined on it. Other results on best approximation and extensions   were obtained in \cite{must02c,must04}.

Let $(X,\norm)$ be a normed space and $X^*$ its dual. For a subspace $Y$ of $X$ put %
$$%
Y^\perp=\{x^*\in X^* : x^*|_Y=0\}.
$$%
and denote by $E_Y(y^*)$ the set of all norm-preserving extensions of a continuous linear functional $y^*$ on $Y$, that is,  %
$$%
E_Y(y^*)=\{x^*\in X^* : x^*|_Y =y^* \;\mbox{and}\; \|x^*\|=\|y^*\|\}. %
$$%

Phelps \cite{phe60} proved the following remarkable result relating the approximation properties of the space $Y^\perp$ and the extension properties of the space $Y.$ \\

{\sc Theorem} (Phelps \cite{phe60})  \emph{Let $Y$ be a a closed subspace of a normed space $X.$ Then $Y^\perp$ is a proximinal subspace of $X^*$ and for every $x^*\in X^*$ the following equality holds }%
$$%
P_{Y^\perp}(x^*)=E_Y(x^*|_Y). %
$$%

\emph{Consequently, $Y^\perp$ is a Chebyshev subspace of $X^*$ if and only if every $y^*\in Y^*$ has a unique norm-preserving extension to the whole} $X.$ \\

Extensions of this results to spaces of Lipschitz functions on metric spaces and to spaces of semi-Lipschitz functions on quasi-metric spaces were given by Must\u a\c ta \cite{must00a,must00b,must02b,must02d,must03,must08} (see also the paper \cite{cobz02} containing a survey of various situations where a Phelps type result can occur).  An   iterative approximation  method to find   the global minimum of a semi-Lipschitz function is proposed
 in \cite{must09}. Romaguera and Sanchis give in \cite{romag-sanchis03} characterizations of preferences on separable  quasi-metric spaces admitting  semi-Lipschitz utility functions, with applications
 to theoretical computer science.

\subsection{Properties of the cone of semi-Lipschitz functions - linearity, completeness}

A problem discussed in \cite{romag-sanchis05} is: under what conditions is the cone $\slip_{\rho,0}(X)$ a linear space?

The following proposition contains some simple remarks  concerning the relations between $\rho$- and $\bar \rho$-semi-Lipschitz functions. %
\begin{prop}\label{p.sL-sp1} %
let $(X,\rho)$ be a quasi-metric space and $f:X\to \Real$ a function. %
\begin{itemize}
  \item[{\rm 1.}] The function $f$ belongs to $\slip_{\rho}(X) $ (to $\slip_{\rho,0}(X)$) if and only if $\,-f$ belongs to $ \slip_{\bar\rho}(X)$ (to $\slip_{\bar\rho,0}(X)$).
  \item[{\rm 2.}] The sets $\slip_{\rho}(X)\cap \slip_{\bar\rho}(X)$  and $\,\slip_{\rho,0}(X)\cap\slip_{\bar\rho,0}(X)$ are linear spaces.
\end{itemize}
\end{prop}

We mention also the following result from \cite{romag-sanchis05} %
\begin{theo}\label{t.sL-sp2} %
Let $(X,\rho)$ be a quasi-metric space. The following are equivalent %
\begin{itemize}
  \item[{\rm 1.}] $\slip_{\rho,0}(X)$ is a linear space and $\|\cdot|_\rho$ is a complete norm on it.
    \item[{\rm 2.}] $\slip_{\rho,0}(X)=\slip_{\bar\rho,0}(X)$ and $\|\cdot|_\rho=\|\cdot|_{\bar\rho}$ on $\slip_{\rho,0}(X)$.
    \item[{\rm 3.}] $(X,\rho)$ is a metric space.
\end{itemize}
  \end{theo}
 In order to treat some completeness questions for spaces of semi-Lipschitz functions,
one  defines  an extended quasi-metric  on $\slip (X)$ by the formula %
\bequ\label{def.sL-metric} %
\delta_\rho(f,g)=\sup\big\{\frac{u((f-g)(x)-(f-g)(y))}{\rho(x,y)} : x,y\in X,\,\rho(x,y) >0\big\}. %
\eequ %

The following example, given by Romaguera and Sanchis \cite{romag-sanchis00a}, shows that $\delta_\rho$ could be effectively an extended quasi-metric. %
\begin{example} %
{\rm For $x,y\in \Real$ let $\rho(x,y)=x-y$ if $x\geq y$ and  $\rho(x,y)=1$ if $x < y,$ i.e.,  $(\Real,\rho)$ is the Sorgenfrey line. The identity mapping $\id:\Real\to \Real$ is semi-Lipschitz with $\|\id|_\rho=1,$ so that $\delta_\rho(\id,0)=1,$ but $\delta_\rho(0,\id)=\infty$ because} $\sup\{((y-x)\vee 0)/\rho(x,y) : x\neq y\} =\infty.$ %
\end{example}

\begin{theo}\label{t.sL-compl1} %
Let $(X,\rho)$ be a quasi-metric space.
\begin{itemize}
  \item[{\rm 1.}] {\rm(\cite{romag-sanchis00a})}\;
The space $\slip_0(X)$ is bicomplete with respect to the extended quasi-metric $\delta_\rho.$
  \item[{\rm 2.}] {\rm(\cite{romag-sanchis05})}\;
The extended quasi-metric $\delta_\rho$ is also right $K$-complete on $\slip_0(X).$
\end{itemize}
\end{theo}

Doichinov \cite{doichin88b,doichin88c,doichin91a} defined and studied a notion of completeness for quasi-metric spaces with the aim to obtain a satisfactory theory of completion (see \cite{doichin88b} for the quasi-metric case and \cite{doichin92} for quasi-uniform spaces). A sequence $(x_n)$ in a quasi-metric space $(X,\rho)$ is called $D$-{\it Cauchy} if there exists another sequence $(y_n)$ such that $\lim_{m,n}\rho(y_m,x_n)=0.$  The quasi-metric space $(X,\rho)$ is called $D$-complete if every $D$-Cauchy sequence converges. A quasi-metric space $(X,\rho)$ is called {\it balanced} if for every sequences $(x_n),(y_n)$ such that $\lim_{m,n}\rho(y_m,x_n)=0$ and  for every $x,y\in X$ and  $r_1\,,r_2\geq 0,\, \rho(x,x_n)\leq r_1$ and $\rho(y_n,y)\leq r_2$ for all $n\in \Nat,$ implies $\rho(x,y)\leq r_1+r_2.$ The concept of balancedness, meaning a kind of symmetry of a quasi-metric space,  was also introduced by Doichinov in \cite{doichin88c},
to develop a satisfactory theory of completion. He proved that a balanced quasi-metric generates a Hausdorff and completely regular topology.

On the dual $X^\flat_p$ of an asymmetric normed space $(X,p)$ one can define an extended quasi-metric $\delta_\flat$ by %
$$%
\delta_\flat(\vphi,\psi)=\sup\{(\vphi-\psi)(x) : x\in X,\, p(x)\leq 1\}. %
$$%
\begin{theo}[\cite{romag-sanchis08}]\label{t.D-compl-dual} If $(X,p)$ is a $T_1$ normed space, then the dual space $X^\flat_p$ is balanced and $D$-complete with respect to the extended quasi-metric $\delta_\flat$. %
\end{theo}

If $(X,\rho)$ is a quasi-metric space and $(Y,q)$ is an asymmetric normed space, then a function $f$ from $X$ to $Y$ is called semi-Lipschitz if there exists $L\geq 0$ such that %
$$%
q(f(x)-f(y))\leq L \rho(x,y),
$$%
for all $x,y\in X.$  The space of all semi-Lipschitz functions from $X$ to $Y$ is denoted by $\slip_{\rho,q}(X,Y)$ and that of all functions vanishing at a fixed point $x_0\in X$ by $\slip_{p,q,0}(X,Y).$ The mapping $\delta_{\rho,q}$ defined for $f,g\in \slip_{\rho,q,0}(X,Y)$  by %
\bequ\label{def.delta-ro-q}%
\delta_{\rho,q}(f,g)=\sup\big\{\frac{q((f-g)(x)-(f-g)(y))}{\rho(x,y)} : x,y\in X,\, \rho(x,y)>0\big\}, %
\eequ
is a quasi-metric on the space $\slip_{\rho,q,0}(X,Y).$

The following completeness result was proved in \cite{romag-sanchis08}. %
\begin{theo}\label{t.sL-compl2} %
Let $(X,\rho)$ be a quasi-metric space and $(Y,q)$ an asymmetric normed space. If the asymmetric normed space $(Y,q)$ is biBanach, then the space $\slip_{\rho,q,0}(X,Y)$ is $D$-complete with respect to the metric $\delta_{\rho,q}$ defined by \eqref{def.delta-ro-q}. %
\end{theo}

 \providecommand{\bysame}{\leavevmode\hbox
to3em{\hrulefill}\thinspace}
\providecommand{\MR}{\relax\ifhmode\unskip\space\fi MR }
\providecommand{\MRhref}[2]{%
  \href{http://www.ams.org/mathscinet-getitem?mr=#1}{#2}
} \providecommand{\href}[2]{#2}

\end{document}